\documentclass[11pt]{amsart}

\usepackage{amssymb,amsfonts,amsmath}
\usepackage{scalerel}
\usepackage{makecell}
\usepackage{amsthm, stmaryrd}
\usepackage{graphicx}
\usepackage[all,arc]{xy}
\usepackage{hyperref}
\hypersetup{colorlinks,allcolors=violet}
\usepackage{enumerate}
\usepackage{bbm}
\usepackage{dsfont}
\usepackage{mathrsfs}
\usepackage{tikz-cd}
\usetikzlibrary{shapes}
\usepackage{import}
\usepackage{float}
\restylefloat{table}
\usepackage{multirow}
\usepackage{pdflscape}
\usepackage{listofitems}
\usepackage{standalone}
\usepackage[T1]{fontenc}
\usepackage[toc,page]{appendix}

\usepackage[margin=1in]{geometry}
\usepackage{mathtools}
\usetikzlibrary{decorations.pathmorphing}
\usepackage{mathrsfs}
\usepackage{cancel}
\usepackage{relsize}
\newtheorem{thm}{Theorem}[section]
\newtheorem*{thm*}{Theorem}
\newtheorem*{metathm*}{Meta Theorem}

\newtheorem*{setup*}{Setup}

\newtheorem{innercustomthm}{Theorem}
\newenvironment{customthm}[1]
  {\renewcommand\theinnercustomthm{#1}\innercustomthm}
  {\endinnercustomthm}

\newtheorem{cor}[thm]{Corollary}
\newtheorem{prop}[thm]{Proposition}
\newtheorem{lem}[thm]{Lemma}

\theoremstyle{definition}
\newtheorem{defn}[thm]{Definition}

\newtheorem{exmp}[thm]{Example}

\newtheorem{notn}[thm]{Notation}

\newtheorem{rem}[thm]{Remark}

\newtheorem*{thm1.2}{\textrm{Theorem 1.2}}

\theoremstyle{remark}
\newcommand{\Mbar}{\overline{\mathcal{M}}}
\newcommand{\M}{\mathcal{M}}

\newtheorem{sch}[thm]{Scholium}

\newcommand{\Z}{\mathbb{Z}}

\newcommand{\QQ}{\mathbb{Q}}

\renewcommand{\P}{\mathbb{P}}
\newcommand{\bbS}{\mathbb{S}}

\newcommand{\Aut}{\operatorname{Aut}}

\newcommand{\val}{\operatorname{val}}

\newcommand{\ch}{\operatorname{ch}}

\newcommand{\cat}[1]{\mathsf{#1}}

\newcommand{\tup}[1]{\underline{#1}}

\newcommand{\nocontentsline}[3]{}
\let\origcontentsline\addcontentsline
\newcommand\stoptoc{\let\addcontentsline\nocontentsline}
\newcommand\resumetoc{\let\addcontentsline\origcontentsline}

\def\C{\mathbb{C}}

\def\G{\mathbf{G}}

\def\P{\mathbb{P}}

\def\Z{\mathbb{Z}}

\def\calC{\mathcal{C}}

\def\calM{\mathcal{M}}

\def\calQ{\mathcal{Q}}

\def\calU{\mathcal{U}}
\def\calV{\mathcal{V}}
\def\calW{\mathcal{W}}
\def\calX{\mathcal{X}}
\def\calY{\mathcal{Y}}
\def\calZ{\mathcal{Z}}
\def\M{\mathcal{M}}

\DeclareMathOperator{\tcirc}{\tilde{\circ}}

\newcommand{\Cat}{\mathcal{C}}

\newcommand{\Ind}{\operatorname{Ind}}
\newcommand{\Res}{\operatorname{Res}}
\newcommand{\Spec}{\operatorname{Spec}}
\newcommand{\Exp}{\operatorname{Exp}}

\newcommand{\triv}{\operatorname{triv}}

\newcommand{\Graph}{\mathsf{Graph}}
\newcommand{\sfgamma}{\mathsf{\Gamma}}
\newcommand{\Ob}{\operatorname{Ob}}

\renewcommand{\G}{\mathbf{G}}
\newcommand{\End}{\mathrm{End}}

\newcommand{\Part}{\mathsf{Part}}
\newcommand{\Partt}{\mathsf{Part}^{[2]}}

\newcommand\cycle[2][\,]{%
  \readlist\thecycle{#2}%
  (\foreachitem\i\in\thecycle{\ifnum\icnt=1\else#1\fi\i})%
}

\newcommand{\WRing}{\Lambda^{[2]}}

\newcommand{\LLambda}{\hat{\Lambda}(\!(t)\!)}

\newcommand{\catC}{\mathsf{C}}

\newcommand{\pfunc}[3][]{{\mathbf{P}_{#1}\left[#2, #3\right]}}

\renewcommand{\-}{\text{-}}

\DeclareMathOperator*{\midboxtimes}{\scalerel*{\boxtimes}{\mathlarger{\mathlarger{\mathlarger{\oplus}}}}}

\makeatletter
\let\c@equation\c@thm
\makeatother
\numberwithin{equation}{section}

\graphicspath{ {images/} }
\title{Graph enumeration for moduli spaces of curves and maps}

\author[S. Kannan]{Siddarth Kannan}\address{Department of Mathematics, Massachusetts Institute of Technology}
\email{\url{spkannan@mit.edu}}
\author[T. Song]{Terry Dekun Song}\address{Department of Pure Mathematics and Mathematical Statistics, University of Cambridge, Cambridge, CB3 0WA}\email{\url{ds2016@cam.ac.uk}}

\begin{document}
\maketitle\thispagestyle{empty}

\begin{abstract}
We develop a calculus based on graph enumeration for $S_n$-equivariant motivic invariants of graphically stratified moduli spaces. We apply our theory to the Deligne--Mumford moduli space $\Mbar_{g, n}$ and to the space of torus-fixed stable maps $\Mbar_{g, n}(X, \beta)^{\C^\star}$ when the target $X$ admits an appropriate $\C^\star$-action, deriving new formulas in each case. A key role is played by the \textit{P\'olya--Petersen character} of a graph, which enriches P\'olya's classical cycle index polynomial. This character is valued in an algebra $\WRing$ of wreath product symmetric functions, which we study from combinatorial and representation-theoretic perspectives. We prove that this algebra may be viewed as the Grothendieck ring of the category of polynomial functors which take symmetric sequences of vector spaces to vector spaces, building on foundational work of Macdonald. This leads to a plethystic action of $\WRing$ on the ring $\Lambda$ of ordinary symmetric functions. Using this action, we derive our formulas, which ultimately involve only ordinary symmetric functions and the Grothendieck ring of mixed Hodge structures. 
\end{abstract}
\tableofcontents

\section{Introduction}
Graphs encode the combinatorics of degenerating algebraic curves, and are thus fundamental to their moduli problems. The combinatorial types of different graphs give recursive stratifications of moduli spaces in the following manner.

\begin{setup*}
    Associated to the moduli space $\Mbar$ is a collection of graphs $\{G\}$ and a collection of moduli spaces $\mathcal{M} = \bigsqcup_{n\geq 0}\mathcal{M}_{n}$, such that there is a decomposition $\Mbar = \bigsqcup_{G}\mathcal{M}_G$ into locally closed strata where $\mathcal{M}_G\cong \prod_{v\in V(G)}\mathcal{M}_{\mathrm{val}(v)}/\mathrm{Aut}(G)$, and $\mathrm{val}(v)$ denotes the valence of a vertex $v$.
\end{setup*} 

In this paper we develop a general calculus for the application of P\'olya's enumeration techniques \cite{PolyaRead} to topological invariants of this class of moduli spaces.


The prototypical example of $\Mbar$ is the Deligne--Mumford compactification\footnote{In general, $\Mbar$ has connected components indexed by discrete numerical parameters. We work with the disjoint union with suitable grading, because the recursive structure is most natural when we consider the whole collection of connected components.} $\bigsqcup_{g,n}\Mbar_{g,n},$ where $\mathcal{M}_n = \bigsqcup_{g\geq 0}\mathcal{M}_{g,n}.$ In this setting, a beautiful formula relating the $S_n$-equivariant motivic invariants of $\Mbar$ and $\M$ has been derived by Getzler and Kapranov \cite{GetzlerKapranov}; we will compare and contrast their approach with ours in more detail in  \S\ref{subsubsec:modular_operads} below. For now we emphasize that our work is independent from theirs, and offers new formulas even for the Deligne--Mumford compactification. In this article, we will apply our formalism to the following examples:
\begin{itemize}
\item $\Mbar_{g,n}$ and its combinatorial subspaces such as stable curves of compact type $\mathcal{M}_{g,n}^{\mathrm{ct}}$;
\item the moduli spaces $\Mbar_{g, n}(X, \beta)^{\C^\star}$ of torus-fixed stable maps to a projective variety with a good $\C^\star$-action (Definition \ref{def:good-action}).
\end{itemize}
Several potential future applications of the theory are discussed in \S\ref{subsec:more-applications} below. The moduli spaces of interest carry natural symmetric group actions that permute marked points. Symmetric function theory coming from the permutation actions is essential in all steps of our calculations. Our work makes many new calculations of Euler characteristics possible. Sample calculations are given in Tables  \ref{table:stable_maps_no_markings}, \ref{table:degree3-frob-chars}, and \ref{table:genus3ct} below.

\subsection{Calculations and applications}\label{sec:calcapp}
All of our results are phrased in terms of the \textit{Serre characteristic}. If $X$ is a variety over $\C$, the Serre characteristic of $X$ is defined as
\[ \cat{e}(X) := \sum_{i} (-1)^i[H^i_c(X; \QQ)] \in K_0(\cat{MHS)}, \]
where $\cat{MHS}$ is the abelian category of mixed Hodge structures over $\QQ$. The Serre characteristic specializes to the $E$-polynomial\footnote{This is sometimes called the Hodge--Deligne polynomial.} in the sense of e.g. \cite[Definition 2.1.4]{mixedhodge}, which agrees with the Hodge polynomial when $X$ is proper with at worst finite quotient singularities.

If $X$ admits an $S_n$-action, then each $H^i_c(X;\QQ)$ is a $S_n$-representation in $\cat{MHS}$, so we may define the $S_n$-equivariant Serre characteristic as follows: first let
\[ \Lambda := \QQ[p_1, p_2, \ldots] \]
be the ring of symmetric functions over $\QQ$, with grading determined by $\deg(p_i) = i$. Then the $S_n$-equivariant Serre characteristic is given by
\[\cat{e}^{S_n}(X) := \sum_{i} (-1)^i[H^i_c(X; \QQ)] \in K_0(\cat{MHS}_{S_n}) \otimes_{\Z}\QQ \cong K_0(\cat{MHS}) \otimes_{\Z} \Lambda_n, \]
where $K_0(\cat{MHS}_{S_n})$ denotes the Grothendieck group of $S_n$-equivariant mixed Hodge structures and $\Lambda_n \subset \Lambda$ is the subspace of symmetric functions which are homogeneous of degree $n$. Note that $\cat{e}^{S_n}(X)$ can easily be specialized to the ordinary Serre characteristic $\cat{e}(X)$. See Remark \ref{rem:numerical} below.
\subsubsection{Fixed genus formula for $\Mbar_{g,n}$} Our first application is to the moduli space $\Mbar_{g, n}$. Our  theorem applies to the following generating function for $S_n$-equivariant Serre characteristics:
\begin{equation}\label{eqn:abarg_defn} \overline{\cat{a}}_g := \sum_{n > 2-2g} \mathsf{e}^{S_n}(\Mbar_{g, n }) \in K_0(\cat{MHS}) \otimes_{\Z} \hat{\Lambda},
\end{equation}
where
\[\hat{\Lambda}:= \QQ[\![p_1, p_2, \ldots]\!] \]
is the degree completion of $\Lambda$. When $g \geq 2$, we give a formula for $\overline{\cat{a}}_g$ in terms of $\overline{\cat{a}}_0$ and $\cat{a}_h$ for $h \leq g$, where
\begin{equation}\label{eqn:ah_defn}
\cat{a}_h:= \sum_{n >  2-2h} \mathsf{e}^{S_n}(\M_{h, n }) \in K_0(\cat{MHS}) \otimes_{\Z} \hat{\Lambda}.
\end{equation}
Our formula involves a sum over graphs: let $\Gamma_g$ denote the groupoid of stable graphs of genus $g$, as defined in e.g. \cite[\S2.2]{cgp} or \cite[\S2]{GetzlerKapranov}, and let $\hat{\Gamma}_g$ be the groupoid obtained by barycentrically subdividing each graph in $\Gamma_g$ (see Definition \ref{def:bary_stable_graphs}). To state our first main theorem, we set some notation for partitions and associated operations on symmetric functions.

\begin{notn}
Write $\Part^\star$ for the set of all integer partitions, and use the notation $\varnothing \in \Part^\star$ for the empty partition.
\begin{itemize}
    \item Given $\mu \in \Part^\star$, we use the notation $\mu =(1^{\mu_1}2^{\mu_2}\dots)$, where $\mu_j$ is the number of parts of size $j$. We set $|\mu|:=\sum_{i\geq 1} i\cdot \mu_i$ for the size of $\mu$, and write $\Part^\star_n \subset \Part^\star$ for the subset of partitions of size $n$.
    \item  For $i \geq 1$, we use $\psi_i: \hat{\Lambda}\to \hat{\Lambda}$ to denote the $i$th Adams operation. This is the ring homomorphism determined by $p_j\mapsto p_{ij}$ for all $j$, and which extends in a natural way to $K_0(\cat{MHS})\otimes_{\Z} \hat{\Lambda}$ using \cite[\S 5]{GetzlerPandharipande}. If $\mu \in \Part^\star$, we define $\psi_\mu : \hat{\Lambda}\to \hat{\Lambda}$ by
    \[\psi_{\mu}(f) = \prod_{i \geq 1} \psi_i(f)^{\mu_i}.\] In particular, $\psi_\varnothing(f) = 1$ for all $f \in \hat{\Lambda}$.
    \item We write
    \[ \overline{\partial}_{\mu}: K_0(\cat{MHS}) \otimes_{\Z}\hat{\Lambda} \to K_0(\cat{MHS}) \otimes_{\Z}\hat{\Lambda} \]
    for the operator\footnote{This operation is sometimes called skewing with respect to $p_\mu = \prod_{i >0} p_i^{\mu_i}$.}
    \[ \overline{\partial}_\mu f = \left(\prod_{i \geq 1} i^{\mu_i}\right) \frac{\partial^{\mu_1 + \mu_2 + \cdots} f}{\partial p_1^{\mu_1} \partial p_2^{\mu_2}\cdots}. \]
    \item Given $f\in \hat{\Lambda},$ we use $f'$, $f''$, and $\dot{f}$ to denote $\frac{\partial f}{\partial p_1},\frac{\partial^2 f}{\partial p_1^2},$ and $\frac{\partial f}{\partial p_2}$, respectively.
\end{itemize}
\end{notn}

Our formula is stated in the language of \textit{$2$-partitions}. A $2$-partition is a function $\Theta: \Part^\star \to \Part^\star$ such that $\Theta(\mu) = \varnothing$ for all but finitely many $\mu$. Each 2-partition indexes a conjugacy class in a direct product of wreath products of symmetric groups: see Remark \ref{rem:prelim_2_part} below, or Lemma \ref{lem:conjugacy-classes}. We write $\Partt$ for the set of $2$-partitions.

\begin{customthm}{A}\label{thm:fixed-genus-thm}
Fix an integer $g \geq 2$. Given an integer $h \geq 0$ and a partition $\pi$, set
\[ \cat{w}_{h}^{\mu}:= \begin{cases}
\overline{\partial}_{\mu}\mathsf{a}_h &\text{if }2h-2 + |\mu| > 0,\\[10pt] 
 \displaystyle{\frac{1}{1 - \mathsf{a}_0''}} &\text{if } h = 0,\,\mu = 1^2,\\[12pt]
 \displaystyle{\frac{1 + 2 \dot{\mathsf{a}}_0}{1 - \psi_2(\mathsf{a}_0'')}} &\text{if } h=0,\,\mu = 2^1.
 \end{cases}
\]
Then we have
\[\overline{\cat{a}}_g = \left(\sum_{\Theta_0, \ldots \Theta_g \in \Partt}  K(\Theta_0, \ldots,  \Theta_g) \prod_{h = 0}^{g} \prod_{\mu \in \Part^\star} \psi_{\Theta_h(\mu)}\left(\cat{w}_h^{\mu}  \right) \right) \circ \left(p_1 + \overline{\cat{a}}_0'\right),\] 
where $K(\Theta_0, \ldots, \Theta_g) \in \QQ$ is defined as follows: for any 
$\G \in \mathrm{Ob}(\hat{\Gamma}_g),$ the tuple $\Theta_0, \ldots, \Theta_g$ of 2-partitions specifies a union of conjugacy classes $\Aut^{\Theta_0, \ldots, \Theta_g}(\G) \subset \Aut(\G),$ 
and
\[ K(\Theta_0, \ldots, \Theta_g) := \sum_{\G \in \mathrm{Iso}(\hat{\Gamma}_g)} \frac{|\Aut^{\Theta_0, \ldots, \Theta_g}(\G)|}{|\Aut(\G)|}.  \]

\end{customthm}


\begin{rem}\label{rem:prelim_2_part}
We briefly explain the appearance of $2$-partitions in Theorem \ref{thm:fixed-genus-thm}. For $\G \in \hat{\Gamma}_g$, write $\tup{\nu}(\G)_i^{(h)}$ for the number of vertices of $\G$ which have valence $i$ and weight $h$. Then there is an embedding of groups
\begin{equation}\label{eqn:intro-embedding-weights}
\Aut(\G) \hookrightarrow \bbS_{\tup{\nu}(\G)}: = \prod_{h = 0}^{g}\left( \prod_{i \geq 0} S_i \wr S_{\tup{\nu}(\G)_i^{(h)}}\right),
\end{equation}
well-defined up to conjugacy; here for integers $m, n \geq 0$ we define $S_m \wr S_n := (S_m)^n \rtimes S_n$. The conjugacy classes of $\bbS_{\tup{\nu}(\G)}$ are in natural bijection with $(g + 1)$-tuples $(\Theta_0, \ldots, \Theta_g)$ of $2$-partitions such that
\begin{equation}\label{eqn:condtition-2-part-weights}
\sum_{\mu \in \Part^\star_i} |\Theta_h(\mu)| = \tup{\nu}(\G)^{(h)}_i.  \end{equation}
for each $i \geq 0$ and $h \in \{0, \ldots, g\}$. Then \[\Aut^{\Theta_0, \ldots, \Theta_g}(\G) \subset \Aut(\G) \]
is defined as the preimage of the corresponding conjugacy class under (\ref{eqn:intro-embedding-weights}). Condition (\ref{eqn:condtition-2-part-weights}) and the fact that $\hat{\Gamma}_g$ has finitely many isomorphism classes imply that $K(\Theta_0, \ldots, \Theta_g) = 0$ for all but finitely many tuples $\Theta_0, \ldots, \Theta_g$. In particular, up to plethysm with $p_1 + \overline{\cat{a}}_0'$, Theorem \ref{thm:fixed-genus-thm} gives a formula for $\overline{\cat{a}}_g$ as a sum of finitely many terms.
\end{rem}

\begin{exmp}
    The formula in Theorem \ref{thm:fixed-genus-thm} is derived by summing contributions from each individual graph in $\hat{\Gamma}_g$, and then taking plethysm with $p_1 + \overline{\cat{a}}_0'$. See Figure \ref{fig:genus-two-fig} for the case $g = 2$.
\end{exmp}

\subsubsection{All-genus formula for $\Mbar_{g,n}$} Our work also gives a formula for the total generating function
\begin{equation}\label{eqn:abar_defn}
\overline{\cat{a}} := \sum_{g \geq 0} \overline{\cat{a}}_g \cdot t^{g-1} \in K_0(\cat{MHS}) \otimes_{\Z} \hat{\Lambda}(\!(t)\!) 
\end{equation}
in terms of
\begin{equation}\label{eqn:a_defn}
\cat{a} := \sum_{g \geq 0} \cat{a}_g \cdot t^{g-1} \in K_0(\cat{MHS}) \otimes_{\Z} \hat{\Lambda}(\!(t)\!),
\end{equation}
where $\overline{\cat{a}}_g$ and $\cat{a}_g$ are as in (\ref{eqn:abarg_defn}) and (\ref{eqn:ah_defn}), respectively. For this formula, we write $\Graph$ for the groupoid of all connected graphs. Given $G\in \mathrm{Ob}(\Graph)$, a $2$-partition $\Theta \in \Partt$ determines a union of conjugacy classes
\[ \Aut^{\Theta}(G) \subset \Aut(G). \]
Given $\Theta \in \Partt$, define
\[||\Theta||:= \sum_{\mu \in \Part^\star} |\Theta(\mu)|\cdot |\mu| \in \Z. \]

\begin{customthm}{B}\label{thm:all_graphs}
We have
\[ \overline{\cat{a}} = \sum_{\Theta \in \Partt} O(\Theta)\cdot t^{||\Theta||/2}  \prod_{\mu \in \Part^\star} \psi_{\Theta(\mu)} (\overline{\partial}_\mu \cat{a}), \]
where $O(\Theta)\in \mathbb{Q}$ is defined by
\[O(\Theta) := \sum_{G \in \mathrm{Iso}(\Graph)} \frac{|\Aut^{\Theta}(G)|}{|\Aut(G)|}. \]
\end{customthm}

\begin{figure}[h]
    \centering
    \includegraphics[scale=1]{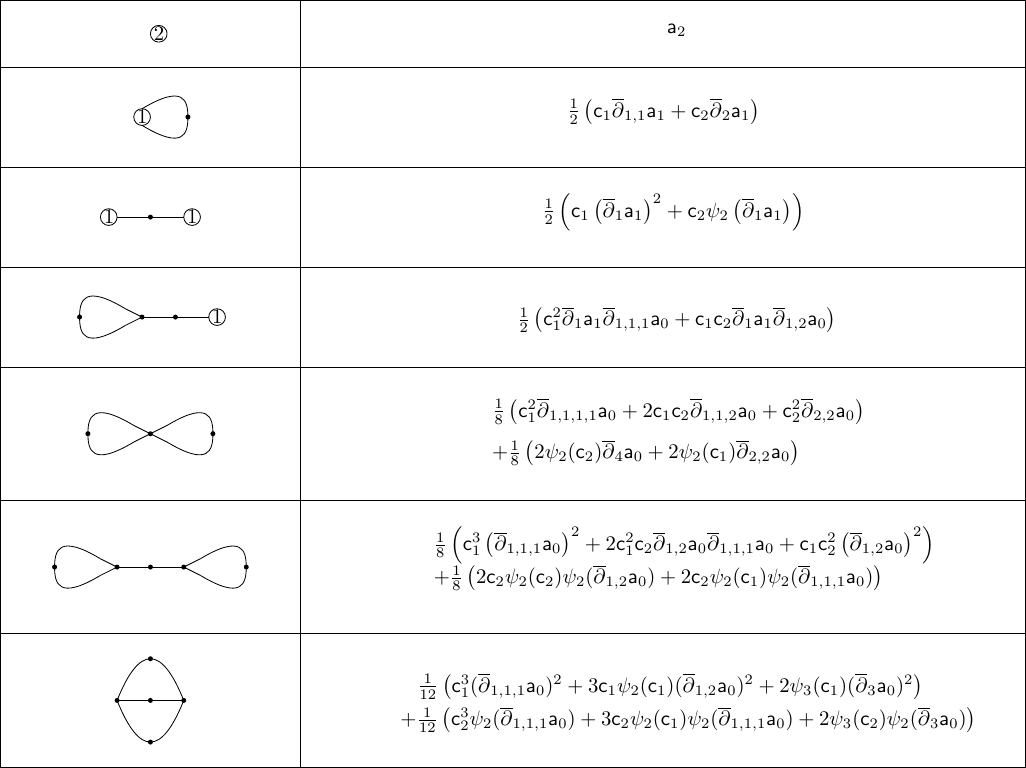}
    \caption{Computing $\overline{\mathsf{a}}_2$ as a sum over graphs. The formula in Theorem \ref{thm:fixed-genus-thm} is obtained by summing each graph contribution, and then taking plethysm with $p_1 + \overline{\cat{a}}_0'$. In the table we have set $\cat{c_1}:= \frac{1}{1 - \cat{a}_0''}$ and $\cat{c}_2:= \frac{1 + 2 \dot{\cat{a}}_0}{1 - \psi_2(\cat{a}_0'')}$.}
    \label{fig:genus-two-fig}
\end{figure}
\begin{rem}
    As $\Aut^{\Theta}(G) \neq \varnothing$ only for finitely many graphs $G$, the constants $O(\Theta)$ are well-defined. Also, $\Aut^{\Theta}(G) \neq \varnothing$ implies that $||\Theta||$ is even, so the factor of $t^{||\Theta||/2}$ is an integral power of $t$.
\end{rem}

\begin{rem}
    The formula in Theorem \ref{thm:all_graphs} can be interpreted as an expansion of the right-hand side of the Getzler--Kapranov formula \cite[Theorem 8.13]{GetzlerKapranov}, which we recall in our notation as Theorem \ref{thm:GKformula} below. Our proof of Theorem \ref{thm:all_graphs} is independent from their derivation. See \S\ref{subsubsec:modular_operads} below for a more in-depth discussion.
\end{rem}

\subsubsection{The Euler characteristic of the moduli space of stable maps} Our techniques also apply to torus-fixed stable maps. Suppose $X$ is a variety with a nontrivial $\C^\star$-action, such that
\begin{itemize}
\item the subset $X^{\C^\star}\subset X$ of $\C^\star$ fixed points consists of a finite set of isolated points, and
\item there are finitely many $\C^\star$-invariant curves in $X$, and each one is smooth and proper, hence isomorphic to $\P^1$.
\end{itemize}
These hypotheses are satisfied for e.g. complete simplicial toric varieties and flag varieties. In this setting, fix a homology class $\beta \in H_2(X;\Z)$, and define the generating function
\[\overline{\cat{a}}_{X, \beta}^{\C^\star}:=  \sum_{g, n \geq 0} \cat{e}^{S_n}(\Mbar_{g, n}(X, \beta)^{\C^\star}) t^{g - 1} \in K_0(\cat{MHS}) \otimes_{\Z} \LLambda. \]
We associate to the $\C^\star$-action on $X$ a graph $G_X$, whose vertices represent $\C^\star$-fixed points and whose edges represent $\C^\star$-invariant curves; in particular, each ${e \in E(G_X)}$ defines a homology class ${\beta_e \in H_2(X;\Z)}$. We let $\cat{\Gamma}_{X, \beta}$ denote the groupoid of triples $(G, f, \delta)$ where $G$ is a connected graph, $f: G \to G_X$ is a graph homomorphism, and $\delta: E(G) \to \Z_{>0}$ is a weighting such that $\sum_{e \in E(G)} \delta(e) \beta_{f(e)} = \beta$. Following work of Graber and Pandharipande \cite{graberpand}, we derive a formula for $\overline{\cat{a}}_{X, \beta}^{\C^\star}$.
Set
\[ \widehat{{\cat{a}}} := t^{-1}(h_1 + h_2) + \overline{\cat{a}}, \]
where $h_n \in \Lambda$ denotes the $n$th homogeneous symmetric function \cite{Macdonald}.
\begin{customthm}{C}\label{thm:torus_fixed_all_graphs}
The formula
\[\overline{\cat{a}}^{\C^\star}_{X, \beta} = \sum_{\Theta \in \Partt} O_{X, \beta}(\Theta)\cdot t^{||\Theta||/2}  \prod_{\mu \in \Part^\star} \psi_{\Theta(\mu)}\left(\overline{\partial}_{\mu}\widehat{\cat{a}}\right)  \]
holds, where $O_{X, \beta}(\Theta) \in \QQ$ is defined by
\[O_{X, \beta}(\Theta) := \sum_{(G, f, \delta) \in \mathrm{Iso}(\cat{\Gamma}_{X, \beta})} \frac{|\Aut^{\Theta}(G, f, \delta)|}{|\Aut(G, f, \delta)|}. \]
\end{customthm}

Via $\C^\star$-localization, Theorem \ref{thm:torus_fixed_all_graphs} determines the $S_n$-equivariant topological Euler characteristic
\[ \chi^{S_n}(\Mbar_{g, n}(X, \beta)) := \sum_{i} (-1)^i \ch(H^i(\Mbar_{g, n}(X, \beta);\QQ)) \in \Lambda \]
in terms of the $S_n$-equivariant topological Euler characteristics of $\Mbar_{g, n}$ and finite graph sums; above we have set $\ch(V) \in \Lambda$ for the Frobenius characteristic of an $S_n$-representation $V$. The $S_n$-equivariant topological Euler characteristics of $\Mbar_{g, n}$ can be computed by combining work of Gorsky \cite{Gorsky} on the $S_n$-equivariant topological Euler characteristics of $\calM_{g,n}$ with either Theorem \ref{thm:fixed-genus-thm} or the Getzler--Kapranov formula (recalled as Theorem \ref{thm:GKformula} below). We used this perspective together with combinatorial techniques for evaluating the relevant graph sums in \cite{genusonechar} in order to compute $\chi^{S_n}(\Mbar_{1, n}(\P^r, d))$.

See Tables \ref{table:stable_maps_no_markings} and \ref{table:degree3-frob-chars} for sample calculations of Euler characteristics via Theorem \ref{thm:torus_fixed_all_graphs} when $X = \P^r$ and $\beta = 3\cdot[L]$, where $[L]$ is the class of a line in $\P^r$.

\begin{rem}
The groupoid $\cat{\Gamma}_{X, \beta}$ has finitely many isomorphism classes, so Theorem \ref{thm:torus_fixed_all_graphs} is a formula for $\overline{\cat{a}}_{X, \beta}^{\C^\star}$ as the sum of finitely many terms. In fact, $\overline{\cat{a}}_{X, \beta}^{\C^\star}$ is a polynomial expression of finitely many terms of the form $\psi_k(\overline{\partial}_\mu\overline{\cat{a}})$. When $X = \P^r$ and $\beta = d\cdot[L]$, this polynomial is of degree $d + 1$. If instead we view $\overline{\cat{a}}_{\P^r, d}^{\C^\star}$ as a function of $r$, it is also a polynomial of degree $d + 1$; this can be deduced from basic properties of the chromatic polynomial \cite[Theorem C]{genusonechar}. See Figure \ref{fig:stable_map_exmp} for the calculation of $\overline{\cat{a}}_{\P^r, d}^{\C^\star}$ when $d = 3$.
\end{rem}

\begin{table}[h]
\def\arraystretch{1.5}
\begin{tabular}{|c|c|}
\hline
$g$ & $\chi \left(\Mbar_{g, 0}(\P^r, 3)\right)$                                                           \\ \hline
$0$ & 

$16\binom{r+1}{4}+21\binom{r+1}{3}+6\binom{r+1}{2}$                                 \\ \hline
$1$ & 
$216\binom{r+1}{4}+247\binom{r+1}{3}+55\binom{r+1}{2}$                                    \\ \hline
$2$ & $3160\binom{r+1}{4}+3342\binom{r+1}{3}+645\binom{r+1}{2}$ \\ \hline

$3$ & $44800\binom{r+1}{4}+45114\binom{r+1}{3}+8088\binom{r+1}{2}$ \\ \hline

$4$ & $630352\binom{r+1}{4}+613213\binom{r+1}{3}+104208\binom{r+1}{2}
$ \\ \hline
\end{tabular}

\caption{The topological Euler characteristic of $\Mbar_{g, 0}(\P^r, 3)$, for $g \leq 4$.}
\label{table:stable_maps_no_markings}
\end{table} 

\begin{table}[h]
\def\arraystretch{1.5}
\begin{tabular}{|c|c|}
\hline
$n$ & $\chi^{S_n} \left(\Mbar_{2, n}(\P^r, 3)\right)$                                                           \\ \hline
$0$ & 

$3160 \binom{r+1}{4}+ 3342 \binom{r+1}{3}+645 \binom{r+1}{2}$                                 \\ \hline
$1$ & 
$\left[23804 \binom{r+1}{4} +23376 \binom{r+1}{3} +4056 \binom{r+1}{2}\right]s_1$                                    \\ \hline
$2$ & $\left[122244\binom{r+1}{4}+115425\binom{r+1}{3}+18876\binom{r+1}{2}\right]s_2 + \left[79420 \binom{r+1}{4} + 72534 \binom{r+1}{3}+11302 \binom{r+1}{2}\right]s_{11}$ \\ \hline

$3$ & \begin{tabular}{@{}c@{}}$\left[538856\binom{r+1}{4}+494952\binom{r+1}{3}+77658\binom{r+1}{2} \right]s_3 + \left[614532\binom{r+1}{4}+549792\binom{r+1}{3}+82844\binom{r+1}{2}\right]s_{21}$ \\ 
 $+ \left[154148\binom{r+1}{4}+131406\binom{r+1}{3}+18408\binom{r+1}{2}\right]s_{111}$\end{tabular}    \\ \hline

\end{tabular}

\caption{The $S_n$-equivariant topological Euler characteristic of $\Mbar_{2, n}(\P^r, 3)$, for $n \leq 3$.}
\label{table:degree3-frob-chars}
\end{table} 

\subsubsection{Moduli spaces of curves of compact type} Our techniques also apply to any subspace of the above moduli spaces defined by a graph-theoretic condition. One example of particular interest is the moduli space
\[ \M_{g, n}^{\mathrm{ct}} \subset \Mbar_{g,n} \]
of curves of compact type; a stable curve is of compact type if and only if its dual graph is a tree. Theorem \ref{thm:fixed-genus-thm} holds for these spaces, if we replace $\hat{\Gamma}_g$ by the subgroupoid $\hat{\Gamma}_g^{(0)} \subset \hat{\Gamma}_g$ where the underlying graph is a tree. See Theorem \ref{thm:graph-genus-gen-fun}. As a basic example of the theory, we work out the generating function 
\[ \cat{a}_g^{(0)} := \sum_{n \geq 0} \cat{e}^{S_n}(\M_{g, n}^{\mathrm{ct}}) \]
when $g = 3$. See Table \ref{table:genus3ct} for the first $4$ terms, and Figure \ref{fig:genus3ct} in \S\ref{sec:comb_subspaces} for the full generating function. Since the formula for the generating function $\cat{a}_g^{(0)}$ depends on $\cat{e}^{S_n}(\M_{g, n})$, which when $g = 3$ is known up to $n = 11$ \cite{BergstromFaber, CLP}, explicit calculations of $\cat{e}^{S_n}(\M_{3, n}^{\mathrm{ct}})$ as in Table \ref{table:genus3ct} are possible for $n \leq 11$.

Our techniques also allow us to prove a functional equation computing $\overline{\cat{a}}$ in terms of $\cat{a}$ and the generating functions $\cat{a}^{(0)}_g$, see Theorem \ref{thm:core_gen_fun}.

\begin{table}[]
\begin{tabular}{|c|c|}
\hline
$n$ & $\cat{e}^{S_n}(\M_{3, n}^{\mathrm{ct}})$                               \\ \hline
$0$ & $\mathbb{L}^6+2\mathbb{L}^5+2\mathbb{L}^4+\mathbb{L}^3+1$                                         \\ \hline
$1$ & $(\mathbb{L}^7+4\mathbb{L}^6+7\mathbb{L}^5+4\mathbb{L}^4+\mathbb{L}^3+1)s_{1}$                    \\ \hline
$2$ & \makecell{$(2\mathbb{L}^7+7\mathbb{L}^6+7\mathbb{L}^5+\mathbb{L}^4-4\mathbb{L}^3-3\mathbb{L}^2)s_{1 1}$\\ $+ (\mathbb{L}^8+6\mathbb{L}^7+17\mathbb{L}^6+17\mathbb{L}^5+6\mathbb{L}^4-2\mathbb{L}^2+1)s_{2}$}    \\ \hline
$3$ &\makecell{$(2\mathbb{L}^7+6\mathbb{L}^6-11\mathbb{L}^4-8\mathbb{L}^3- 3\mathbb{L}^2+1)s_{1 1 1}$\\ $+(4\mathbb{L}^8+24\mathbb{L}^7+43\mathbb{L}^6+22\mathbb{L}^5-9\mathbb{L}^4-16\mathbb{L}^3-9\mathbb{L}^2+1)s_{2 1}$ \\$+(\mathbb{L}^9+8\mathbb{L}^8+32\mathbb{L}^7+52\mathbb{L}^6+32\mathbb{L}^5+6\mathbb{L}^4-6\mathbb{L}^3-5\mathbb{L}^2)s_{3}$}              
\\ \hline
$4$ & \makecell{$(2\mathbb{L}^7+\mathbb{L}^6-10\mathbb{L}^5-16\mathbb{L}^4-8\mathbb{L}^3+\mathbb{L}^2+2\mathbb{L}+1)s_{1 1 1 1}$\\ $+(9\mathbb{L}^8+40\mathbb{L}^7+36\mathbb{L}^6-25\mathbb{L}^5-55\mathbb{L}^4-31\mathbb{L}^3-4\mathbb{L}^2+5\mathbb{L}+2)s_{2 1 1}$ \\$+(2\mathbb{L}^9+22\mathbb{L}^8+67\mathbb{L}^7+66\mathbb{L}^6+9\mathbb{L}^5-26\mathbb{L}^4-24\mathbb{L}^3-5\mathbb{L}^2+2\mathbb{L}+1)s_{2 2}$ \\$+(6\mathbb{L}^9+50\mathbb{L}^8+144\mathbb{L}^7+141\mathbb{L}^6+22\mathbb{L}^5-54\mathbb{L}^4-50\mathbb{L}^3-15\mathbb{L}^2+2\mathbb{L}+1)s_{3 1}$ \\$+ (\mathbb{L}^{10}+10\mathbb{L}^9+53\mathbb{L}^8+125\mathbb{L}^7+125\mathbb{L}^6+48\mathbb{L}^5-10\mathbb{L}^4-22\mathbb{L}^3-7\mathbb{L}^2)s_{4}$} \\ \hline
\end{tabular}
\caption{The $S_n$-equivariant Serre characteristic of $\M_{3, n}^{\mathrm{ct}}$ for $n \leq 4$, computed using Theorem \ref{thm:graph-genus-gen-fun}, via the graph contributions in Figure \ref{fig:genus3ct}. We use the notation $\mathbb{L} := [H^2_c(\mathbb{A}^1;\QQ)] \in K_0(\cat{MHS})$.}
\label{table:genus3ct}
\end{table}

\begin{rem}\label{rem:numerical}
    If $X$ is a variety with $S_n$-action, then $\cat{e}^{S_n}(X)$ also determines the ordinary Serre characteristic, and hence the numerical $E$-polynomial. Indeed, if one defines \[\mathrm{rk}:K_0(\cat{MHS)\otimes_{\Z}}\hat{\Lambda} \to K_0(\cat{MHS) \otimes_{\Z}}\QQ[\![x]\!]\] by $p_1 \mapsto x$ and $p_n \mapsto 0$ for $n > 1$, then we have
    \[ \mathrm{rk}(\cat{e}^{S_n}(X)) = \cat{e}(X) \cdot \frac{x^n}{n!}.  \]
    Therefore all of our theorems on generating functions for $S_n$-equivariant Serre characteristics specialize to formulas for ordinary Serre characteristics.
\end{rem}

\subsection{P\'olya--Petersen characters}
 Our approach separates the equivariant Serre characteristic into a combinatorial part and a geometric part. The geometric part is the calculation of the series $\cat{a}$, which is essentially the information of the $S_n$-equivariant virtual Hodge numbers of $\calM_{g, n}$. The combinatorial part is what we call the \textit{P\'olya--Petersen character} of a finite graph. This construction is a common refinement of P\'olya's cycle index polynomial \cite{PolyaRead} and Petersen's type-$B$ character associated to a cycle \cite{semiclassicalremark}: in addition to vertex permutations recorded by the cycle index polynomial, the new character also keeps track of permutations of the half-edges $H(G)$ which are compatible with the root map $H(G)\to V(G)$. The automorphism group of the set map $\Aut(H(G) \to V(G))$ is a product of wreath products of symmetric groups as we explain in \S\ref{sec:intpart}. 

To keep track of the character theory of products of wreath symmetric groups, we construct in \ref{sec:ringwring} the ring $\Lambda^{[2]}$ of 2-symmetric functions.\footnote{This is distinct from the ring $\Lambda \otimes \Lambda$ of bisymmetric functions.} The ring $\Lambda^{[2]}$ has the following descriptions:\begin{itemize}
    \item the polynomial ring
    \[ \QQ[p_1(\mu), p_2(\mu), \ldots \mid \mu \in \Part^\star], \]
    \item the tensor product $\bigotimes_{i\geq 0} \Lambda(S_i)$, where $\Lambda(S_i)$ denotes the $S_i$-wreath symmetric functions of Macdonald \cite{Macdonald,MacDonald1980},
    \item the Grothendieck ring of the category of bounded $\mathbb{S}^{[2]}$-modules, where an $\bbS^{[2]}$-module is the data of a finite-dimensional rational representation for each finite product of wreath products of symmetric groups (\S\ref{sec:S2modules}),
    \item the Grothendieck ring of polynomial functors from $\mathbb{S}$-modules to vector spaces, where a $\bbS$-module is a finite-dimensional rational representation of $S_n$ for each $n \geq 0$.\footnote{$\bbS$-modules are sometimes called \textit{symmetric sequences} of vector spaces.}
\end{itemize}

The last characterization gives rise to a a plethystic action of $\Lambda^{[2]}$ on $\Lambda\otimes \Lambda$, namely a map \[\tilde{\circ}:\Lambda^{[2]}\times (\Lambda\otimes \Lambda)\to \Lambda\] that we describe and characterize in \S\ref{sec:tcirc}. The P\'olya--Petersen character is valued in the ring $\WRing$, and graph contributions are expressed in terms of the $\tcirc$ operator.


\begin{defn}
Let $G$ be a possibly decorated graph and let $\mathrm{Aut}(G)$ denote the decoration-preserving automorphism of the graph. Its P\'olya--Petersen character is defined as \[\zeta_G:=\mathrm{ch}\left(\mathrm{Ind}_{\mathrm{Aut}(G)}^{\mathrm{Aut}(H(G)\to V(G))}\mathrm{triv}_{\mathrm{Aut}(G)}\right)\in R(\mathrm{Aut}(H(G)\to V(G)))\subset \WRing,\]
where $\ch$ is the character of a representation, in the form of the generalized Frobenius characteristic defined in \S\ref{subsec:frob_char}.
\end{defn}


The following result is an example of how P\'olya--Petersen characters appear in our work. It implies Theorem \ref{thm:all_graphs} after the calculation of explicit formulas for the characters $\zeta_G$, and makes precise how the calculation of $\overline{\cat{a}}$ is separated into combinatorial and geometric parts.
\begin{thm*}
Let $\cat{Graph}$ be the groupoid of all connected graphs, and let $\Delta: \Lambda \to \Lambda \otimes \Lambda$ be the coproduct on the ring of symmetric functions. Then
   \begin{equation}\label{eqn:genformula}
       \overline{\cat{a}} = \sum_{G \in \mathrm{Iso}(\Graph)} t^{|E(G)|}(\zeta_G \tcirc \Delta\cat{a}) 
   \end{equation}
\end{thm*}


In fact, formula (\ref{eqn:genformula}) holds on the level of the Grothendieck ring of varieties, as discussed in \S\ref{sec:graph_strat}. The generalization of (\ref{eqn:genformula}) which allows for vertices of various genera is discussed in \S\ref{sec:ppchar}, and implies Theorem \ref{thm:fixed-genus-thm} .

\subsection{Calculation techniques}

We now give some comments on the shape of the formulas $\overline{\cat{a}}_g$ and $\overline{\cat{a}}$ that illustrate the technical tools useful for calculations. 

\subsubsection{Enriched plethysm and differential operators} We emphasize that while the Pólya--Petersen characters take value in the 2-symmetric function ring $\WRing,$ the formulas in Theorems \ref{thm:fixed-genus-thm}, \ref{thm:all_graphs}, and \ref{thm:torus_fixed_all_graphs} only involve ordinary symmetric functions in $\hat{\Lambda}$ and standard operations on them. 

This is an instance of the general fact that the enriched plethysm $\tcirc: \WRing\times (\Lambda\otimes \Lambda)\to \Lambda$ is determined by the action of the enriched power sums $p_n(\mu) \in \WRing$, and these act by a standard projection followed by an Adams operation. Moreover, under the ring isomorphism $\WRing\cong \bigotimes_{i\geq 0}\Lambda(S_i),$ the enriched plethysm $\tcirc: \WRing\times (\Lambda\otimes \Lambda)\to \Lambda$ is compatible with and determined by the $S_i$-plethysms $\circ_{S_i}:\Lambda(S_i)\times (\Lambda_i\otimes \Lambda)\to \Lambda$ defined by Macdonald \cite{MacDonald1980}. Precise statements are given in \S\ref{sec:tcirc}. The upshot is that applying the action of P\'olya--Petersen characters on bisymmetric functions produces ordinary symmetric functions that can be efficiently computed and related to the existing literature on $S_n$-equivariant Euler characteristics of moduli spaces.

\subsubsection{Grafting}
Our derivation of Theorem \ref{thm:fixed-genus-thm} overcomes the basic challenge that there are infinitely many dual graph strata in $\bigsqcup_{n}\Mbar_{g,n}.$ To arrive at a formula with a finite number of terms, we observe that relative to $\Mbar_g$, the combinatorial complexity of strata comes from:
\begin{itemize}
    \item attaching trees of genus zero vertices, which we call `rational tails,' and 
    \item subdividing edges by stable genus zero vertices, where the paths of stable genus zero vertices are called `caterpillars.'
\end{itemize}
The formula in Theorem \ref{thm:fixed-genus-thm} translates the combinatorial operations described above into plethystic formulas that we now briefly summarize. See  \S\ref{sec:catails} for more details.
\begin{defn}
    Let $\Mbar_{g,n}^{\mathrm{nrt}}\subset \Mbar_{g,n}$ be the subspace of stable curves that do not contain any nontrivial connected subcurve of arithmetic genus $g.$ Let $$\overline{\mathsf{a}}_{g}^{\mathrm{nrt}}:=\sum_{n\geq 0}\mathsf{e}^{S_n}(\Mbar_{g,n}^{\mathrm{nrt}})$$ be the generating function for the $S_n$-equivariant Serre characteristics.
\end{defn} An intermediate step in deriving Theorem \ref{thm:fixed-genus-thm} is the well-known formula (Lemma \ref{lem:attaching_rat_tails}) \[\overline{\mathsf{a}}_g = \overline{\mathsf{a}}_{g}^{\mathrm{nrt}}\circ (p_1+\overline{\mathsf{a}}_0')\] which decomposes graph strata in $\Mbar_{g,n}$ into contributions from the minimal genus $g$ subgraph--- termed the `core' of the dual graph--- and rational tails, with generating functions given by $\overline{\mathsf{a}}_g^{\mathrm{nrt}}$ and $p_1+\overline{\mathsf{a}}_0',$ respectively. 

We attach caterpillars by allowing genus zero vertices of valence two in the statement of Theorem \ref{thm:fixed-genus-thm}. The moduli space of caterpillars, defined in §\ref{sec:catails}, forms an $S_2\times \bbS$-variety, and its Frobenius characteristics are calculated in \cite[Remark 4.2]{BergstromMinabe2} and \cite[
\S6.1]{genusonechar} as $$\frac{1}{2}p_1^2\otimes \frac{1}{1 - \mathsf{a}_0''} + \frac{1}{2}p_2\otimes \frac{1 + 2 \dot{\mathsf{a}}_0}{1 - \psi_2(\mathsf{a}_0'')}\in  K_0(\mathsf{MHS}) \otimes_{\Z}\Lambda_2\otimes_{\Z} \hat{\Lambda}.$$ The two summands correspond to the terms $\mathsf{w}^{\mu}_{h}$ in Theorem \ref{thm:fixed-genus-thm} for $h = 0$ and $|\mu| = 2$ which replace genus zero vertices of valence two by caterpillars.

After these two grafting operations, we only need to work with graphs in the finite groupoid $\hat{\Gamma}_g$, which indexes the dual graphs of curves in $\Mbar_{g}$, or equivalently the cones in $\mathcal{M}_g^{\mathrm{trop}}$ \cite{ACP, CCUW, cgp}, in order to compute the generating function $\sum_{n\geq 0}\mathsf{e}^{S_n}(\Mbar_{g,n})$. A problem left open in our work is to develop a more general grafting procedure for P\'olya--Petersen characters.

\subsection{Related work}\label{sec:motpers}

The present article builds on and complements many previous results on the topology of moduli spaces of curves.

\subsubsection{Modular operads}\label{subsubsec:modular_operads} In \cite{GetzlerKapranov}, Getzler and Kapranov study Deligne's weight spectral sequence for the family of normal crossings compactifications $\{\M_{g, n} \hookrightarrow \Mbar_{g,n}\}_{2g - 2 + n >0}$ from an operadic perspective. Taking Euler characteristics, they deduce an inspirational formula.

\begin{thm}\cite[Theorem 8.13]{GetzlerKapranov}\label{thm:GKformula}
    With $\overline{\cat{a}}$ as in (\ref{eqn:abar_defn}) and $\cat{a}$ as in (\ref{eqn:a_defn}), we have \[\overline{\mathsf{a}} = \mathrm{Log}\left(\mathrm{exp}\left(\sum_{n\geq 1}t^n\left(\frac{n}{2}\frac{\partial^2}{\partial p_n^2}+\frac{\partial}{\partial p_{2n}}\right)\right)\mathrm{Exp}(\mathsf{a})\right),\] where \[\mathrm{Exp}(\cat{a}) := \sum_{\mu \in \Part^\star} \frac{\psi_{\mu}(\cat{a})}{\prod_{i \geq 1} i^{\mu_i}\mu_i!},\] and $\mathrm{Log}$ is the plethystic inverse to $\mathrm{Exp}$ with formula determined by Cadogan \cite{Cadogan}.
\end{thm}

In the language of loc. cit., Theorem \ref{thm:GKformula} is a formula for the Euler characteristic of the \textit{Feynman transform} of a modular operad. As remarked in \cite[\S0]{GetzlerKapranov}, Theorem \ref{thm:GKformula} is in a sense an $S_n$-equivariant upgrade of Wick's theorem, which extracts graph sums from asymptotic expansions of integrals \cite[Equation 0.2]{GetzlerKapranov}. As explained by the authors, the idea of attaching representations of $S_n$ at vertices of graphs appears in Hanlon--Robinson's graph enumeration techniques \cite{HanlonRobinson} which extend the work of P\'olya \cite{PolyaRead}.


Aside from its beauty, Theorem \ref{thm:GKformula} has the following advantages:
\begin{itemize}
    \item[(A1)] It is straightforward to implement on a computer algebra system (see e.g. computer implementation by Bergstr\"om \cite{BergstromData}, or the application in \cite{BergstromTommasi}).
    \item[(A2)] It is manifestly invertible \cite[Corollary 8.15]{GetzlerKapranov}.
\end{itemize} On the other hand, \cite{GetzlerKapranov} leaves the following questions open:
\begin{itemize}
\item[(Q1)] When the right-hand side of Theorem \ref{thm:GKformula} is expanded as a power series, is there a natural graph-theoretic interpretation for its coefficients?
\item[(Q2)] How can one directly compute the ``fixed genus" generating function $\overline{\cat{a}}_g$ as a finite combination of the functions $\cat{a}_h$ for $h \leq g$?
\item[(Q3)] How can similar techniques be applied to other moduli spaces with graphical stratifications?
\end{itemize}
Question (Q1) is answered by Theorem \ref{thm:all_graphs}, and question (Q2) is answered for all $g \geq 2$ by Theorem \ref{thm:fixed-genus-thm}. When it comes to question (Q3), the general theory developed in this paper applies to a wide range of moduli spaces; Theorem \ref{thm:torus_fixed_all_graphs} is an example. Question (Q2) has been answered for $g \leq 2$ in the literature: the cases $g \leq 1$ are due to Getzler \cite{GetzlerGenusZero, GetzlerSemiClassical}, and $g = 2$ was worked out recently by Diaconu \cite[Theorem 4.4]{Diaconu}. Figure \ref{fig:genus-two-fig} gives a graph-theoretic interpretation for all of the terms in Diaconu's formula. Petersen \cite{semiclassicalremark} gave an alternative calculation of $\overline{\cat{a}}_1$ which serves as a major inspiration for the present work, and we discuss this more in \S\ref{subsubsec:wreath} below.

Absent a general calculus for the graph sums appearing in Theorems \ref{thm:fixed-genus-thm} and \ref{thm:all_graphs}, Theorem \ref{thm:GKformula} retains advantage (A1) over our work. We hope to revisit these graph sums in the future: see \ref{sec:orbisumrmk}. When it comes to advantage (A2), we remark that Theorem \ref{thm:fixed-genus-thm} is in fact invertible, by first taking plethysm with $p_1 - \mathsf{a}_0'$ on both sides and then using the inclusion-exclusion principle. As such, one can obtain graph sum formulas for $\cat{a}_g$ in terms of $\overline{\cat{a}}_h$ for $h \leq g$. Then information about the set of series $\{\overline{\cat{a}}_h\}_{h \geq 0}$ can in principle be passed through these graph sums to compute the $S_n$-equivariant weight-graded Euler characteristics of $H_c^\star(\M_{g, n};\QQ)$, which are encoded by $\cat{a}$. It would be interesting to see if this approach could recover any existing calculations of these Euler characteristics; see \cite{CFGP} for weight $0$, \cite{PW2} for weight $2$, and \cite{PW11} for weight $11$. 



\subsubsection{Wreath products and polynomial functors}\label{subsubsec:wreath} A primary source of inspiration for the present work is  Petersen's calculation \cite{semiclassicalremark} of $\cat{e}^{S_n}(\Mbar_{1,n})$ using $S_2$-\textit{wreath symmetric functions}. The key formula in loc. cit. can be summarized as follows.

\begin{thm} \cite[Proposition 3.7]{semiclassicalremark}
    Let $\mathcal{D}ih_{k,n}\subset \Mbar_{1,n}$ be the $S_n$-invariant substack parametrizing necklaces of $k$ rational curves, namely the stable curves in $\Mbar_{1,n}$ whose dual graphs are a $k$-cycle. Then
    $$\sum_{n\geq 1} \mathsf{e}^{S_n}(\mathcal{D}ih_{k,n}) = \left(\mathrm{Ind}_{D_k}^{S_2\wr S_k}\mathrm{triv}\right)\circ_{S_2} \left(\sum_{n \geq 1} \cat{e}^{S_2 \times S_n}(\M_{0, n+2})\right),$$
    where $\circ_{S_2}$ denotes wreath product plethysm.
\end{thm}

The automorphism group of a $k$-cycle is given by the dihedral group $D_k.$ Consequently, $\bigsqcup_{n}\mathcal{D}ih_{k,n}$ is the $D_k$-quotient of the $k$-fold product of $\bigsqcup_n \mathcal{M}_{0,n}.$ To turn this description into a formula, Petersen points out that the embedding $D_{k}\subset S_2\wr S_k$ reflects how $D_k$ acts on the half-edges along each vertex, while the more familiar embedding $D_k\subset S_k$ for $k\geq 3$ does not. The natural setting for the calculation is thus the ring $\Lambda(S_2)$ of $S_2$-wreath symmetric functions, as it records the character theory of the groups $S_2\wr S_k,$ for all $k$. This ring was first constructed and studied by Macdonald in \cite{MacDonald1980}, which is another major inspiration for the present article.

Macdonald explains in loc. cit. that representations of $S_2\wr S_k$ are identified with \textit{polynomial functors} taking $S_2$-representations to vector spaces. This identification gives a clean conceptual picture for the above formula: the $S_2$-wreath symmetric function $\mathrm{Ind}_{D_k}^{S_2\wr S_k}\mathrm{triv}$ is a polynomial that assembles the genus zero vertex contributions to characteristics of $k$-cycles.

Our formalism generalizes Petersen's necklace calculation to strata given by arbitrary graphs. Automorphism groups of finite graphs embed into finite products of wreath products of symmetric groups, and the character theory of all such products are recorded by the 2-symmetric function ring $\WRing$. The P\'olya--Petersen characters can then be seen as polynomials that assemble $\bbS$-modules along arbitrary graphs. From this perspective, our work expands the Getzler--Kapranov formula into individual graph contributions.

\subsection{Future directions}\label{subsec:more-applications}
Our formulas can be adapted to all moduli spaces that admit appropriate recursive graphical stratifications. Potential applications of the theory include
\begin{itemize}
\item universal compactified Jacobians $\overline{\mathcal{J}}_{g, n}^{\phi}$ over $\Mbar_{g, n}$ \cite{OdaSeshadri, Caporaso, PaganiTommasi},
\item moduli of torus-fixed quasimaps $\overline{\calQ}_{g, n}(X, \beta)^{\C^\star}$ \cite{MOP, qmapconstruction}, with an eye towards motivic analogues of existing wall-crossing formulas \cite{qmapwall1,qmapwall3, qmapwall4, qmapwall5},
\item moduli of genus zero relative stable maps to toric varieties \cite{dhruvtoric, KannanP1},
\item moduli of admissible $G$-covers of stable curves for a finite group $G$ \cite{AV, ACV,JKK,PetersenOperadGCovers}, 
\item the Fulton--Macpherson compactification $X[n]$ of the configuration space \cite{FM} and the closely related Chen--Gibney--Krashen spaces $T_{d,n}$ \cite{ChenGibneyKrashen}, 
\item and finally the \textit{weighted} version of all the above moduli spaces, in the sense of Alexeev \cite{ALEXEEV1994} and Hassett \cite{Hassett}.
\end{itemize} 

\subsubsection{Orbisum computations}\label{sec:orbisumrmk} It would be worthwhile to systematically determine the graph sums $K(\Theta_0, \ldots, \Theta_g)$ and $O(\Theta)$ that appear as coefficients in Theorem \ref{thm:fixed-genus-thm} and Theorem \ref{thm:all_graphs}. In the language of \cite{CFGP}, these numbers are examples of \textit{orbisums} over finite groupoids. In loc. cit., the calculation of orbisums is reduced to the enumeration of so-called ``orbigraphs.'' Crucially, the the orbisums evaluated in \cite{CFGP} are signed, and this leads to cancellation in large families of orbigraphs. It would be interesting to see whether their techniques can be adapted to our setting. Another potential approach is via the connection between orbigraph enumeration and matrix integrals, used to great effect by Bini and Harer \cite{BiniHarer} in order to compute non-equivariant topological and orbifold Euler characteristics of $\Mbar_{g,n}$. Finally, there is a potential connection to the Fock space formalism used in \cite{HahnMarkwig2} to evaluate certain Hurwitz numbers, which can be expressed as graph sums similar to those arising in our work \cite{HahnMarkwig1}. 



\subsubsection{Asymptotic analysis} Since our approach singles out each graph's contribution in each formula, it enables quantitative comparison across different graph types and their asymptotic behaviors as numerical parameters tend to infinity. For instance, it has been observed in \cite{genusonechar} that, in line with the expectation from enumerative geometry, the contribution of maps from curves with rational tails to $\chi^{S_n}(\Mbar_{1,n}(\mathbb{P}^r,d))$ dominates that of maps from cycles of rational curves as $d\to \infty$. It is desirable to characterize the contribution of rational tails in greater generality.

\subsubsection{Stable maps and modular operads} We would be interested to know whether the spaces of torus-fixed stable maps considered in Theorem \ref{thm:torus_fixed_all_graphs} fit into the framework of modular operads or their colored variants \cite{PetersenOperadGCovers}, and whether there is an analogue of Getzler--Kapranov's Feynman transform or Theorem \ref{thm:GKformula} in this setting.



\subsection{Outline of the paper} The algebra $\WRing$ of $2$-symmetric functions is constructed in \S\ref{sec:ringwring}, and its connections to polynomial functors and $\bbS^{[2]}$-modules are developed in \S\ref{sec:WRingPfun}. In \S\ref{sec:graph_strat}, we recall various graph stratifications of $\Mbar_{g,n}$, and in \S\ref{sec:ppchar} we use these stratifications in conjunction with the P\'olya--Petersen character to derive Theorems \ref{thm:fixed-genus-thm} and \ref{thm:all_graphs}. In \S\ref{sec:comb_subspaces} we explain how our theory adapts to combinatorial subspaces of $\Mbar_{g,n}$, deriving in particular a formula for the moduli space $\M_{g, n}^{\mathrm{ct}}$ of curves of compact type (Theorem \ref{thm:graph-genus-gen-fun}), and an alternative functional equation determining the total generating function $\overline{\cat{a}}$ (Theorem \ref{thm:core_gen_fun}). Finally in \S\ref{sec:stable-maps} we apply our theory to torus-fixed stable maps, proving Theorem \ref{thm:torus_fixed_all_graphs}. Appendix \ref{appendix:wreath_products} contains background information on wreath products and associated symmetric functions, while Appendix \ref{sec:appendixPFun} contains the proof of Theorem \ref{thm:s2polyfun}, which states that $\WRing$ is the Grothendieck ring of polynomial functors from the category of $\bbS$-modules over $\QQ$ to the category of finite-dimensional $\QQ$-vector spaces.

\subsection*{Acknowledgements}
We are grateful to Jonas Bergstr\"om, Sarah Brauner, Samir Canning, Melody Chan, Dan Petersen, Oscar Randal-Williams, and Dhruv Ranganathan for helpful conversations and correspondence. SK is supported by NSF DMS-2401850. TS is supported by a Cambridge Trust international scholarship.

\section{The ring $\WRing$}
In this section we set up the ring $\WRing$ of 2-symmetric functions. The P\'olya--Petersen characters, to be defined in Section \ref{sec:ppchar}, will be elements in this ring. As $\WRing$ is modeled on Macdonald's theory of wreath symmetric functions, we state the the main result that informs our perspectives on $\WRing$. A brief overview of Macdonald's framework is given in the Appendices. We define $\Lambda(G)$ as the free polynomial ring
\[ \Lambda(G) := \QQ[ p_i(c) \mid i> 0, c \in G_*  ] \]
where $G_*$ denotes the set of conjugacy classes of $G$.

\begin{thm}\label{thm:MacdonaldSetup}\cite{MacDonald1980}
    Let $G$ be a finite group. After tensoring with a splitting field $k$ of $G$ contained in $\mathbb{C}$, the ring $\Lambda(G)$ of $G$-wreath symmetric functions is isomorphic to:
    \begin{enumerate}
        \item the Grothendieck ring, tensored with $k$, of bounded $G\wr \mathbb{S}$-vector spaces over $k$, namely sequences $(V_n)_{n\geq 0}$ where each $V_n$ is a finite-dimensional $G\wr S_n$-representation, such that $V_N = 0$ for $N\gg 0$, and the tensor product is given by
        \[ V_m \boxtimes W_n := \Ind_{G \wr S_m \times G\wr S_n}^{G \wr S_{m + n}} V_m \otimes_\C W_n; \]
        \item the Grothendieck ring, tensored with $k$, of bounded polynomial functors $\pfunc[]{G\-\mathsf{Vect}_{k}}{\mathsf{Vect}_{k}}$ from $G$-representations in $k$ to $k$-vector spaces;
        \item the free $\lambda$-ring on $G_*$ tensored with $k$.
    \end{enumerate}
\end{thm}

We now turn to the details of constructing $\WRing$. The analogies to Theorem \ref{thm:MacdonaldSetup} will be made clear in \S\ref{sec:WRingPfun} below.
\subsection{Integer partitions}\label{sec:intpart}
The group characters that are recorded by $\WRing$ are the automorphism groups of maps between finite sets, or equivalently automorphisms of generalized integer partitions, for which we set up some notation.

\begin{defn}
    Let $\mathsf{Map}$ denote the groupoid with objects being maps between finite sets, where an isomorphism between $f: X \to Y$ and $f' : X' \to Y'$ is given by a pair $(\varphi_X, \varphi_Y)$ of bijections \[\varphi_X : X \to X',\varphi_Y: Y \to Y'\] fitting into a commutative square \[\begin{tikzcd}
        X & Y \\
        {X'} & {Y'}
        \arrow["f", from=1-1, to=1-2]
        \arrow["{\varphi_X}"', from=1-1, to=2-1]
        \arrow["{\varphi_Y}", from=1-2, to=2-2]
        \arrow["{f'}"', from=2-1, to=2-2]
    \end{tikzcd}.\]
\end{defn}
    
    \begin{defn}
    We set $[0] : = \varnothing$ and $[n] := \{1, \ldots, n\}$ for $n \in \Z_{>0}$. If $m\in \Z_{>0}$, then the \textit{wreath symmetric group} $S_m \wr S_n$ is defined as the automorphism group of the surjection
    \[ \left(\coprod_{i = 1}^{n} [m]\right) \to [n]\] as an object in $\mathsf{Map}$. Recall the standard isomorphism
    \[ S_m \wr S_n \cong (S_m)^{n} \rtimes S_n,\]
    where on the right-hand side $S_n$ acts on $(S_m)^n$ by permuting the coordinates. When $m = 0$, define $S_0\wr S_n$ as $S_n$.
\end{defn}

It is helpful to pass between maps of finite sets and generalized integer partitions.

\begin{defn}
A \textit{generalized integer partition} is a finitely-supported function \[\nu: \Z_{\geq 0} \to \Z_{\geq 0},\]
where we write $\nu_i$ for $\nu(i)$, and interpret $\nu_i$ as the number of parts of $\nu$ of size $i$. Given $\nu$, define
\[ |\nu| = \sum_{i \geq 1} i \nu_i, \]
and write $\nu \vdash n$ if $|\nu| = n$.  Let $\mathsf{Part}$ denote the set of all generalized integer partitions, and write
\[ \mathsf{Part}_n \subset \mathsf{Part} \] for the subset of partitions $\nu$ with $|\nu| = n$. A generalized integer partition $\nu$ is called an \textit{integer partition} if $\nu_0=0$. The subset of integer partitions is denoted by $\Part^\star\subset \Part$ and $\Part_n^\star\subset \Part_n$, respectively. Note that $\Part$ and $\Part^\star$ are both commutative semigroups under function addition:  given $\lambda=(\lambda_0,\lambda_1,\lambda_2,\dots), \nu=(\nu_0,\nu_1,\nu_2,\dots)\in \Part$, define $\lambda \odot \nu$ as the partition that has $\lambda_i + \nu_i$ parts of size $i$ for all $i \geq 0.$
\end{defn}

Under the above definition, the usual ring $\Lambda = \QQ[p_1, p_2, \ldots]$ of symmetric functions can be thought of as the semigroup algebra of $\Part^\star$.

\begin{defn}
    The automorphism group of a generalized integer partition $\nu$ is defined as \[\bbS_\nu := \prod_{i \geq 0} S_i \wr S_{\nu_i}.\] When $\nu_0=0$, the group $\bbS_\nu$ is isomorphic to a subgroup of $S_{|\nu|}$ which preserves an integer partition of type $\nu.$ The sets $\mathsf{Part}$ and $\mathsf{Part}_n$ are hence enriched with groupoid structures.
\end{defn}

The groups $\bbS_\nu$ are exactly the automorphism groups of finite set maps.
\begin{lem}
Let $f: X \to Y$ be an object in $\mathsf{Map}$. Define an integer partition $\nu(f) \in \mathsf{Part}$ by setting
\[ {\nu(f)}_i := \# \{ y \in Y \mid \#f^{-1}(y) = i \}.\]
Then
\[ \Aut_{\mathsf{Map}}(f:X \to Y) \cong \bbS_{\nu(f)}.\]
\end{lem}


\begin{rem}
    We note that $\nu(f)_0=0$ if and only if $f$ is surjective.
\end{rem}

\begin{rem}
    More conceptually, $\mathsf{Map}$ carries a symmetric monoidal structure under disjoint unions and isomorphisms induced by set isomorphisms $A\sqcup B\cong B\sqcup A.$ Every object $f$ in $\mathsf{Map}$ is isomorphic to the disjoint union $\bigsqcup_{i\geq 0}\left(\bigsqcup_{\nu(f)_i}[i]\right)\to \bigsqcup_{i\geq 0}[\nu(f)_i].$ With this perspective, the lemma states that the automorphism groups in $\mathsf{Map}$ are compatible with the symmetric monoidal structure.
\end{rem}

\subsection{The ring $\WRing$} \label{sec:ringwring}
For $\nu \in \Part$, we write \[R(\bbS_{\nu}):= K_0(\mathrm{Rep}_{\bbS_\nu})\otimes_{\Z}{\mathbb{Q}}\]
for the free $\mathbb{Q}$-vector space spanned by isomorphism classes of finite-dimensional $\bbS_\nu$-representations over $\QQ$. Equivalently, $R(\bbS_\nu)$ is the space of $\QQ$-valued class functions of $\bbS_{\nu}$.
\begin{defn}
    As a $\QQ$-vector space, define \[\WRing:= \bigoplus_{\nu\in \Part}R(\bbS_\nu).\]
\end{defn}

The multiplication on $\WRing$ is defined via the following operation of merging partitions.

\begin{rem}
If $\lambda, \nu \in \Part$, there is an inclusion map of partition automorphism groups \[\bbS_{\lambda}\times \bbS_{\nu}\to \bbS_{\lambda\odot\nu},\] well-defined up to conjugating in $\bbS_{\lambda\odot\nu}$: it is defined by taking disjoint union of the set maps $f_{\lambda},f_{\nu}\in \mathsf{Map}$ where $\mathrm{Aut}_{\mathsf{Map}}(f_{\lambda})\cong \bbS_{\lambda}$ and $\mathrm{Aut}_{\mathsf{Map}}(f_{\nu})\cong \bbS_{\nu}.$ \end{rem}

    
    
\begin{defn}\label{defn:ringWRing}
    Define bilinear maps $\boxtimes: R(\bbS_{\lambda})\times R(\bbS_{\nu})\to R(\bbS_{\lambda\odot \nu})$ by the formula \[ V \boxtimes W:=  \Ind_{\bbS_{\lambda} \times \bbS_{\nu}}^{\bbS_{\lambda \odot \nu}} V \otimes W .\] This construction extends naturally to a bilinear map $\boxtimes: \WRing \times \WRing\to \WRing$.
\end{defn}
Write $\mathbf{0}$ for the unique element of $\Part_0 \subset \Part$.
\begin{lem}
    The map $\boxtimes: \WRing \times \WRing\to \WRing$ defines a commutative $\QQ$-algebra structure on $\WRing$, where the identity is given by $1\in R(\mathbb{S}_{\mathbf{0}}).$
\end{lem}
\begin{proof}
   The product $\boxtimes$ is distributive as we can check that $(V_1\oplus V_2)\boxtimes W = (V_1\boxtimes W)\oplus (V_2\boxtimes W)$ and similarlu  $V\boxtimes (W_1\oplus W_2) = (V\boxtimes W_1) \oplus (V\boxtimes W_2)$.  Associativity follows from the fact that the subgroup $\bbS_{\lambda}\times \bbS_{\nu}\times \bbS_{\mu}\to \bbS_{\lambda\odot\nu\odot \mu}$ is well-defined up to conjugacy, and commutativity comes from identifying $\lambda\odot \nu = \nu\odot \lambda$ in $\mathsf{Part}.$ $1\in R(\mathbb{S}_{\mathbf{0}})$ is a multiplicative unit because $\lambda\odot \mathbf{0} = \lambda = \mathbf{0}\odot \lambda$ for all partitions $\lambda$.
\end{proof}

It will be helpful to relate the ring $\WRing$ to the wreath product symmetric functions $\{\Lambda(S_m)\}_{m\geq 0}$ constructed by Macdonald \cite{MacDonald1980, Macdonald}. In the following, $(m^k)$ denotes the partition of $[mk]$ with $k$ parts of size $m.$

\begin{lem}
    The subspace $\bigoplus_{k\geq 0} R(\bbS_{(m^k)})$ forms a subalgebra of $\WRing$ which is isomorphic to $\Lambda(S_m)$.
\end{lem}
\begin{proof}
    We observe that $\bbS_{(m^k)}\cong S_m\wr S_k$ and the injection $\bbS_{(m^{k_1})}\times \bbS_{(m^{k_2})}\to \bbS_{(m^{k_1+k_2})}$ is identical to $(S_m\wr S_{k_1})\times (S_m\wr S_{k_2})\to S_m\wr S_{k_1+k_2}$. Hence, the induction formula for $\boxtimes$ is identical to Macdonald's definition of the product structure on $\Lambda(S_m)$ in \cite[§1.B.4]{Macdonald}.
\end{proof}

Since $\bbS_\lambda:=\prod_{i\geq 0}S_i\wr S_{\lambda_i}$, we have $R(\bbS_{\lambda}) = \bigotimes_{i\geq 0}R(S_i\wr S_{\lambda_i})$. The product $\boxtimes$ can be decomposed along the tensor products as follows.
\begin{lem}
    Under the identification $R(\bbS_{\lambda}) = \bigotimes_{i\geq 0}R(S_i\wr S_{\lambda_i})$, the map $$\boxtimes: R(\bbS_{\lambda})\times R(\bbS_{\nu})\to R(\bbS_{\lambda\odot \nu})$$ is the tensor product $\bigotimes_{i\geq 0} m_i,$ where $m_i: R(S_i\wr S_{\lambda_i})\times R(S_i\wr S_{\nu_i})\to R\left(S_i\wr S_{(\lambda\odot \nu )_i}\right)$ is the induction map from the group homomorphisms $S_i\wr S_{\lambda_i}\times S_i\wr S_{\nu_i}\to S_i\wr S_{\lambda_i+\nu_i}$ described above. 
\end{lem}
\begin{proof}
    The multiplication $\boxtimes: R(\bbS_{\lambda})\times R(\bbS_{\nu})\to R(\bbS_{\lambda\odot \nu})$ is induced by $\bbS_{\lambda}\times \bbS_{\nu}\to \bbS_{\lambda\odot \nu}$, which is itself a product of the group homomorphisms $S_i\wr S_{\lambda_i}\times S_i\wr S_{\nu_i}\to S_i\wr S_{\lambda_i+\nu_i}.$ Applying $R(\-)$ takes a product of group homomorphisms to tensor product of linear maps.
\end{proof}

\begin{lem}
    Let $\bigotimes_{i\geq 0}\Lambda(S_i)$ be the colimit of \[\cdots\to \bigotimes_{i=0}^N \Lambda(S_i)\xrightarrow{-\otimes 1} \bigotimes_{i=0}^{N+1}\Lambda(S_i)\to \cdots\] so that $\bigotimes_{i\geq 0}\Lambda(S_i) = \bigcup_{N=0}^\infty \bigotimes_{i=0}^N \Lambda(S_i).$ Then we have $$\WRing\cong \bigotimes_{i\geq 0}\Lambda(S_i).$$
\end{lem}
\begin{proof}
    Firstly, we prove the isomorphism on the level of $\mathbb{Q}$-vector spaces. Unwinding the definitions, $$\WRing= \bigoplus_{\lambda\in \Part}R(\bbS_\lambda) \cong \bigoplus_{\lambda\in \Part}\bigotimes_{i\geq 0}R(S_i\wr S_{\lambda_i}).$$ The last item is isomorphic to $\bigotimes_{i\geq 0}\left( \bigoplus_{n\geq 0} R(S_i\wr S_n)\right)$: the summands are matched by identifying integer partitions $\lambda$ as sequences of integers $(\lambda_0, \lambda_1,\lambda_2,\dots)$ that are zero for all but finitely many terms. The previous lemma checks that the $\mathbb{Q}$-linear isomorphism is multiplicative, hence the two are isomorphic as rings.
\end{proof}




\subsection{Power sums} 
In this section, we define the analogues of the power sum symmetric functions in the ring $\WRing$. As with $\Lambda$ and $\Lambda(S_i)$, we will prove that the generalized power sums freely generate $\WRing$ as a polynomial ring.

\begin{defn}
    Given a pair $(j, \mu)$ where $j \geq 1$ is an integer and $\mu \in \Part^\star$, define $p_j(\mu)\in R(S_{|\mu|}\wr S_j) \subset \WRing$ as the virtual character associated to the conjugacy class of an element of the form $((\sigma, e, \ldots, e), \kappa) \in S_{|\mu|} \wr S_{j}$ where $\sigma\in S_{|\mu|}$ is a permutation of cycle type $\mu$ and $\kappa\in S_{j}$ is a $j$-cycle.
\end{defn}


In the ring $\Lambda = \QQ[p_1, p_2, \ldots]$, monomials in the power sums are indexed by integer partitions: an integer partition $\theta \in \Part^\star$ gives the monomial
\[ p_\theta := \prod_{i \geq 1} p_i^{\theta_i}. \]
This correspondence between partitions and monomials is useful for bookkeeping. Analogously, monomials in the power sums of $\WRing$ are indexed by \textit{$2$-partitions}, which we now define.
\begin{defn}\label{defn:2part}
A \textit{$2$-partition} $\Theta$ is a function $\Theta: \Part^\star \to \Part^\star$ such that $\Theta(\mu) = \varnothing$ for $|\mu| \gg 0$. The set of $2$-partitions is denoted as $\Partt.$ Because $\Part^\star$ is a commutative semigroup, $\Partt$ carries a commutative semigroup structure under function addition. 
\end{defn}

\begin{defn}
Given $\nu \in \Part$, we define
\[ \Partt_\nu := \{\Theta \in \Partt \mid \sum_{\mu \in \Part^\star_i} |\Theta(\mu)| = \nu_i\, \text{for all }i \geq 0 \}. \]
\end{defn}
The point of $2$-partitions is that they naturally index conjugacy classes in the groups $\bbS_\nu$, for $\nu \in \Part$.
\begin{lem}\label{lem:conjugacy-classes}
For each $\nu \in \Part$, there is a natural bijection
\[ (\bbS_\nu)_* \to \Partt_\nu, \]
such that for any $\lambda \in \Part$, the following diagram commutes:
\[ \begin{tikzcd}
    &(\bbS_\lambda)_* \times (\bbS_\nu)_* \arrow[d] \arrow[r] & \Partt_\lambda \times \Partt_\nu \arrow[d] \\
    &(\bbS_{\lambda \odot \nu} )_* \arrow[r] & \Partt_{\lambda \odot \nu}.
\end{tikzcd} \]
\end{lem}
\begin{proof}
    Recall from Lemma \ref{lem:wrconj} that the conjugacy classes of $G\wr S_n$ are in bijection with maps $f: G_*\to \Part^\star$ such that $\sum_{[g]\in G_*}|f([g])| = n.$ Thus, the set of conjugacy classes of $\bbS_{\lambda}$ is the subset of $\Partt$ given by $$\prod_{i\geq 0} \{f \in \mathsf{Map}((S_i)_*, \mathsf{Part}^\star)\mid \sum_{[g]\in (S_i)_*} |f([g])| = \lambda_i\}$$
    which agrees with the definition of $\Partt_{\lambda}$ given above. The diagram commutes because the addition on $\Partt$ is given by merging partitions on $\Part^\star.$
\end{proof}
\begin{rem}
    The wreath symmetric groups in this work arise as automorphism groups of half-edges mapping to vertices. The proof of the Lemma suggests that in our setting, the domain and target of a 2-partition should be interpreted as partitions of the half-edges and of vertices, respectively.
\end{rem}
Using Lemma \ref{lem:conjugacy-classes} we can prove the following structural description of $\WRing$:
\begin{thm}\label{thm:WRingsemigp}
The ring $\WRing$ is canonically identified with the semigroup algebra of $\Partt$. In particular,
\[ \WRing = \QQ[p_j(\mu) \mid j > 0, \mu \in \Part^\star]. \]
\end{thm}

\begin{proof}
Lemma \ref{lem:conjugacy-classes} implies that we can identify \[\WRing = \bigoplus_{\lambda \in \Part} R(\bbS_\lambda) \cong \bigoplus_{\lambda \in \Part} \QQ \cdot (\bbS_\lambda)_* \] with the free vector space over $\QQ$ spanned by $\Partt$. The commutative diagram in Lemma \ref{lem:conjugacy-classes} implies that the multiplication on $\WRing$ from Definition \ref{defn:ringWRing} agrees with the multiplication structure on the semigroup ring generated by $\Partt$. Under this identification, $p_j(\mu) \in \WRing$ corresponds to the $2$-partition which assigns $\mu$ to the partition with one part of size $j$. As $j$ and $\mu$ vary over all positive integers and elements of $\Part^\star$ respectively, we see that these $2$-partitions freely generate $\Partt$, so the theorem is proved. 
\end{proof}

\begin{defn}\label{defn:2_part_conj}
    Given $\boldsymbol{\tau} \in \bbS_\nu$, define $\Theta^{\boldsymbol{\tau}} \in \Partt_\nu$ to be the image of the conjugacy class of $\boldsymbol{\tau}$ under the identification of Lemma \ref{lem:conjugacy-classes}.
\end{defn}
\begin{rem}\label{rem:thetatau}
    We give a more explicit recipe for the construction of $\Theta^{\boldsymbol{\tau}}$, by describing how to compute $\Theta^{\boldsymbol{\tau}}(\mu)$ for some $\mu \in \Part^\star$. Write $\boldsymbol{\tau} = (\tau_i)_{i\geq 0}\in \bbS_{\nu}$. Consider the $|\mu|$-th term of $\boldsymbol{\tau}$ given by $$\tau_{|\mu|} = ((g_1,\dots, g_{\lambda_{|\mu|}}), \sigma), \text{ where } g_i\in S_{|\mu|}, \quad \text{ and }\quad \sigma\in S_{\nu_{|\mu|}}.$$ Decompose $\sigma$ into disjoint cycles:
    \[\sigma = \sigma_1\cdots \sigma_\ell. \]
    Then $\Theta^{\boldsymbol{\tau}}(\mu)$ is defined as the partition with one part of size $j$ for each $\sigma_i$ such that
    \begin{enumerate}[(i)]
    \item $\sigma_i$ is a $j$-cycle, and
    \item if we write
    \[\sigma_i = (k_1, k_2, \ldots, k_j), \]
    then the element
    \[g_{k_1}\cdots g_{k_j} \in S_{|\mu|} \]
    has cycle type $\mu$.
    \end{enumerate} 
    
\end{rem}

\begin{defn}
    Given a pair of integer partitions $\theta, \mu \in \Part^\star$, define \[p_\theta(\mu) := \prod_{i \geq 1} p_i(\mu)^{\theta_i}\in \WRing.\]
    Given a $2$-partition $\Theta$, we define the monomial
    \[ p_\Theta := \prod_{\mu \in \Part^\star} p_{\Theta(\mu)}(\mu).\]
\end{defn}

By Theorem \ref{thm:WRingsemigp}, the monomials $p_\Theta$ for $\Theta \in \Partt$ form a linear basis for $\WRing$, and we have $p_{\Theta} \cdot p_{\Theta'} = p_{\Theta + \Theta'}$ for all $\Theta, \Theta' \in \Partt$. 


\subsection{The Frobenius characteristic of a $\bbS_\lambda$-representation}\label{subsec:frob_char}
\begin{defn}
    Let $\nu \in \Part$, and let $V$ be a complex representation of the group $\bbS_\nu$. Denote by $\ch(V) \in R(\bbS_\lambda)\subset \WRing$ the class of the representation $V$ in $R(\bbS_\nu)$. We refer to $\ch(V)$ as the \textit{Frobenius characteristic} of $V$.
\end{defn}

For any generalized partition $\nu \in \Part$, recall from Lemma \ref{lem:conjugacy-classes} that there is a canonical bijection
\[ (\bbS_\nu)_* \to \Partt_{\nu} \text{ given by } [\boldsymbol{\tau}] \mapsto \Theta^{\boldsymbol{\tau}}.\]
 The class function determined by $\boldsymbol{\tau}$ is given by $p_{\Theta^{\boldsymbol{\tau}}} \in R(\bbS_\lambda) \subset \WRing.$ 

\begin{defn}
Given $\boldsymbol{\tau} \in \bbS_\lambda$, we refer to $p_{\Theta^{\boldsymbol{\tau}}} \in \WRing$ as the \textit{cycle index} of $\boldsymbol{\tau}$.
\end{defn}

We have the following formula for the Frobenius characteristic of $\ch(V)$ in terms of cycle indices, which follows from the general isomorphism between $R(\bbS_{\lambda})$ and the space of $\C$-valued class functions on $\bbS_\lambda$.

\begin{lem}\label{lem:pre_ch_formula}
Suppose that $V$ is a finite-dimensional complex $\bbS_\lambda$-representation. Then
\[\ch(V) = \frac{1}{|\bbS_\lambda|} \sum_{\boldsymbol{\tau} \in \bbS_\lambda} \mathrm{Tr}(\boldsymbol{\tau}|V) \cdot p_{\Theta^{\boldsymbol{\tau}}}.\]
\end{lem}
\subsection{Forgetful morphisms}
There are two forgetful maps from $\WRing$ to $\Lambda$ corresponding to the two group homomorphisms $$\mathrm{Sym}(Y)\leftarrow\mathrm{Aut}_{\mathsf{Map}}(f: X\to Y)\to \mathrm{Sym}(X).$$

\begin{defn}
Define algebra homomorphisms
\[\varphi_1: \WRing  \to \Lambda,\]\[ \varphi_2: \WRing \to \Lambda\]
by assigning \[\varphi_1(p_j(\mu)):=  p_j\]\[\varphi_2(p_j(\mu)):= p_j \circ p_\mu\] respectively and extending multiplicatively.
\end{defn}

Though we will not use these maps in the present work, we include them here with an eye towards future applications of the theory. The homomorphisms $\varphi_1$ and $\varphi_2$ have the following representation-theoretic meaning. The projections $S_i \wr S_{\nu_i} \to S_{\nu_i}$ fit together to define a map $\bbS_\nu \to \prod_{i \geq 0} S_{\nu_i}$. Viewing the latter group as a Young subgroup of $S_{\sum_i \nu_i}$, we get a homomorphism
\[ \bbS_\nu \to S_{\sum_i \nu_i}, \]
well-defined up to conjugacy. The map $\varphi_1$ has the property that
\[ \varphi_1(\ch(V)) =  \ch(V \otimes_{\C[\bbS_{\nu}]}\QQ[S_{\sum_i\nu_i}]). \]
On the other hand, we have an inclusion of groups $\bbS_{\nu} \to S_{|\nu|}$, by including each $S_i \wr S_{\nu_i} \to S_{i \nu_i}$, and then including the product $\prod_{i} S_{i \nu_i}$ as a Young subgroup of $S_{|\nu|}$. The morphism $\varphi_2$ has the property that
\[ \varphi_2(\ch(V)) = \ch(V \otimes_{\QQ[\bbS_\nu]} \QQ[S_{|\nu|}]=\ch(\Ind_{\bbS_\nu}^{S_{|\nu|}} V). \]
These claims about $\varphi_1$ and $\varphi_2$ can be proved by analyzing the behavior of conjugacy classes under the described group homomorphisms.

\section{$\bbS^{[2]}$-modules and polynomial functors}\label{sec:WRingPfun}
While the previous section gives a concrete algebraic description of $\Lambda^{[2]},$ a more conceptual perspective is given by its categorifications to both a representation category and a category of polynomial functors. These enhancements give a streamlined definition of the enriched plethysm action of $\Lambda^{[2]}$ on $\Lambda\otimes \Lambda$ and provide access to the $\lambda$-ring structure. Recall that the ring \[\Lambda = \QQ[p_1, p_2, \ldots] = \bigoplus_{n \geq 0} R(S_n)\]
can be understood as the Grothendieck ring of bounded \textit{$\bbS$-modules}, where a $\bbS$-module $\calV$ is defined as a collection $\{\mathcal{V}(n)\mid n\geq 0\}$ where each $\mathcal{V}(n)$ is a finite-dimensional ${S}_n$-representation over $\QQ$; in this context, \textit{bounded} means that $\calV(n) = 0$ for all but finitely many $n$. Equivalently, an $\bbS$-module is a $\QQ$-vector space $\calV$ together with a direct sum decomposition
\[\calV = \bigoplus_{n \geq 0} \calV(n), \]
such that each $\calV(n)$ is a finite-dimensional $S_n$-representation; boundedness is equivalent to requiring that $\calV$ is also finite-dimensional.

\subsection{$\bbS^{[2]}$-modules}\label{sec:S2modules}
We define the category of $\bbS^{[2]}$-modules analogously to the category of $\bbS$-modules.
\begin{defn}
    An \textit{$\mathbb{S}^{[2]}$-module} $\mathcal{V}$ is a collection $\{\mathcal{V}(\lambda)\mid \lambda\in \Part\}$ where each $\mathcal{V}(\lambda)$ is a finite-dimensional rational $\mathbb{S}_\lambda$-representation. Equivalently, a $\bbS^{[2]}$-module is a $\QQ$-vector space $\calV$ together with a direct sum decomposition
    \[\calV = \bigoplus_{\lambda \in \Part} \calV(\lambda) \]
    where each $\calV(\lambda)$ is a finite-dimensional $\bbS_\lambda$-representation. We say $\calV$ is \textit{bounded} if $\calV$ is finite-dimensional. 
\end{defn}
 Let ${\bbS^{[2]}}\-\mathsf{Vect}$ be the category of $\bbS^{[2]}$-modules, and let $\bbS^{[2]}\-\cat{Vect}^\star$ be the subcategory of bounded $\bbS^{[2]}\-$modules. These categories each admit entry-wise direct sums $\bigoplus$. A symmetric monoidal structure on $\bbS^{[2]}\-\cat{Vect}$ (and on $\bbS^{[2]}\-\cat{Vect}^\star$) is defined by the box product: if $\calV$ and $\calW$ are $\bbS^{[2]}$-modules, define \[(\mathcal{V}\boxtimes \mathcal{W})(\lambda) := \bigoplus_{\nu\odot \mu=\lambda} \Ind_{\bbS_{\nu}\times \bbS_{\mu}}^{\bbS_{\lambda}}\mathcal{V}(\nu)\otimes \mathcal{W}(\mu).\]

\begin{rem}The direct sum decomposition $\WRing = \bigoplus_{\nu\in \Part}R(\bbS_{\nu})$ leads to a ring isomorphism between $\WRing$ and $K_0({\bbS^{[2]}}\-\mathsf{Vect}^\star)$.
\end{rem}

\begin{rem}
    The category $\mathsf{Part}$ may be identified with the groupoid of disconnected generalized corollas: to a generalized partition $\nu= (0^{\nu_0}1^{\nu_1}\dots )$ we associate a graph $\mathrm{Cor}_{\nu}$ that is a disjoint union of $\nu_i$ copies of an $i$-corolla, which is $i$ half-edges attached to a vertex. The group $\bbS_{\nu}$ agrees with the automorphism group of $\mathrm{Cor}_{\nu}$. The collection of connected corollas correspond to the partitions that have only a single part and have automorphism groups $S_n$.
\end{rem}


Extending the above remark, an $\mathbb{S}$-module is a functor from the groupoid of connected corollas to $\mathsf{Vect}_{\QQ}$ and that an $\mathbb{S}^{[2]}$-module is a functor from the groupoid of all corollas to $\mathsf{Vect}_{\QQ}$. The following formula is a representation-theoretic version of building corollas by taking disjoint unions of connected corollas.

To begin with, we note that for a $S_i$-representation $V$, the tensor power $V^{\otimes \nu_i}$ is an $S_i\wr S_{\nu_i}$-representation; given a map of $S_i$-representations $\varphi: V\to W,$ its $\nu_i$-th tensor power $\varphi^{\otimes \nu_i}: V^{\otimes \nu_i}\to W^{\otimes \nu_i}$ is a map of $S_i\wr S_{\nu_i}$-representations.

\begin{defn}
    Define $(\-)^{\otimes}: \mathbb{S}\-\mathsf{Vect}_{}\to \mathbb{S}^{[2]}\-\mathsf{Vect}_{}$ as the functor that sends an $\mathbb{S}$-module $\mathcal{W}$ to the $\mathbb{S}^{[2]}$-module $\mathcal{W}^{\otimes}$ defined by \[\mathcal{W}^{\otimes}(\nu):= \bigotimes_{i\geq 0} (\mathcal{W}(i))^{\otimes \nu_i},\] and sends a morphism $\boldsymbol{\varphi}: \mathcal{W}_1\to \mathcal{W}_2$, consisting of $\varphi_i: \mathcal{W}_1(i)\to \mathcal{W}_2(i),$ to the map $\boldsymbol{\varphi}^{\otimes }: \calW_1^\otimes \to \calW_2^\otimes$ given by $$\bigotimes_{i\geq 0}\varphi_i^{\otimes \nu_i}:  \mathcal{W}_1^{\otimes}(\nu)\to \mathcal{W}_2^{\otimes}(\nu)$$
    for each $\nu \in \Part$. For a fixed $\nu \in \Part$, we set \[\mathcal{W}^{\otimes \nu}:= \bigotimes_{i \geq 0}\left(\mathcal{W}(i)\right)^{\otimes \nu_i},\] which is an $\mathbb{S}_{\nu}$-representation.
\end{defn}

It is helpful to see the $\bbS^{[2]}$-module $\mathcal{W}^{\otimes \nu}$ as a monomial with exponent $\nu$ with the $\bbS$-module $\mathcal{W}$ being the variable. This will be made precise in the next section, in which we establish the correspondence between $\bbS^{[2]}$-modules and polynomial functors of $\bbS$-modules.

\subsection{Polynomial functors}




Basic facts on polynomial functors  between linear categories are recalled in Appendix \ref{sec:appendixPFun}. Write $\cat{Vect}$ for the category of finite-dimensional vector spaces over $\QQ$.

\begin{defn} Let $\mathcal{V}$ be a bounded $\mathbb{S}^{[2]}$-module. Define $P_{\mathcal{V}}$ as the functor from $\mathbb{S}\-\mathsf{Vect}$ to $\mathsf{Vect}_{}$ given by \[\mathcal{W}\mapsto \bigoplus_{\nu\in \Part}\mathcal{V}(\nu)\otimes_{\mathbb{S}_\nu} \mathcal{W}^{\otimes \nu}.\]\end{defn}


Let $G$ be a finite group. In \cite{MacDonald1980}, Macdonald proves that over a splitting field $k$ of $G$, the category of bounded $(G\wr \bbS)$-modules over $k$ is equivalent to the category of polynomial functors from $G\-\cat{Vect}_k$ to $\cat{Vect}_k$, where $G\-\cat{Vect}_k$ is the category of finite dimensional $G$-representations over $k$. In our setting, the analogue of Macdonald's theorem is as follows.

\begin{thm}\label{thm:s2polyfun}
    The functor $P_{\mathcal{V}}$ is polynomial. The assignment $\mathcal{V}\mapsto P_{\mathcal{V}}$ defines an equivalence between $\mathbb{S}^{[2]}$-$\mathsf{Vect}_{}^\star$ and the category $\mathbf{P}[\mathbb{S}\-\mathsf{Vect}_{}, \mathsf{Vect}_{}]$ of polynomial functors $\bbS\-\cat{Vect} \to \cat{Vect}_{}$.
\end{thm}
The proof of Theorem \ref{thm:s2polyfun} follows \cite{MacDonald1980} closely and is given in \ref{sec:pfpolyfun}. 

\begin{rem}
    We note that $\QQ$ is a splitting field for all symmetric groups, hence Macdonald's results concerning polynomial functors on the category of $S_i$-vector spaces work over $\QQ.$ The same carries over to polynomial functors over $\bbS$-modules. Therefore, $\WRing$ and $\bbS^{[2]}$-modules have been defined over $\QQ$ throughout.
\end{rem}

Theorem \ref{thm:s2polyfun} implies that there is a natural isomorphism
\[ \WRing \cong K_0(\mathbf{P}[\mathbb{S}\-\mathsf{Vect}_{}, \mathsf{Vect}_{}]). \]
Notice that there is a natural functor
\[ \mathbf{P}[\mathbb{S}\-\mathsf{Vect}_{}, \mathsf{Vect}_{}] \times \mathbf{P}[\mathsf{Vect}_{}, \bbS\-\mathsf{Vect}_{}] \to \mathbf{P}[\mathsf{Vect}_{}, \mathsf{Vect}_{}] \]
given by composition of functors, which upon taking Grothendieck rings will descend to a map
\[\WRing\times K_0(\mathbf{P}[\mathsf{Vect}_{}, \bbS\-\mathsf{Vect}_{}] )\to \Lambda. \]

The representation-theoretic avatars of polynomial functors $\cat{Vect} \to \bbS\-\cat{Vect}$ are given by \textit{$\mathbb{S}\times \mathbb{S}$-modules}. An $\bbS \times \bbS\-$module $\calV$ is the data of a finite-dimensional $(S_m \times S_n)$-representation $\calV(m, n)$ for all $m, n \in \Z_{\geq 0}$. The equivalence between $\bbS \times \bbS$-modules and polynomial functors $\mathsf{Vect}_{}$ to $\bbS\-\mathsf{Vect}_{}$ is given by a currying construction which extends the correspondence between $\bbS$-modules and polynomial functors from $\mathsf{Vect}_{}$ to $\mathsf{Vect}_{}.$ 

In our applications, $\bbS \times \bbS$-modules arise as the cohomology groups of moduli spaces with two classes of marked points, or equivalently legs, on their dual graphs: one class of legs are glued to form edges, and the other remain as `free' marked points.

\begin{defn}\label{defn:boundedSS}
    An $\mathbb{S}\times \mathbb{S}$-module $\mathcal{V}$ is \textit{bounded} if $\mathcal{V}(i,m)$ vanishes for all but finitely many $(i,m).$ It is \textit{$(1)$-bounded} if there exists $M \in \Z_{\geq 0}$ such that for $m > M$, $\calV(i,m) = 0$ for all $i.$    

    Let $\mathcal{V}$ be a $(1)$-bounded $\mathbb{S}\times \mathbb{S}$-module. Define $P_{\mathcal{V}}$ as the functor from $\mathsf{Vect}_{}$ to $\bbS\-\mathsf{Vect}_{}$ given by 
    \[ W \mapsto \bigoplus_{i \geq 0} \left( \bigoplus_{m \geq 0} \calV(i, m) \otimes_{S_m} W^{\otimes m} \right), \]
    where we are using the direct sum convention for $\bbS$-modules.
\end{defn}

\begin{rem}
    Visualizing an $\bbS\times\bbS$-module $\mathcal{V}$ as a grid of representations with $\mathcal{V}(i,m)$ at entry $(i,m),$ we see that $\mathcal{V}$ is (1)-bounded if and only if the grid has finitely many non-zero columns.
\end{rem}

\begin{rem}
    The Grothendieck ring of (1)-bounded $\bbS\times\bbS$-modules is isomorphic to $\hat{\Lambda}\otimes {\Lambda}.$ In comparison, the Grothendieck rings of $\bbS\times\bbS$-modules resp. bounded $\bbS\times\bbS$-modules are $\hat{\Lambda}\otimes \hat{\Lambda}$ and $\Lambda\otimes \Lambda.$
\end{rem}

\begin{prop}\label{prop:SSpoly}
If $\calV$ is a (1)-bounded $\bbS \times \bbS$-module, then $P_{\mathcal{V}}$ is a polynomial functor from $\mathsf{Vect}_{}$ to $\bbS\-\cat{Vect}$. The assignment $\mathcal{V}\mapsto P_{\mathcal{V}}$ induces an equivalence of categories between $(1)$-bounded $\mathbb{S}\times \mathbb{S}$-modules and $\mathbf{P}[\mathsf{Vect}_{}, \mathbb{S}\-\mathsf{Vect}_{}]$ as well as an equivalence of categories between bounded $\mathbb{S}\times \mathbb{S}$-modules and $\mathbf{P}[\mathsf{Vect}_{}, \mathbb{S}\-\mathsf{Vect}_{}^\star].$
\end{prop}
The proof is an explicit linearization that is given in \ref{sec:pfSSpoly}.


\subsection{The plethystic action of $\WRing$ on $\hat{\Lambda} \otimes {\Lambda}$}\label{sec:tcirc}

Using their interpretations as polynomial functors, we obtain an action of $\WRing$ on $\hat{\Lambda} \otimes {\Lambda}$.

\begin{defn}
    Let $\mathcal{V}$ be a (1)-bounded $\mathbb{S}\times \mathbb{S}$-module, and let $\mathcal{W}$ be a bounded $\mathbb{S}^{[2]}$-module. Abusing notation, identify them with polynomial functors $\mathcal{V}\in \mathbf{P}[\mathsf{Vect}_{}, \bbS\-\mathsf{Vect}_{}]$ and $\mathcal{W}\in \mathbf{P}[\bbS\-\mathsf{Vect}_{}, \mathsf{Vect}_{}]$. The enriched plethysm $\mathcal{W}\tcirc \mathcal{V}$ is the $\mathbb{S}$-module corresponding to the composition of polynomial functors $\mathsf{Vect}_{}\xrightarrow{\mathcal{V}} \bbS\-\mathsf{Vect}_{}\xrightarrow{\mathcal{W}}\mathsf{Vect}_{}.$ 
\end{defn}

\begin{defn}
    Let $\mathcal{V}$ be an $\bbS\times \bbS$-module. Define $\mathcal{V}_i\in (S_i\times\bbS)\-\mathsf{Vect}$ by $\mathcal{V}_i(m):=\mathcal{V}(i,m).$ The formula for box products on $\bbS$-modules applied on $S_i\times \mathbb{S}$-modules gives $\mathcal{V}_i^{\boxtimes \nu_i}$ as $(S_i\wr S_{\nu_i})\times \mathbb{S}$-modules.
    
\end{defn}

\begin{rem}
    When $\mathcal{V}$ is (1)-bounded, each $\mathcal{V}_i$ are bounded as $S_i\times \bbS$-modules.
\end{rem}

The notation serves to describe the $\mathbb{S}$-module under the composition operation.
\begin{lem}\label{lem:tcirc}
    Let $\mathcal{V}$ be a (1)-bounded $\mathbb{S}\times \mathbb{S}$-module, and let $\mathcal{W}$ be a bounded $\mathbb{S}^{[2]}$-module. Then $\mathcal{W}\tcirc \mathcal{V}$ is the $\mathbb{S}$-module given by the formula $$\bigoplus_{\nu\in \Part} \mathcal{W}(\nu)\otimes_{\bbS_{\nu}} \left(\midboxtimes_{i\geq 0}\mathcal{V}_i^{\boxtimes \nu_i}\right).$$
\end{lem}

The proof of Lemma \ref{lem:tcirc} is given in \ref{sec:pftcirc}.
\begin{rem}\label{rem:unbounded}
    The enriched plethysm described above in fact extends to all $\bbS\times\bbS$-modules. Let $\calV$ be a possibly unbounded $\bbS\times\bbS$-module, and let $\mathcal{W}$ be a bounded $\bbS^{[2]}$-module. The boundedness of $\mathcal{W}$ implies that the set $$\mathsf{Part}_{\mathcal{W}}:=\{\nu\in \Part\mid \mathcal{W}(\nu)\neq 0\}$$ is finite. Further, as each element in $\mathsf{Part}_{\mathcal{W}}$ has finitely many parts, the set $$\mathsf{I}_{\mathcal{W}}:=\{i\in \mathbb{Z}_{\geq 0}\mid \exists \nu\in \mathsf{Part}_{\mathcal{W}}: \nu_i\neq 0\}$$ is finite. Therefore, the formula in Lemma \ref{lem:tcirc} gives a well-defined $\bbS$-module $\mathcal{W}\tcirc \mathcal{V},$ which is possibly unbounded.
\end{rem}

  The composition of polynomial functors descends to Grothendieck rings.


\begin{defn}
    $\tilde{\circ}: \WRing\times (\hat{\Lambda}\otimes {\Lambda})\to {\Lambda}$ is defined as the unique map on Grothendieck groups $$\tilde{\circ}: K_0(\mathbf{P}[\bbS\-\mathsf{Vect}, \mathsf{Vect}])\times K_0(\mathbf{P}[\mathsf{Vect}, \bbS\-\mathsf{Vect}])\to K_0(\mathbf{P}[\mathsf{Vect}, \mathsf{Vect}])$$ that extends the assignment $\left([\mathcal{W}], [\mathcal{V}]\right)\mapsto [\mathcal{W}\tilde{\circ}\mathcal{V}].$ By remark \ref{rem:unbounded}, it extends to $\tcirc: \Lambda^{[2]}\times (\hat{\Lambda}\otimes \hat{\Lambda})\to \hat{\Lambda}$ that factors through the inclusion $\hat{\Lambda}\otimes {\Lambda}\hookrightarrow \hat{\Lambda}\otimes \hat{\Lambda}$ as described in the remark.
\end{defn}


The polynomial functor formalism and the $\bbS$-module formula presented above lead to the following concrete formulas for the enriched plethysm. Before stating them, we set some notation on a convenient basis for $\Lambda\otimes \hat{\Lambda}$. To begin with, we recall the notation $p_\mu = \prod_{i \geq 1} p_i^{\mu_i}$ for $\mu \in \Part^\star$, and that $\{p_{\mu}\mid \mu\in \Part^\star\}$ forms a linear basis for $\Lambda$.

\begin{defn}\label{defn:projmu}
    Let $\mu\in \Part^\star.$ Define $\mathrm{proj}_{\mu}: \hat{\Lambda}\otimes \hat{\Lambda}\to \hat{\Lambda}$ as the linear projection \[\hat{\Lambda} \otimes \hat{\Lambda} = \prod_{\mu \in \Part^\star} \frac{p_{\mu}}{\prod_{i \geq 1} i^{\mu_i}\mu_i!}\otimes \hat{\Lambda} \to \frac{p_{\mu}}{\prod_{i \geq 1} i^{\mu_i}\mu_i!} \otimes \hat{\Lambda} \xrightarrow{\cong} \hat{\Lambda}.\] Restricting the map to $\hat{\Lambda}\otimes \Lambda\subset \hat{\Lambda}\otimes \hat{\Lambda}$ gives $\mathrm{proj}_{\mu}: \hat{\Lambda}\otimes {\Lambda}\to {\Lambda}.$
\end{defn}


\begin{thm}\label{thm-tcirc}
    The map ${\tcirc}: \WRing\times (\hat{\Lambda}\otimes \hat{\Lambda})\to \hat{\Lambda}$ is characterized by the following properties.
    \begin{enumerate}[(1)]
        \item For all $q\in \hat{\Lambda}\otimes \hat{\Lambda},$ the map $\tcirc q: \WRing\to \hat{\Lambda}$ given by $f\mapsto f\tcirc q$ is an algebra homomorphism.
        \item For $\mu\in \Part^\star, q\in \hat{\Lambda}\otimes \hat{\Lambda}$, and $n \geq 0$, we have $$p_n(\mu)\tcirc q = \psi_n(\mathrm{proj}_{\mu}(q)).$$
    \end{enumerate}
    \end{thm}

\begin{proof}
    Item (1) follow from general facts about composing polynomial functors.

    For (2), by linearity it suffices to assume that $q = \ch(\mathcal{V}),$ where $\mathcal{V}$ is an $\bbS\times \bbS$-module given by $(m,n) \mapsto \mathcal{V}^{(1)}(m)\otimes \mathcal{V}^{(2)}(n)$ for $\bbS$-modules $\mathcal{V}^{(1)}, \mathcal{V}^{(2)}.$
    
    We recall that $$p_n(\mu)\in R(S_{|\mu|}\wr S_n) = R(\mathbb{S}_{(|\mu|^n)})$$ where $(|\mu|^n)$ denotes the partition with $n$ parts of size $|\mu|.$ Implementing the linearization formula in Lemma \ref{lem:tcirc}, we see that with $\mathcal{V} = \left(\bigoplus_{\ell\geq 0} \mathcal{V}^{(1)}(\ell)\right)\otimes \mathcal{V}^{(2)} ,$ the only $i\geq 0$ for which $\lambda_i\neq 0$ is $i = |\mu|,$ in which case $\lambda_i = n,$ and \[\mathcal{V}_{|\mu|} = \mathcal{V}^{(1)}(|\mu|)\otimes \mathcal{V}^{(2)}.\]

     Therefore, the formula for $p_n(\mu)\tcirc q$ simplifies to $$p_n(\mu)\otimes_{\mathbb{S}_{(|\mu|^n)}} \left(\mathcal{V}^{(1)}(|\mu|)\otimes \mathcal{V}^{(2)}\right)^{\boxtimes n},$$
     where $p_n(\mu)$ above is thought of as a virtual representation of $\bbS_{(|\mu|^n)}$, and $\left(\mathcal{V}^{(1)}(|\mu|)\otimes \mathcal{V}^{(2)}\right)^{\boxtimes n}$ is a $\bbS$-module with a $\bbS_{(|\mu|^n)}$-action.
    
    This agrees with the formula for the $S_{|\mu|}$-enriched plethysm $p_n(\mu)\circ_{|\mu|}(\mathcal{V}^{(1)}(|\mu|)\otimes \mathcal{V}^{(2)}).$ By Lemma \ref{lem-Gwrformulas}, its Frobenius character is equal to $\psi_n(\mathrm{proj}_\mu(q))$ as defined above.

    Since $\WRing$ is generated by all $p_n(\mu)$ as a polynomial ring, the map $\tcirc$ is uniquely determined by property (1).
\end{proof}



The $\bbS\times \bbS$-modules that will be relevant for our applications arise from $\bbS$-modules via the following construction.

\begin{defn}
    Let $\mathcal{V}$ be an $\bbS$-module. Define $\Delta\mathcal{V}$ as the $\bbS\times \bbS$-module defined by $\Delta\mathcal{V}(i,m) = \mathrm{Res}^{S_{i+m}}_{S_i\times S_m} \mathcal{V}(i+m).$
\end{defn}

\begin{rem}
    $\Delta\mathcal{V}$ is (1)-bounded if and only if $\mathcal{V}$ is bounded, if and only if $\Delta\mathcal{V}$ is bounded.
\end{rem}

The map $\tcirc \Delta,$ defined as the composition $ \WRing\times \hat{\Lambda}\xrightarrow{\mathrm{id}\times \Delta} \WRing\times (\hat{\Lambda}\otimes \hat{\Lambda})\xrightarrow{\tcirc} \hat{\Lambda},$ can be viewed as a plethystic action of $\WRing$ on $\hat{\Lambda}.$ It has the following appealing formula in terms of differential operators on $\hat{\Lambda}.$

\begin{defn}\cite[Example 5.3]{Macdonald}
    Let $f\in \Lambda$ be of bounded degree. Define\footnote{It is also known as the skewing operator in literature, in relation to Schur functions associated to skew partitions.} $D(f): \hat{\Lambda}\to \hat{\Lambda}$ as $$D(f):=f\left(\frac{\partial}{\partial p_1},2 \frac{\partial}{\partial p_2}, 3\frac{\partial}{\partial p_3},\dots \right).$$ Under the Hall inner product on $\Lambda$, $D(f)$ is the adjoint to the multiplication map $\cdot f: \hat{\Lambda}\to \hat{\Lambda}.$
\end{defn}

\begin{lem}\label{lem:Dpmu}
    Let $\mu\in \Part^\star$ and $f\in \hat{\Lambda},$ then $$p_n(\mu) \tcirc \Delta f = p_n \circ D(p_\mu)(f).$$
\end{lem}
\begin{proof}
    Let $m = |\mu|$ and let $\mathcal{V}$ be an $\bbS$-module. By linearity, it suffices to prove for $f = \ch(\mathcal{V}).$ In the following, we let $(\Delta\mathcal{V})_m$ denote the $S_m\times \bbS$-module given by $\mathrm{Res}^{\bbS\times\bbS}_{S_m\times \bbS}\Delta \mathcal{V}$ and let \[\langle -, -\rangle_{S_m}: \Lambda_{m}\otimes \Lambda_{m}\to \mathbb{Q}\] denote the inner product on class functions. Taking tensor product with $\hat{\Lambda},$ we extend this to $\langle -, -\rangle: \Lambda_{m}\otimes (\Lambda_{m}\otimes \hat{\Lambda})\to \hat{\Lambda}$; this is a restriction of the Hall inner product on $\hat{\Lambda}$. 
    
    Applying Theorem \ref{thm-tcirc},  $$p_n(\mu) \tcirc \Delta \ch(\mathcal{V}) = p_n\circ \langle p_{\mu},\ch(\Delta \mathcal{V})_m\rangle_{S_{m}} = p_n\circ \sum_{k\geq 0} \left\langle p_{\mu}, \ch\left(\Res^{S_{m+k}}_{S_m\times S_k} \mathcal{V}\right)\right\rangle.$$ 
    By \cite[Proposition 8.10]{GetzlerKapranov}, the right hand side is equal to $p_n\circ D(p_{\mu})(\ch(\mathcal{V})).$
\end{proof}

\begin{defn}\label{defn:partialbar}
    For a symmetric function $f \in \hat{\Lambda}$ and $\mu\in \Part^\star,$ we define \[ \overline{\partial}_\mu f := D(p_\mu) (f) = \left(\prod_{i \geq 1}  i^{\mu_i}\right) \cdot \frac{\partial^{\mu_1 + \mu_2 + \cdots} f}{\partial p_1^{\mu_1} \partial p_2^{\mu_2}\cdots}. \] Writing $\tcirc$ for $\tcirc \Delta$, the plethystic action $\WRing \times \hat{\Lambda} \to \hat{\Lambda}$ is determined by
\[ p_n(\mu) \tcirc f = \psi_n(\overline{\partial}_\mu f). \]
Extending multiplicatively, for $\tilde{\mu}\in \Part^\star,$ we set $\psi_{\tilde{\mu}}:=\prod_{i\geq 1} \psi_{i}^{\tilde{\mu}_i}.$ Then $p_{\tilde{\mu}}(\mu)\tcirc f=\psi_{\tilde{\mu}}(\overline{\partial}_\mu f).$
\end{defn}

\subsection{Graded modules and mixed Hodge structures}
All of the constructions of the previous section generalize in the following two ways:
\begin{itemize}
\item instead of considering $\bbS^{[2]}$, $\bbS \times \bbS$, and $\bbS$-modules in the category of finite dimensional rational vector spaces, we may consider them in the category $\cat{MHS}$ of mixed Hodge structures over $\QQ$;
\item we can in addition suppose that we are working with \textit{graded} mixed Hodge structures.
\end{itemize}
\begin{defn}
A graded $\bbS$-module $\calV^{\bullet}$ in $\cat{MHS}$ is a $\bbS$-module in $\cat{MHS}$ together with a direct sum decomposition
\[ \calV^{\bullet} = \bigoplus_{i \geq 0} \calV^{(i)} \]
where each $\calV^{(i)}$ is a $\bbS$-module in $\cat{MHS}$. Graded $\bbS^{[2]}$-modules and $(\bbS \times \bbS)$-modules in $\cat{MHS}$ are defined similarly.
\end{defn}
Note that box products and direct sums of graded $\bbS$-modules are also graded, in a natural way. For a graded $\bbS$-module $\calV^\bullet$ in $\cat{MHS}$ we define
\[ \ch_{\bbS}^{\cat{MHS}}(\calV^{\bullet}) := \sum_{i \geq 0}(-1)^{i}[\calV^{(i)}] \in K_0(\cat{MHS_{\bbS}}) \cong K_0(\cat{MHS}) \otimes \hat{\Lambda}. \]

Similary we can define \[\ch_{\bbS \times \bbS}^{\cat{MHS}}(\calV^{\bullet}) \in K_0(\cat{MHS}) \otimes \hat{\Lambda} \otimes \hat{\Lambda}\quad\mbox{and}\quad \ch_{\bbS^{[2]}}^{\cat{MHS}}(\calV^{\bullet}) \in K_0(\cat{MHS) \otimes}\WRing\] for graded $\bbS \times \bbS$-modules and bounded graded $\bbS^{[2]}$-modules, respectively. The action $\tcirc$ of Theorem \ref{thm-tcirc} extends to an action
\[\tcirc : K_0(\cat{MHS}) \otimes \WRing \times K_0(\cat{MHS}) \otimes \hat{\Lambda} \otimes \hat{\Lambda}  \]
in a standard way, since only the Adams operations need to be defined on $K_0(\cat{MHS})$. This extension is well-known, c.f. \cite{GetzlerPreprint, GetzlerPandharipande, BergstromTommasi}: first, given a mixed Hodge structure $M$, we can set
\begin{equation}\label{eqn:adams_mhs}
    \sum_{\mu \in \Part_n^\star} \frac{\psi_{\mu}([M])}{\prod_{i \geq 1} i^{\mu_i}\mu_i!} = [\mathrm{Sym}^nM], 
\end{equation}
and then the Adams operations $\psi_n([M])$ for $n \geq 1$ are determined by Newton's equations \cite[(5.6)]{GetzlerPandharipande}; in the notation of loc. cit., the left-hand side of (\ref{eqn:adams_mhs}) is denoted $s_n \circ [M]$. Analyzing the interface between signs and symmetric group actions on graded vector spaces as in \cite[Proposition 7.3]{GetzlerKapranov}, we have the following lemma.

\begin{lem}\label{lem:tcircMHS}
Suppose $\calW^{\bullet}$ is a bounded graded $\bbS^{[2]}$-module in $\cat{MHS}$ and $\calV^{\bullet}$ is a graded $(\bbS \times \bbS)$-module in $\cat{MHS}$. Then, with $\tcirc$ defined as in Lemma \ref{lem:tcirc}, we have
\[ \ch_{\bbS}^{\cat{MHS}}(\calW^{\bullet} \tcirc \calV^{\bullet}) = \ch_{\bbS^{[2]}}^{\cat{MHS}}(\calW^{\bullet}) \tcirc \ch_{\bbS \times \bbS}^{\cat{MHS}}(\calV^{\bullet}). \]
\end{lem}

\subsection{Enrichment to varieties}

The composition operation $\tcirc$ may be categorified to the level of algebraic varieties.
\begin{defn}
A $\bbS$-variety $\calX$ is a sequence $\calX(n)$ of varieties for each $n \geq 0$, together with a $S_n$-action on each $\calX(n)$. We may define $\bbS^{[2]}$ and $(\bbS \times \bbS)$-varieties analogously. 
\end{defn}

For brevity, let $\mathscr{S}$ denote any of $\bbS, \bbS\times\bbS,$ or $\bbS^{[2]}$. The Grothendieck group of $\mathscr{S}$-varieties $K_0(\mathsf{Var}_\mathscr{S})$ is defined defined as the quotient of free abelian group generated by $\mathscr{S}$-varieties by equivariant cut-and-paste relations. For $\mathscr{S} = \bbS$ or $\mathscr{S} = \bbS^{[2]}$, the group $K_0(\mathsf{Var}_\mathscr{S})$ is made into a ring using the box product: if $\calX$ and $\calY$ are $\bbS$-varieties, we set
\[ (\calX \boxtimes \calY)(n) := \coprod_{m = 0}^{n} \Ind_{S_m \times S_n}^{S_{m + n}} \calX(m) \times \calY(m - n); \]
if they are instead $\bbS^{[2]}$-varieties, we set
\[ (\calX \boxtimes \calY)(\nu) := \coprod_{\mu \odot \lambda = \nu}  \Ind_{\bbS_{\mu} \times \bbS_{\lambda}}^{\bbS_{\nu}} \calX(\mu) \times \calY(\lambda).\]

    
    Now suppose that $\calY$ is an arbitrary $\bbS \times \bbS$-variety, and assume that $\calX$ is a \textit{bounded} ${\bbS^{[2]}}$-variety, in the sense that $\calX(\nu) = \varnothing$ for all but finitely many $\nu$. Put $\calY_i$ for the $(S_i \times\bbS)$-variety
    \[\calY_i(n) := \calY(i,n). \]
    With this notation, we may define a $\bbS$-variety $\mathcal{X}\tcirc \mathcal{Y}$ by \begin{equation}\label{eqn:variety_composition}
        \mathcal{X}\tcirc \mathcal{Y}:=\bigsqcup_{\nu\in \Part}\left(\mathcal{X}(\nu)\times \midboxtimes_{\hspace{1pt}i\geq 0} \mathcal{Y}_i^{\boxtimes \nu_i}\right)\big/{\bbS_{\nu}}.
    \end{equation}
    This is in parallel to the plethysm operation on the level of modules, with direct sums replaced by disjoint unions, tensor products by products, and $\bbS_\nu$-invariants by $\bbS_\nu$-quotients. Altogether, (\ref{eqn:variety_composition}) defines a composition operation \[\tilde{\circ}: K_0(\mathsf{Var}_{\mathbb{S}^{[2]}}^\star)\times  K_0(\mathsf{Var}_{\mathbb{S}\times \bbS})\to K_0(\mathsf{Var}_{\mathbb{S}}),\]
    where we have set $\cat{Var}_{\bbS^{[2]}}^\star$ for the category of bounded $\bbS^{[2]}$-varieties.

    The compactly supported cohomology groups of $\mathscr{S}$-varieties are mixed Hodge structures that carry $\mathscr{S}$-actions, so they are graded $\mathscr{S}$-modules in $\cat{MHS}$.
    \begin{defn}
    We define the equivariant Serre characteristics as follows.
    \begin{enumerate}
    \item If $\calX$ is an arbitrary $\bbS$-variety, define a graded $\bbS$-module $\calU^\bullet_\calX$ in $\cat{MHS}$ by
    \[ \calU^{(i)}_\calX(n) := H^i_c(\calX(n);\QQ), \]
    and define the $\bbS$-equivariant Serre characteristic $\cat{e}^{\bbS}(\calX) \in K_0(\cat{MHS}) \otimes \hat{\Lambda}$ by
    \[ \cat{e}^{\bbS}(\calX) := \ch^{\cat{MHS}}_{\bbS}(\calU^{\bullet}_{\calX}). \]
    \item If $\calX$ is an arbitrary $\bbS \times \bbS$-variety, define a graded $\bbS \times \bbS$-module $\calV_{\calX}^{\bullet}$ by
    \[\calV^{(i)}_{\calX}(m, n) := H^i_c(\calX(m,n);\QQ), \]
    and define $\cat{e}^{\bbS \times \bbS}(\calX) \in K_0(\cat{MHS}) \otimes \hat{\Lambda} \otimes \hat{\Lambda}$ by
    \[ \cat{e}^{\bbS \times \bbS}(\calX) := \ch^{\cat{MHS}}_{\bbS \times \bbS}(\calV^{\bullet}_{\calX}). \]
    \item If If $\calX$ is a bounded $\bbS^{[2]}$-variety, define a bounded graded $\bbS^{[2]}$-module $\calW_{\calX}^{\bullet}$ by
    \[\calW^{(i)}_{\calX}(\lambda) := H^i_c(\calX(\lambda);\QQ), \]
    and define $\cat{e}^{\bbS^{[2]}}(\calX) \in K_0(\cat{MHS}) \otimes \WRing$ by
    \[ \cat{e}^{\bbS \times \bbS}(\calX) := \ch^{\cat{MHS}}_{\bbS^{[2]}}(\calW^{\bullet}_{\calX}). \]

    \end{enumerate}
    \end{defn}
    For each $\mathscr{S}$, the equivariant Serre characteristic $\cat{e}^{\mathscr{S}}$ is a group homomorphism; via the K\"unneth formula, it is also a ring homomorphism for $\mathscr{S} = \bbS$ and $\mathscr{S}= \bbS^{[2]}$. From Lemma \ref{lem:tcircMHS}, we have that the composition operations on the level of varieties and mixed Hodge structures are compatible.

\begin{thm}\label{thm:tcirc_varieties}
For a bounded $\bbS^{[2]}$-variety $\calX$ and an $\bbS\times \bbS$-variety $\calY$, we have
\[\mathsf{e}^{\bbS}(\calX \tcirc \calY) = \mathsf{e}^{\bbS^{[2]}}(\calX) \tcirc \mathsf{e}^{\bbS \times \bbS}(\calY).\]
\end{thm}

\subsection{Genus gradings}
We introduce genus gradings on $\mathscr{S}$-varieties so that we treat moduli spaces of curves of various genera on the same footing.

\begin{defn}
Suppose $\mathscr{S} = \bbS$ or $\mathscr{S} = \bbS \times \bbS$. A \textit{genus-graded} $\mathscr{S}$-variety $\calX$ is a sequence $\calX(g, \-)$ of $\mathscr{S}$-varieties defined for each $g \geq 0$. We write $\calX_g(\-)$ for the $\mathscr{S}$-variety defined by $\calX_{(g)}(\-) = \calX(g, \-)$.
\end{defn}

Suppose given a genus-graded $\bbS$-variety $\calX$. We define the genus-graded Serre characteristic by
\begin{equation}\label{eqn:genus_graded_serre}
\cat{e}_t^{\bbS}(\mathcal{X}) = \sum_{g\geq 0} \cat{e}^{\bbS}(\mathcal{X}_{(g)})t^{g-1} \in K_0(\cat{MHS)} \otimes \hat{\Lambda}(\!(t)\!).
\end{equation}
The genus-graded Serre characteristic of a genus-graded $\bbS \times \bbS$-variety is defined similarly, and takes value in $K_0(\cat{MHS}) \otimes (\hat{\Lambda} \otimes \hat{\Lambda)}(\!(t)\!)$.

\subsection{Graded plethystic action}
For our formulas we need to upgrade (\ref{eqn:variety_composition}) to the case where $\calX$ is an (ungraded, bounded) $\bbS^{[2]}$-variety, and $\calY$ is a genus-graded $\bbS \times \bbS$-variety. For this, we first need to upgrade box products to the genus-graded setting.

\begin{defn}
   If $\calX$ and $\calY$ are genus-graded $\bbS$-varieties, define their box product by
   \[(\mathcal{X}\boxtimes \mathcal{Y})(g, n) := \coprod_{h = 0}^{g} \coprod_{m = 0}^{n} \Ind_{S_m \times S_{n - m}}^{S_n} \mathcal{X}(h, m)\times \mathcal{Y}(g - h, n - m).\]
\end{defn}

Now if $\calY$ is a genus-graded $\bbS \times \bbS$-variety, define a genus-graded $S_i \times \bbS$-variety $\calY_i$ by
\[ \calY_i(g, n) := \calY(g, i, n). \]
Now we define, for a bounded $\bbS^{[2]}$-variety $\calX$ and a genus-graded $\bbS \times \bbS$-variety $\calY$ a genus-graded $\bbS$-variety $\calX \tcirc \calY$ by
\begin{equation}\label{eqn:graded_tcirc_varieties}
\calX \tcirc \calY := \coprod_{\nu \in \Part} \left(\calX(\nu) \times \midboxtimes_{i \geq 0} \calY_i^{\boxtimes \nu_i}\right)/\bbS_{\nu}.
\end{equation}
We now describe the corresponding action on the level of symmetric functions.
\begin{defn}\label{defn:graded_plethysm}
The action $\tcirc$ of \[\Lambda^{[2]} \times \hat{\Lambda} \otimes \hat{\Lambda} \to \hat{\Lambda}\] induces an action 
\[ \tcirc: \WRing \times (\hat{\Lambda} \otimes \hat{\Lambda})(\!( t)\!) \to \hat{\Lambda} (\!( t)\!), \]
by setting
\[ p_n(\mu) \tcirc (f \otimes g) t = \left[p_n(\mu) \tcirc (f \otimes g) \right]\cdot  t^n \]
for all $n \geq 0$ and $\mu \in \Part^\star$, and requiring $\tcirc{g}: \Lambda^{[2]} \to \hat{\Lambda}(\!(t)\!),$ $f\mapsto f\tcirc g$ to be an algebra homomorphism for all $g\in (\hat{\Lambda}\otimes \hat{\Lambda})(\!(t)\!).$
\end{defn}

The action defined in Definition \ref{defn:graded_plethysm} extends to
\[\tcirc:K_0(\cat{MHS}) \otimes \WRing \times K_0(\cat{MHS}) \otimes (\hat{\Lambda} \otimes \hat{\Lambda})(\!(t)\!) \to K_0(\cat{MHS}) \otimes \hat{\Lambda}(\!(t)\!) \]
in a natural way. With this extension, a natural generalization of Theorem \ref{thm:tcirc_varieties} holds, using the composition operation in (\ref{eqn:graded_tcirc_varieties}) and accounting for the degree shift in (\ref{eqn:genus_graded_serre}): if $\calX$ is a bounded $\bbS^{[2]}$-variety and $\calY$ is a genus-graded $\bbS^{[2]}$-variety, we have
\begin{equation}\label{eqn:weird_degrees}
    t\cdot\cat{e}_t^{\bbS}(\calX \tcirc \calY) =  \cat{e}^{\bbS^{[2]}}(\calX) \tcirc \left(t\cdot\cat{e}_t^{\bbS \times \bbS}(\calY)\right).
\end{equation} 

\begin{lem}\label{lem:degree_shift}
    Suppose that $\calX$ is a $\bbS^{[2]}$-variety supported in the single degree $\nu$, and suppose that $\calY$ is an arbitrary genus-graded $\bbS \times \bbS$-variety. Put $\calZ$ for the genus-graded $\bbS$-variety defined by \[\calZ := \calX \tcirc \calY,\] and set $V := \sum_{i \geq 0} \nu_i$. Then
    \[ \cat{e}^{\bbS^{[2]}}(\calX) \tcirc \cat{e}_t^{\bbS \times \bbS} (\calY) =  t^{1 - V} e_{t}^{\bbS}(\calZ). \]
\end{lem}
\begin{proof}
When $\calX$ is supported in the single degree $\nu$, we have
\[\cat{e}^{\bbS^{[2]}}(\calX) \tcirc \left(t\cdot\cat{e}_t^{\bbS \times \bbS}(\calY)\right) = t^V\left(\cat{e}^{\bbS^{[2]}}(\calX) \tcirc \cat{e}_t^{\bbS \times \bbS}(\calY)\right), \]
so the claim follows by dividing both sides of (\ref{eqn:weird_degrees}) by $t^{V}$.
\end{proof}


\begin{rem}
    The appearance of of the $t^{1-V}$ factor in Lemma \ref{lem:degree_shift} will eventually account for Euler characteristics of graphs in \S\ref{sec:ppchar}; similar accounting is done in \cite[Prop 8.5]{GetzlerKapranov}, where a factor $\hbar^{-n}$ appears in front of the genus-graded $\bbS$-module $\Exp_{n}(\mathcal{V})$ associated to taking disjoint unions of $n$ vertices from $\mathcal{V}$.
\end{rem}

We conclude this section by setting some notation for genus-degree shifts of genus-graded $\bbS$-varieties.

\begin{defn}
    Let $\mathcal{X}$ be a genus-graded $\bbS$-variety. For $d\in \mathbb{Z},$ define a genus-graded $\bbS$-variety $\calX[d]$ by \[\mathcal{X}[d](g,n):=\mathcal{V}(g-d,n).\] 
    We have that
    $\mathsf{e}_t^{\bbS}(\mathcal{X}[d]) = t^d \mathsf{e}_t^{\mathscr{S}}(\mathcal{X}).$
\end{defn}

\section{Graph stratifications for moduli spaces of stable curves}\label{sec:graph_strat}
For each integer $g \geq 0$, define an $\bbS$-variety $\Mbar_g$ by setting
\[ \Mbar_{g}(n) = \begin{cases}
\varnothing &\text{if }2g-2 + n \leq 0, \\
\Mbar_{g, n} &\text{if }2g - 2 + n> 0
\end{cases}.\]
We set $\M_g \subset \Mbar_g$ for the $\bbS$-subspace parameterizing smooth curves. Also define a graded $\bbS$-variety $\Mbar$ by setting
\[\Mbar_{(g)}(n) = \Mbar_{g, n},\] 
and write $\M \subset \Mbar$ for the graded $\bbS$-subspace of smooth curves. In this section we will recall well-known graph stratifications of $\Mbar_g$ and $\Mbar$.

\subsection{Graph preliminaries} The following definition of a graph is well-suited for working with moduli spaces of curves.
\begin{defn}
A \textit{graph} $G$ is a tuple
\[ G=  (V(G), H(G), r_G, \iota_G), \]
where
\begin{itemize}
\item $V(G)$ is a finite set of \textit{vertices} and $H(G)$ is a finite set of \textit{half-edges};
\item $r_G : H(G) \to V(G)$ is a map of sets, called the \textit{root map};
\item $\iota_G : H(G) \to H(G)$ is a fixed-point free involution.
\end{itemize}
We obtain an equivalence relation on $H_G$ by setting $h \sim \iota_G(h)$ for all $h \in H(G)$; the elements of the set $E_G$ of equivalence classes are the \textit{edges} of $G$. A graph $G$ has a \textit{geometric realization} which is a topological space. We say $G$ is \textit{connected} if its geometric realization is connected.
\end{defn}

Let $\mathsf{Graph}$ be the groupoid of graphs, where an isomorphism \[f: G \to G'\] is a pair of bijections
\[f_H: H(G) \to H(G')  \]
 and
 \[f_V: V(G)\to V(G') \]
 such that $\iota_{G'} \circ f_H = f_H \circ \iota_G$ and $r_{G'} \circ f_H = f_V \circ r_G$. For any graph $G$ we have a canonical inclusion of groups \[\Aut_\mathsf{Graph}(G) \hookrightarrow \Aut_{\mathsf{Map}}( r_G: H(G) \to V(G)), \]
 and we identify $\Aut_{\mathsf{Graph}}(G)$ with its image in $\Aut_{\mathsf{Map}}( r_G: H(G) \to V(G))$. 
 We define the \textit{valence} of a vertex $v \in V(G)$ to be
 \[\val(v) := |r_G^{-1}(v)|. \]
From the graph $G$ we define the partition $\nu(G) \in \Part$ by setting
\[ \nu(G)_i := |\{v \in V(G) \mid \val(v) = i \}| \]
 If, for each $i$, we choose
 \begin{itemize}
\item an ordering of the vertices $v \in V(G)$ with $\val(v) = i$
\item for each vertex $v$ of valence $i$, an ordering of $r_G^{-1}(v)$,
 \end{itemize}
 we obtain an isomorphism

\begin{equation}\label{eqn:aut_embedding}\psi_G:\Aut_{\mathsf{Map}}(r_G) \xrightarrow[\phantom{long}]{\sim} \bbS_{\nu(G)}.
\end{equation}
 Different choices of orderings lead to isomorphisms which differ by conjugation in  $\bbS_{\nu(G)}$. Altogether, this realizes $\Aut(G)$ as a subgroup of $\bbS_{\nu(G)}$, up to conjugacy.

\subsection{Vertex-decorated graphs}
We will also need graphs with genus decorations at the vertices.
\begin{defn}
Let $\cat{VGraph}$ to be the groupoid of pairs $(G, w)$ where $G$ is a finite connected graph, and $w: V(G) \to \Z_{\geq 0}$ is any function. We write \[\cat{VGraph}_g \subset \cat{VGraph} \]
for the full subgroupoid of vertex-weighted graphs $\G = (G, w)$ for which \[g = \sum_{v \in V(\G)} w(v) + |E(G)| - |V(G)| + 1  ;\] if $\G \in \mathrm{Ob}(\cat{VGraph}_g)$, we say $\G$ has \textit{genus} equal to $g$. Objects in $\cat{VGraph}$ are called vertex-weighted graphs.
\end{defn}

\begin{defn}
    Let $\G = (G, w) \in \mathrm{Ob}(\cat{VGraph}_g)$ be a vertex-weighted graph of genus $g$ with root map $r_{G}: H(G)\to V(G).$ The \textit{fine graph partition} $\underline{\nu}(\G)$ of $\G$ is defined as a map $\{0,\dots, g\}\to \Part$ given by $$h\mapsto \underline{\nu}(\G)^{(h)}\in \Part, \text{ with }\underline{\nu}(\G)^{(h)}_i:=\#\{v\in V(G)\mid \mathrm{val}(v) = i, w(v) = h\}.$$
\end{defn}

Let $\nu(G)\in \Part$ be the partition associated to the underlying graph of $\G,$ namely its image under the forgetful functor $\mathsf{VGraph}\to \mathsf{Graph}.$ Then the tuple of partitions $\underline{\nu}(\G)$ refines $\nu(G)$ in the sense that $\underline{\nu}(\G)^{(0)}\odot \cdots \odot \underline{\nu}(\G)^{(g)} = \nu(G).$

\begin{defn}
    The automorphism group of a vertex weighted graph $\G = (G, w) \in \mathrm{Ob}(\cat{VGraph})$ is the subgroup preserving the vertex weighting: $$\mathrm{Aut}_{\cat{VGraph}}(\G):=\{\sigma\in \mathrm{Aut}_{\cat{Graph}}(\G)\mid w(v) = w(\sigma(v)) \text{ for all } v\in  V(G_C)\, \}.$$
\end{defn}

We always use the boldface notation $\G$ for vertex-weighted graph, and we write $\Aut(\G)$ for $\Aut_{\cat{VGraph}}(\G)$. 

For $\G = (G, w) \in \mathrm{Ob}(\cat{VGraph}_g)$, there is a canonical way to factor the embedding \[\Aut(\G) \hookrightarrow \bbS_{\nu(G)}\]
in a way that accounts for the vertex weights. 

\begin{lem}\label{lem:refsubgrp}
    Define \[\bbS_{\underline{\nu}(\G)}:=\prod_{h=0}^g \bbS_{\underline{\nu}(\G)^{(h)}} = \prod_{h=0}^g\left(\prod_{i\geq 0} S_i\wr S_{\underline{\nu}(\G)^{(h)}_i}\right).\] Then the group embedding \[\mathrm{Aut}(\G)\hookrightarrow \bbS_{\nu(G)}\] factors through $$\mathrm{Aut}(\G)\hookrightarrow \bbS_{\underline{\nu}(\G)} \hookrightarrow \prod_{i\geq 0} S_i\wr S_{\sum_{h=0}^g \underline{\nu}(\G)^{(h)}_i} = \bbS_{\nu(G)}.$$ 
\end{lem}

Recall from Section \ref{sec:ringwring} that conjugacy classes in $\bbS_{\tup{\nu}(\G)}$ are classified by $(g+1)$-tuples of $2$-partitions $(\Theta_0, \ldots, \Theta_g)$ such that $\Theta_h \in \Partt_{\tup{\nu}(\G)^{(h)}}$ for each $h$.

 \subsection{The coarse stratification of $\Mbar$}

 We now describe a stratification of the graded $\bbS$-variety $\Mbar$ where the strata are indexed by isomorphism classes of connected graphs.

 \begin{defn}
      The \textit{coarse dual graph} $G_C \in \mathrm{Ob}(\Graph)$ of $C$ is the graph with vertices as the set of irreducible components of the normalization of $C$ and half-edges as the preimages of nodes in the normalization of $C$. The involution $\iota_{G_C}$ on half-edges identifies the two branches at a node of $C$, and the map $r_{G_C}$ assigns a branch of a node to the irreducible component containing it.
 \end{defn}

 Given a connected graph $G$, define a graded $\bbS$-variety $\Mbar_G$ by setting
 \[( \Mbar_G)_{(g)}(n) \subset \Mbar_{(g)}(n) \]
 to be the locus of curves whose coarse dual graph is equal to $G$. Then we have a $S_n$-equivariant stratification
 \begin{equation}\label{eqn:coarse_stratification}
 \Mbar_{(g)}(n) = \coprod_{G \in \mathrm{Iso}(\cat{Graph})} (\Mbar_G)_{(g)}(n);
 \end{equation}
 note that for fixed $g, n$, only finitely many strata on the right-hand side of (\ref{eqn:coarse_stratification}) are nonempty. We can describe the graded $\bbS$-variety $\Mbar_G$ in terms of $\M$ as follows. For any graded $\bbS$-variety $\calX$ and integer $m \geq 0$, write $\Res^{\bbS}_{S_{m \times \bbS}} \calX$  for the $S_m$-object in the category of graded $\bbS$-varieties determined by
\[\left(\Res^{\bbS}_{S_{m \times \bbS}} \calX\right)(n) = \Res^{S_{m + n}}_{S_m \times S_n} \calX(m + n). \]
Observe that for any integer $k \geq 0$, the graded $\bbS$-variety
\[\left(\Res^{\bbS}_{S_{m \times \bbS}} \calX\right)^{\boxtimes k}\]
admits an $(S_m \wr S_k)$-action, and if $\lambda \in \Part$, then
\[ \prod_{i \geq 0} \left(\Res_{S_i \times \bbS}^{\bbS} \calX\right)^{\boxtimes \lambda_i} \]
admits a $\bbS_\lambda$-action. Therefore, for any connected graph $G$, the product
\[ \midboxtimes_{v \in V(G)} \Res^{\bbS}_{S_{\val(v) \times \bbS}} \calX\] admits a $\bbS_{\nu(G)}$-action. Returning to moduli of stable curves, we see that there is an isomorphism of graded $\bbS$-varieties
\begin{equation}\label{eqn:coarse_strata}\Mbar_G \cong \left[\left(\midboxtimes_{v \in V(G)} \Res^{\bbS}_{S_{\val(v) \times \bbS}} \M \right) /\Aut(G)\right][1-\chi(G)],
 \end{equation}
 where $\Aut(G)$ acts on the product as a subgroup of $\bbS_{\nu(G)}$. The degree shift accounts for the first Betti number of the graph $G.$
 
 \subsection{The fine dual graph stratification}
 To give a stratification of the $\bbS$-variety $\Mbar_g$, we will use a finer notion of a dual graph.
 \begin{defn}
     Let $C$ be a connected nodal curve of arithmetic genus $g$. The \textit{fine dual graph} $\G_C \in \mathrm{Ob}(\cat{VGraph}_g)$ of $C$ is the pair $\G_C =(G_C, w_C)$ where $G_C$ is the coarse dual graph of $C$ and $w_C: V(G_C) \to \Z_{\geq 0}$ assigns each irreducible component of the normalization its geometric genus. 
 \end{defn}


Given $\G \in \mathrm{Ob}(\cat{VGraph}_g)$, define an $\bbS$-variety $\Mbar_\G$, where
\[ \Mbar_\G(n) \subset \Mbar_{g}(n) \]
is the locus of stable curves whose fine dual graph is equal to $\G$. We thus have an $S_n$-equivariant stratification
\begin{equation}\label{eqn:vgraph_strat}\Mbar_{g}(n) = \coprod_{\G \in \mathrm{Iso}(\cat{VGraph}_g)} \Mbar_{\G}(n);
\end{equation}
analogously to (\ref{eqn:coarse_stratification}), only finitely many of the strata appearing on the right-hand side of (\ref{eqn:vgraph_strat}) are nonempty for fixed $n$. A fine stratum corresponding to $\G = (G, w)$ satisfies
\begin{equation}\label{eqn:fine_strata} \Mbar_\G \cong \left( \midboxtimes_{v \in V(G)} \Res^{\bbS}_{S_{\val(v)} \times \bbS} \calM_{w(v)} \right)/\Aut(\G),
\end{equation}
where $\Aut(\G)$ acts on the product as a subgroup of $\bbS_{\underline{\nu}(\G)}$. 
\subsection{Caterpillars and rational tails}\label{sec:catails}
The fine graph stratification (\ref{eqn:vgraph_strat}) has infinitely many strata. There are two sources of infinitude in $\cat{VGraph}_g$, as we will now describe. Given $\G \in \cat{VGraph}_g$, define the \textit{core} of $\G$ to be the minimal vertex-weighted subgraph $\G^{\mathrm{core}} \subset \G$ such that the genus of $\G^{\mathrm{core}}$ is equal to $g$. Then the two sources of infinitude in $\mathsf{VGraph}_g$ are as follows:
\begin{itemize}
\item arbitrarily large trees of weight zero vertices can be attached to the core of the graph;
\item if $e \in E(\G)$ is an edge, then $e$ can be subdivided arbitrarily many times by weight zero vertices.
\end{itemize}
We would like finite stratifications where possible, because they will lead to finite formulas for the Serre characteristic $\mathsf{e}(\Mbar_g)$. Define the $\bbS$-subspace $\Mbar_{g}^{\mathrm{nrt}}$ of $\Mbar_g$ by setting
\[\Mbar_{g}^{\mathrm{nrt}}(n) \subset \Mbar_g(n) \]
to be the subspace of curves whose dual graphs are equal to their own cores. The following formula is standard in the theory of stable curves.
\begin{lem}\label{lem:attaching_rat_tails}
For $g \geq 1$, we have
\[ \mathsf{e}^{\bbS}(\Mbar_g) = \mathsf{e}^{\bbS}(\Mbar_g^{\mathrm{nrt}}) \circ \left(p_1 + \frac{\partial \mathsf{e}^{\bbS}(\Mbar_0)}{\partial p_1}\right). \]
\end{lem}
Modulo the computation of $\mathsf{e}^{\bbS}(\Mbar_0)$, Lemma \ref{lem:attaching_rat_tails} reduces the calculation of $\mathsf{e}^{\bbS}(\Mbar_g)$ to that of $\mathsf{e}^{\bbS}(\Mbar_g^{\mathrm{nrt}})$, and in a sense eliminates the first source of infinitude discussed above. The characteristic $\mathsf{e}^{\bbS}(\Mbar_1^{\mathrm{nrt}})$ was computed by Getzler \cite{GetzlerSemiClassical} and then by Petersen \cite{semiclassicalremark}. Petersen's calculation fits into the framework developed herein, but for ease of exposition we will assume that $g \geq 2$ for the remainder of this section. To stratify $\mathsf{e}(\Mbar_g^{\mathrm{nrt}})$, we introduce some natural subcategories of $\mathsf{VGraph}_g$.
\begin{defn}\label{def:bary_stable_graphs}
Given a graph $\G \in \mathrm{Ob}(\mathsf{VGraph}_g)$, we say $\G$ is \textit{stable} if ${2w(v) - 2 + \val(v)> 0}$ for all vertices $v \in V(\G)$. Let 
\[\Gamma_g \subset \mathsf{VGraph}_g\]
be the full subgroupoid of stable graphs. Given $\G \in \mathrm{Ob}(\mathsf{VGraph}_g^{\mathrm{stab}})$, write $\hat{\G} \in \mathrm{Ob}(\cat{VGraph}_g)$ for the graph obtained from $\G$ by subdividing each edge $e \in E(\G)$ by a single weight zero vertex. Let
\[\hat{\Gamma}_g\subset \mathsf{VGraph}_g\]
be the full subgroupoid consisting of objects of the form $\hat{\G}$ for $\G \in \mathrm{Ob}(\Gamma_g)$. 
\end{defn}
The category $\hat{\Gamma}_g$ can be thought of as the groupoid of barycentrically subdivided stable graphs. In particular, we have
\[ |\mathrm{Iso}(\hat{\Gamma}_g)| = |\mathrm{Iso}(\Gamma_g)| < \infty, \]
and we will use $\hat{\Gamma}_g$ to give a finite stratification of the $\bbS$-variety $\Mbar_g^{\mathrm{nrt}}$.

\begin{defn}
Given a connected nodal curve $C$ of genus $g \geq 2$ such that $\G_C = \G_C^{\mathrm{core}}$, we define \[\mathfrak{G}_C \in \mathrm{Ob}(\hat{\Gamma}_g) \]
in two steps:
\begin{enumerate}
\item Take $\G_C$ and smooth out all weight zero vertices of valence $2$, i.e. treat them as part of the underlying edge (the resulting graph is sometimes called the \textit{stabilization} of $\G_C$);
\item Subdivide each edge in the resulting graph exactly once, with a vertex of weight zero.
\end{enumerate}
We call $\mathfrak{G}_C$ the \textit{reduced fine dual graph} of $C$.
\end{defn}

Given $\mathfrak{G} \in \mathrm{Ob}(\hat{\Gamma}_g)$, define an $\bbS$-variety
\[ \Mbar_\mathfrak{G}^{\mathrm{nrt}}\subset \Mbar_g^{\mathrm{nrt}} \]
by letting \[\Mbar_\mathfrak{G}^{\mathrm{nrt}}(n) \subset \Mbar_{g}^{\mathrm{nrt}}(n)\]
denote the locus of curves with reduced fine dual graph equal to $\mathfrak{G}$; then
\begin{equation}\label{eqn:reduced_stratification}
\Mbar_g^{\mathrm{nrt}} = \coprod_{\mathfrak{G} \in \mathrm{Iso}(\hat{\Gamma}_g)} \Mbar_{\mathfrak{G}}^{\mathrm{nrt}},
\end{equation}
so we get a finite stratification of $\Mbar_g^{\mathrm{nrt}}$. Now define a $(S_2 \times \bbS)$-variety $\Cat$ as follows: for $n \geq 1$, let
\[ \Cat(n) \subset \Res_{S_2 \times S_n}^{S_{n + 2}} \Mbar_{0, n+2} \]
be the locus of stable genus zero curves whose coarse dual graph is a path (this includes the case of a single isolated vertex), and such that the first two markings are on the two extreme components of the curve. Then $S_2$ acts on $\Cat(n)$ by permuting the first two marked points, and $S_n$ acts by permuting the other $n$ markings. We also formally set $\Cat(0) = \Spec \C$, with the trivial $(S_2 \times S_0)$-action. The utility of this definition is that it yields the following description of the stratum $\Mbar_{\mathfrak{G}}^{\mathrm{nrt}}$. Let $V^{b}(\mathfrak{G})$ be the set of vertices for $w(v) = 0$ and $\val(v) = 2$, and set 
\[ V^{\mathrm{stab}}(\mathfrak{G}) = V(\mathfrak{G}) \smallsetminus V^{b}(\mathfrak{G}) \]
\begin{equation}\label{eqn:reduced_stratum}
    \Mbar_{\mathfrak{G}}^{\mathrm{nrt}} \cong \left[\left(\midboxtimes_{v \in V^{b}(\mathfrak{G})} \Cat\right) \boxtimes\left( \midboxtimes_{v \in V^{\mathrm{stab}}(\mathfrak{G})} \Res^{\bbS}_{S_{\val(v)} \times \bbS} \M_{w(v)} \right)\right] /\Aut(\mathfrak{G}),
\end{equation}
where $\Aut(\mathfrak{G})$ acts as a subgroup of $\bbS_{\underline{\nu}(\mathfrak{G})}$. The product formula (\ref{eqn:reduced_stratum}) reflects the fact that any stable curve in $\Mbar_g^{\mathrm{nrt}}$ is obtained from a graph in $\hat{\Gamma}_g$ by choosing a (possibly empty) chain of pointed rational curves for each bivalent vertex of weight $0$, choosing a point of the $\bbS$-variety $\M_{w(v)}$ for each remaining vertex, and then gluing as determined by the graph.

\section{The P\'olya--Petersen character}\label{sec:ppchar}

To translate the stratifications of \S\ref{sec:graph_strat} into concrete formulas involving symmetric functions, we introduce the P\'olya--Petersen character of a graph. Let $G \in \mathrm{Ob}(\Graph)$. Recall that the isomorphism (\ref{eqn:aut_embedding}) gives up to conjugacy an embedding of groups $\Aut(G) \hookrightarrow \bbS_{\nu(G)}$.
\begin{defn}
The \textit{P\'olya--Petersen character} of $G$ is 
\[ \zeta_G := \ch\left( \Ind_{\Aut(G)}^{\bbS_{\nu_G}} \mathrm{triv} \right) \in \WRing_{\nu_G}. \]
\end{defn}
The P\'olya--Petersen character immediately lets us calculate Serre characteristic of the graded $\bbS$-varieties $\Mbar_G$ which stratify $\Mbar$.
\begin{prop}\label{prop:ecoarse}
For any $G \in \mathrm{Ob}(\Graph)$, we have \[\cat{e}_t^{\bbS}(\Mbar_G) = \left(\zeta_G \cdot t^{|E(G)|}\right) \tcirc \Delta \cat{e}_t^{\bbS \times \bbS}(\M).\]
\end{prop}
\begin{proof}
We have
\begin{align*} 
\Mbar_G[\chi(G) - 1] &\cong \left(\boxtimes_{v \in V(G)} \Res^{\bbS}_{S_{\val(v)} \times \bbS} \M\right)/\Aut(G)
\\&\cong \left(\Ind_{\Aut(G)}^{\bbS_{\nu(G)}}\Res_{\Aut(G)}^{\bbS_{\nu(G)}}\boxtimes_{v \in V(G)} \Res^{\bbS}_{S_{\val(v)} \times \bbS} \M\right)/\bbS_{\nu(G)}
\\&\cong \left(\Ind_{\Aut(G)}^{\bbS_{\nu(G)}} \Spec \C \times \boxtimes_{v \in V(G)} \Res^{\bbS}_{S_{\val(v)} \times \bbS} \M\right)/\bbS_{\nu(G)}
\end{align*}
The first isomorphism is (\ref{eqn:coarse_strata}). The second isomorphism comes from the fact that given a subgroup $H\subset G$ and an $H$-variety $X,$ there is $X/H\cong (\Ind_{H}^G X)/G.$ Similarly, the third isomorphism is from the fact that given a subgroup $H\subset G$ and a $G$-variety $Y,$ there is $\Ind_{H}^G\Res_H^G Y\cong G/H\times Y,$ where the right hand side admits the diagonal $G$-action. 
We can check that these isomorphisms preserve the relevant $S_n$-action and hence are isomorphisms of $\bbS$-varieties. The claimed formula now follows from taking equivariant Serre characteristics: the factor of $t^{|E(G)|}$ comes from the formula
\[1 - \chi(G) = |E(G)|- |V(G)| + 1\]
and Lemma \ref{lem:degree_shift}.
\end{proof}

For $\boldsymbol{\tau} \in \bbS_\nu$, we recall from Definition \ref{defn:2_part_conj} that $\Theta^{\boldsymbol{\tau}} \in \Partt_\nu$ denotes its cycle type, and that for any $\Theta \in \Partt$, we have
\[ p_\Theta = \prod_{\mu \in\Part^\star} p_{\Theta(\mu)}(\mu). \]
Using again the embedding  $\Aut(G) \hookrightarrow \bbS_{\nu(G)}$ (\ref{eqn:aut_embedding}), and using the notation $\boldsymbol{\tau}$ for elements of $\Aut(G)$, we write $\Theta^{\boldsymbol{\tau}}$ for the $2$-partition encoding the conjugacy class of the image of $\boldsymbol{\tau}$ in $\bbS_{\nu(G)}$. With this notation in hand, basic character theory yields the following formula for $\zeta_G$. 

\begin{lem}\label{lem:ungraded-PP-formula}
 Suppose $G \in \mathrm{Ob}(\Graph)$. Then
\[\zeta_G = \frac{1}{|\Aut(G)|} \sum_{\boldsymbol{\tau} \in \Aut(G)} p_{\Theta^{\boldsymbol{\tau}}}. \]
\end{lem}

\begin{exmp}
Consider the graph $G$ given in Figure \ref{fig:graph_exmp}. In this case $\nu(G) = (2^33^2)$, so
\[\bbS_{\nu(G)} = (S_2 \wr S_3) \times (S_3 \wr S_2), \]
and $\Aut(G)$ is an extension of $S_3$ by $\Z/2\Z$, where the $\Z/2\Z$ factor corresponds to a horizontal flip, and the elements of $S_3$ permute the vertices of valence $2$ and their adjacent edges.
    \begin{figure}[h]
        \centering
        \includegraphics[scale=1]{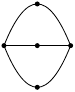}
        \caption{}    
        \label{fig:graph_exmp}
    \end{figure}

 Using Lemma \ref{lem:ungraded-PP-formula}, we compute the P\'olya--Petersen character of $G$ as
\begin{align*}
\zeta_G = \frac{1}{12}\left( \underbrace{p_1(1^2)^3 p_1(1^3)^2}_{\text{identity}} \right. &+ \underbrace{3p_1(1^2) p_2(1^2)p_1(1^12^1)}_{2\text{-cycles without flip}} + \underbrace{2p_3(1^2)p_1(3)^2}_{3\text{-cycles without flip}} \\& + \left.\underbrace{p_1(2^1)^3 p_2(1^3)}_{\text{flip}} + \underbrace{3 p_1(2^1) p_2(1^2) p_2(1^3)}_{2\text{-cycles with flip}} + \underbrace{2 p_3(2^1) p_2(3^1)}_{3\text{-cycles with flip}} \right).
\end{align*}
\end{exmp}

\begin{defn}\label{defn:AutTheta}
    For $G \in \mathrm{Ob}(\Graph)$ and $\Theta \in \Partt$, define
\[ \Aut^{\Theta}(G) := \{ \boldsymbol{\tau} \in \Aut(G) \mid \Theta^{\boldsymbol{\tau}} = \Theta \}. \]
\end{defn}
We observe that $\Aut^{\Theta}(G) \neq \varnothing$ implies that $\Theta \in \Partt_{\nu(G)}$, and that each $\Aut^{\Theta}(G)$ is a union of conjugacy classes of $\Aut(G)$. 
\begin{cor}\label{cor:single-graph-all-genus}
We have
\[ \mathsf{e}_t^{\bbS}(\Mbar_G) = t^{|\nu(G)|/2}\sum_{\Theta \in \Partt}\frac{|\Aut^{\Theta}(G)|}{|\Aut(G)|}  \prod_{\mu \in \Part^\star} \psi_{\Theta(\mu)}(\overline{\partial}_\mu \mathsf{e}_t^{\bbS}(\M) ). \]
\end{cor}
Corollary \ref{cor:single-graph-all-genus} implies Theorem \ref{thm:all_graphs}, whose statement we now recall. We have set the notation
\[\overline{\cat{a}} = \cat{e}^{\bbS}_t(\Mbar)\quad \mbox{and} \quad \cat{a}  = \cat{e}_t^{\bbS}(\M),\]
and for a $2$-partition $\Theta$ we set $||\Theta|| = \sum_{\mu \in \Part^\star} |\Theta(\mu)|\cdot|\mu|$.

\begin{customthm}{B}
The total generating function is given by
\[ \overline{\cat{a}} = \sum_{\Theta \in \Partt} O(\Theta)\cdot t^{||\Theta||/2}  \prod_{\mu \in \Part^\star} \psi_{\Theta(\mu)} (\overline{\partial}_\mu \cat{a}) \]
where $O(\Theta)\in \mathbb{Q}$ is defined by
\[O(\Theta) := \sum_{G \in \mathrm{Iso}(\Graph)} \frac{|\Aut^{\Theta}(G)|}{|\Aut(G)|}. \]
\end{customthm}

\begin{proof}[Proof of Theorem \ref{thm:all_graphs}]
We have
\[ \overline{\cat{a}} = \sum_{G \in \mathrm{Iso}(\cat{Graph})} \cat{e}_t^{\bbS}(\Mbar_G). \]
Applying the formula of Corollary \ref{cor:single-graph-all-genus}, we obtain
\[\overline{\cat{a}} = \sum_{G \in \mathrm{Iso}(\cat{Graph})}t^{|\nu(G)|/2}\sum_{\Theta \in \Partt}\frac{|\Aut^{\Theta}(G)|}{|\Aut(G)|}  \prod_{\mu \in \Part^\star} \psi_{\Theta(\mu)}(\overline{\partial}_\mu \cat{a}).\]
Note that if $\Aut^\Theta(G) \neq \varnothing$, then $||\Theta|| = |\nu(G)|$. Therefore, we can switch the summations over $G$ and $\Theta$ to obtain
\[\overline{\cat{a}} = \sum_{\Theta \in \Partt} t^{||\Theta||/2} \sum_{G \in \mathrm{Iso}(\cat{Graph})} \frac{|\Aut^{\Theta}(G)|}{|\Aut(G)|} \prod_{\mu \in \Part^\star} \psi_{\Theta(\mu)}(\overline{\partial}_\mu \cat{a}), \]
and Theorem \ref{thm:all_graphs} is proved.
\end{proof}


\begin{rem}
Theorem \ref{thm:all_graphs} is proved by summing the contribution of each coarse dual graph $G \in \mathrm{Ob}(\Graph)$ individually. For us, the P\'olya--Petersen character is crucial for this purpose: we do not, for example, see a way to express these contributions in terms of the character
\[ z_G^\mathrm{\mathrm{Half}}:=\Ind_{\Aut(G)}^{\Aut(H(G))} \triv \in \Lambda_{|H(G)|}, \]
which remembers only the action of $\Aut(G)$ on the set of half-edges $H(G)$. One na\"ive approach to graph strata would be to use the adjoint operator $D(z_G^{\mathrm{Half}})$ with respect to multiplication by $z_G^\mathrm{\mathrm{Half}}$, and guess that the generating function
\[ D(z_G^{\mathrm{Half}})( h_{|V(G)|} \circ \cat{a}) \]
is the contribution from the graph $G$, where $h_k \in \Lambda$ denotes the $k$th homogeneous symmetric function. This is not refined enough to specify individual graph strata. For example, if $P$ is the connected graph with one edge and two vertices, then
\[z_{P}^{\mathrm{
Half
}} = h_2 \in \Lambda_2, \]
and one checks that
\[D(z_{P}^{\mathrm{
Half
}})(h_2 \circ\cat{a}) = \zeta_{P} \tcirc \cat{a} + \cat{a} \cdot (\zeta_L \tcirc \cat{a}), \]
where $L$ is the loop graph with one vertex and one edge. In other words, the formula receives contribution from both the graph $P$ and $L.$

\end{rem}

\subsection{Vertex-weighted characters} We now turn to vertex weighted graphs and the fine graph stratification that they define. This will lead to the proof of Theorem \ref{thm:fixed-genus-thm}. The P\'olya--Petersen character of a vertex weighted graph is defined via the embedding of its automorphism group described in Lemma \ref{lem:refsubgrp}.

\begin{defn}
    Let $\mathbf{G}\in \cat{VGraph}_g$ be a vertex weighted graph of genus $g.$ The P\'olya--Petersen character of $\G$ is defined as $$\zeta_{\G} = \Ind_{\mathrm{Aut}(\G)}^{\bbS_{\underline{\nu}(\G)}}\mathrm{triv}\in R(\bbS_{\underline{\nu}(\G)})\subset \bigotimes_{h = 0}^g \Lambda^{[2]}.$$
\end{defn}

The terms in the tensor product on the right hand side are in bijection with the possible vertex weights on the graph. In this way, the P\'olya--Petersen character picks out the precise genus contribution from each vertex and calculates the equivariant Euler characteristic of the corresponding fine graph stratum. 

To work with these genus-decorated characters, we write 
\[\bigotimes_{h = 0}^{g} \WRing = \QQ[p^{(i)}_1(\mu), p^{(i)}_{2}(\mu), \ldots \mid \mu \in \Part^\star,\, 0\leq i \leq g],\]
and define
\[\mathrm{proj}_i : \bigoplus_{h = 0}^{g} \hat{\Lambda} \otimes \hat{\Lambda} \to \hat{\Lambda}\otimes \hat{\Lambda}\]
as the projection onto the $i$th factor of the direct sum. The proof of the following formula follows from the same calculation as in Proposition \ref{prop:ecoarse}.
\begin{rem}
    The map $\mathrm{proj}_i$ defined above is not to be confused with the projection maps $\mathrm{proj}_\mu: \hat{\Lambda} \otimes \hat{\Lambda}\to \hat{\Lambda}$ from Definition \ref{defn:projmu} where $\mu\in \Part^\star.$ 
\end{rem}

\begin{lem}\label{lem:gradedPPlem}
    Define $$\tcirc_g: \left(\bigotimes_{h=0}^g\WRing\right)\times \bigoplus_{h=0}^g(\hat{\Lambda}\otimes\hat{\Lambda})\to \hat{\Lambda}$$ by
    \[p_n^{(i)}(\mu) \tcirc_g q = \psi_n(\mathrm{proj}_\mu(\mathrm{proj}_i(q))), \]
    and requiring that $f \mapsto f\tcirc_g q$ is an algebra homomorphism for all $q \in \hat{\Lambda} \otimes \hat{\Lambda}$. Then
    $$\mathsf{e}(\Mbar_\G ) = \zeta_\G\tcirc_g \left(\bigoplus_{h=0}^g \Delta\mathsf{e}(\mathcal{M}_{h})\right).$$
\end{lem}

We will apply Lemma \ref{lem:gradedPPlem} to derive Theorem \ref{thm:fixed-genus-thm}. Let $\mathbf{G}$ be a vertex-weighted graph with automorphism group $\Aut_{\cat{VGraph}}(\mathbf{G}),$ and let $$\underline{\boldsymbol{\tau}} = (\boldsymbol{\tau}_0, \dots, \boldsymbol{\tau}_g)\in \Aut(\mathbf{G})\subset \bbS_{\underline{\nu}(\mathbf{G})},$$ so that each $\boldsymbol{\tau}_h\in \bbS_{\underline{\nu}(\mathbf{G})^{(h)}}.$ As before, we associate a 2-partition $\Theta^{\boldsymbol{\tau}_{h}}\in \Partt_{\underline{\nu}(\mathbf{G})^{(h)}}$ encoding the conjugacy class of $\boldsymbol{\tau}_h$. The following lemma upgrades the statements of Lemma \ref{lem:ungraded-PP-formula} and Corollary \ref{cor:single-graph-all-genus} to the genus-graded setting. The first part is the usual character formula for finite groups, and the second part follows from the stratification formula (\ref{eqn:fine_strata}).

\begin{lem}\label{lem:vgraph-Serre}
    Let $\G\in \mathrm{Ob}(\cat{VGraph}_g),$ then $$\zeta_{\G}=\frac{1}{|\mathrm{Aut}(\G)|}\sum_{\underline{\boldsymbol{\tau}}\in\Aut(\G)}\bigotimes_{h=0}^g p_{\Theta^{\boldsymbol{\tau}_{h}}}\in \bigotimes_{h=0}^g\Lambda^{[2]}.$$ For each $\underline{\Theta} = (\Theta_0,\dots, \Theta_g)\in \prod_{h = 0}^g\Partt_{\underline{\nu}(\G)^{(h)}},$ define $$\mathrm{Aut}^{\underline{\Theta}}(\G):=\{\underline{\boldsymbol{\tau}}\in \mathrm{Aut}(\G)\mid (\Theta^{\boldsymbol{\tau}_{0}},\dots, \Theta^{\boldsymbol{\tau}_{g}}) = \underline{\Theta}\}\subset \Aut(\G).$$ Then $$\mathsf{e}^{\bbS}(\Mbar_{\G}) = \sum_{\underline{\Theta} \in \prod_{h=0}^g\Partt}\frac{|\Aut^{\underline{\Theta}}(\G)|}{|\Aut(\G)|}  \prod_{h=0}^g\prod_{\mu \in \Part^\star} \psi_{\Theta_h(\mu)}(\overline{\partial}_\mu \mathsf{e}^{\bbS}(\M_h)).$$
\end{lem}

Applying the formula for P\'olya--Petersen characters of vertex-weighted graphs in Lemma \ref{lem:vgraph-Serre} to the reduced fine dual graph stratification (\ref{eqn:reduced_stratum}), we obtain the following formula for the Serre characteristic of a reduced fine dual graph stratum.

\begin{lem}\label{lem:reduced-fine-Serre}
    Let $\mathfrak{G}\in \mathrm{Ob}(\hat{\Gamma}_g)$, and let $\tcirc_g$ be as above. Set \[\cat{c}:=\mathsf{e}^{S_2 \times \bbS}(\mathcal{C}) - \mathsf{e}^{S_2 \times \bbS}(\Res^{\bbS}_{S_2\times \bbS} \mathcal{M}_{0})\in K_0(\cat{MHS}) \otimes{\Lambda} \otimes \hat{\Lambda}. \] Then 

    $$\mathsf{e}^{\bbS}(\Mbar_\mathfrak{G}^{\mathrm{nrt}} ) = \zeta_\mathfrak{G}\tcirc_g \left((\Delta\mathsf{e}^{\bbS}(\mathcal{M}_0)+\mathsf{c})\oplus \bigoplus_{h=1}^g \Delta\mathsf{e}^{\bbS}(\mathcal{M}_{h})\right).$$
    \end{lem}

The $\bbS\times\bbS$-module $\Delta(\mathcal{M}_0) + \mathsf{c}$ is precisely the contribution from genus zero vertices in (\ref{eqn:reduced_stratum}): the caterpillar generating function $\cat{e}(\calC)$ is inserted at bivalent vertices. We can express Lemma \ref{lem:reduced-fine-Serre} in terms of power sums as follows.

\begin{lem}\label{lem:reduced-fine-Serre-clean}
    For an integer $h \geq 0$ and $\mu\in \Part^\star$, define $$\mathsf{w}_h^{\mu} = \begin{cases}      
 \displaystyle{\frac{1}{1 - \mathsf{a}_0''}} &\text{if } \,h = 0,\,\mu = 1^2\\
   \\
 \displaystyle{\frac{1 + 2 \dot{\mathsf{a}}_0}{1 - \psi_2(\mathsf{a}_0'')}} &\text{if } \,h = 0,\, \mu = 2^1\\ \\ \displaystyle{\overline{\partial}_{\mu}\mathsf{e}^{\bbS}(\mathcal{M}_h)} &\text{otherwise}.
    \end{cases}$$
     Then 
    $$\mathsf{e}^{\bbS}(\Mbar_{\mathfrak{G}}^{\mathrm{nrt}}) = \sum_{\underline{\Theta} \in \prod_{h=0}^g\Partt}\frac{|\Aut^{\underline{\Theta}}(\mathfrak{G})|}{|\Aut(\mathfrak{G})|}   \prod_{h = 0}^g \prod_{\mu \in \Part^\star}  \psi_{\Theta_{h}(\mu)}(\mathsf{w}_h^{\mu}).$$
\end{lem} 
\begin{proof}
    For $h\neq 0$, we define $\mathsf{s}_h^{\mu}:=p_{\Theta_{h}(\mu)}\tcirc(\Delta\mathsf{e}^{\bbS}(\mathcal{M}_h)).$ For $h = 0,$ set $\mathsf{s}_0^{\mu}:=p_{\Theta_{0}(\mu)}(\mu)\tcirc(\Delta\mathsf{e}^{\bbS}(\mathcal{M}_0)+\mathsf{c}).$ Then by expanding the formula of Lemma \ref{lem:reduced-fine-Serre} analogously to Corollary \ref{cor:single-graph-all-genus}, we have $$\mathsf{e}^{\bbS}(\Mbar_{\mathfrak{G}}^{\mathrm{nrt}}) = \sum_{\underline{\Theta} \in \prod_{h=0}^g\Partt}\frac{|\Aut^{\underline{\Theta}}(\G)|}{|\Aut(\G)|}   \prod_{h = 0}^g \prod_{\mu \in \Part^\star}  \mathsf{s}_h^{\mu}.$$ It suffices to prove that $\mathsf{s}_h^{\mu} = \psi_{\Theta_{h}(\mu)}(\mathsf{w}_h^{\mu}).$ For $h\neq 0,$ this is true by applying Definition \ref{defn:partialbar}. The formula in Theorem \ref{thm-tcirc} yields $$p_{\Theta_{0}(\mu)}(\mu)\tcirc(\Delta\mathsf{e}^{\bbS}(\mathcal{M}_0)+\mathsf{c}) = \psi_{\Theta_{0}(\mu)}(\mathrm{proj}_{|\mu|}(\Delta\mathsf{e}^{\bbS}(\mathcal{M}_0)+\mathsf{c})).$$ When $|\mu|\neq 2$, $\mathrm{proj}_{\mu}(\Delta\mathsf{e}(\mathcal{M}_0)+\mathsf{c}) = \mathrm{proj}_{\mu}(\Delta\mathsf{e}(\mathcal{M}_0)),$ so the description of $\mathsf{w}^{\mu}_0$ follows from the same reasoning as earlier. When $|\mu|=2, $ we have $\mathrm{proj}_{\mu}(\Delta\mathsf{e}^{\bbS}(\mathcal{M}_0)+\mathsf{c}) = \mathrm{proj}_\mu(\mathsf{e}^{S_2 \times \bbS}(\mathcal{C})),$ and $$\mathsf{e}^{S_2 \times \bbS}(\mathcal{C})= \frac{1}{2}p_1^2\otimes \left(\frac{1}{1 - \mathsf{a}_0''}\right) + \frac{1}{2}p_2\otimes \frac{1 + 2 \dot{\mathsf{a}}_0}{1 - \psi_2(\mathsf{a}_0'')},$$ (see \cite[\S6]{genusonechar} or \cite{BergstromMinabe1})\footnote{The characteristic $\cat{e}^{S_2 \times \bbS}(\calC)$ can also be computed via a mild generalization of the P\'olya--Petersen character which allows for graphs with legs, but we omit this derivation here.} so that $$\mathrm{proj}_{1^2}(\mathsf{e}^{S_2 \times \bbS}(\mathcal{C})) = \frac{1}{1 - \mathsf{a}_0''} \quad \mbox{and}\quad\mathrm{proj}_{2^1}(\mathsf{e}^{S_2 \times \bbS}(\mathcal{C})) = \frac{1 + 2 \dot{\mathsf{a}}_0}{1 - \psi_2(\mathsf{a}_0'')},$$ which gives $\mathsf{w}_0^{\mu}$ for $\mu = 1^2, 2^1.$
\end{proof}
Now we may use Lemma \ref{lem:reduced-fine-Serre-clean} to prove Theorem \ref{thm:fixed-genus-thm}.
\begin{proof}[Proof of Theorem \ref{thm:fixed-genus-thm}]
In view of the stratification (\ref{eqn:reduced_stratification}), summing over all objects $\mathfrak{G}\in \mathrm{Ob}(\hat{\Gamma}_g)$ yields the formula $$\mathsf{e}^{\bbS}(\Mbar_{g}^{\mathrm{nrt}}) = \sum_{\mathfrak{G}\in \mathrm{Ob}(\hat{\Gamma}_g)} \sum_{\underline{\Theta} \in \prod_{h=0}^g\Partt}\frac{|\Aut^{\underline{\Theta}}(\mathfrak{G})|}{|\Aut(\mathfrak{G})|}   \prod_{h = 0}^g \prod_{\mu \in \Part^\star}  \psi_{\Theta_{h}(\mu)}(\mathsf{w}_h^{\mu}).$$ Theorem \ref{thm:fixed-genus-thm} follows from switching the sum over $\hat{\Gamma}_g$ and $2$-partitions, and then applying Lemma \ref{lem:attaching_rat_tails}.
\end{proof}
\section{Combinatorial subspaces of the moduli space of curves}\label{sec:comb_subspaces}
Our techniques are well-suited for deriving formulas for combinatorially defined subspaces of the $\bbS$-variety $\Mbar_g$. 
\subsection{Graph-genus generating functions}
\begin{defn}Let $\gamma \in \Z_{\geq 0}$. We set
\[ \M^{(\gamma)}_g \subset \Mbar_g \]
for the $\bbS$-subspace consisting of pointed stable curves $C$ whose coarse dual graph $G_C$ satisfies
\[\dim_{\QQ} H_1(G_C;\QQ) = \gamma; \]
this is equivalent to requiring that
\begin{equation}\label{eqn:graph_genus}
    E(G_C) - V(G_C) + 1= \gamma.
\end{equation}
This subspace is especially interesting when $\gamma = 0$, when we have that
\[\M^{(0)}_g = \M^{\mathrm{ct}}_g \]
is the moduli space of curves of compact type. Let
\[ \overline{\cat{a}}_g^{(\gamma)} = \cat{e}^{\bbS}(\M^{(\gamma)}_g). \]
\end{defn}
Then emulating the techniques of previous sections, we derive the following formula for $\overline{\cat{a}}^{(\gamma)}_g$. Let $\hat{\Gamma}_g^{(\gamma)} \subset \hat{\Gamma}_g$ denote the subgroupoid of graphs satisfying (\ref{eqn:graph_genus}).
\begin{thm}\label{thm:graph-genus-gen-fun} Suppose $g \geq 2$. Then we have
\[\overline{\cat{a}}_g^{(\gamma)} = \left(\sum_{\Theta_0, \ldots, \Theta_g \in \Partt} K^{(\gamma)}(\Theta_0, \ldots, \Theta_g)  \prod_{h = 0}^g\prod_{\mu \in \Part^\star} \psi_{\Theta(\mu)} (\cat{w}_h^{\mu})\right) \circ (p_1 + \overline{\cat{a}}_0'),\]
where
\[ \cat{w}_{h}^{\mu}:= \begin{cases}
\overline{\partial}_{\mu}\mathsf{a}_h &\text{if }2h-2 + |\mu| > 0,\\[10pt] 
 \displaystyle{\frac{1}{1 - \mathsf{a}_0''}} &\text{if } h = 0,\,\mu = 1^2,\\[12pt]
 \displaystyle{\frac{1 + 2 \dot{\mathsf{a}}_0}{1 - \psi_2(\mathsf{a}_0'')}} &\text{if } h=0,\,\mu = 2^1.
 \end{cases}
\]
and
\[ K^{(\gamma)}(\Theta_0, \ldots, \Theta_g) := \sum_{\G \in \mathrm{Iso}(\hat{\Gamma}_g^{(\gamma)})} \frac{|\Aut^{\Theta_0, \ldots, \Theta_g}(\G)|}{|\Aut(\G)|}. \]
\end{thm}
See Figure \ref{fig:genus3ct} for the contribution of each graph to the generating function $\cat{a}_3^{(0)}$ for $S_n$-equivariant Serre characteristics of moduli spaces of curves of compact type in genus $3$.

\begin{figure}[h]
    \centering
    \includegraphics[scale=1]{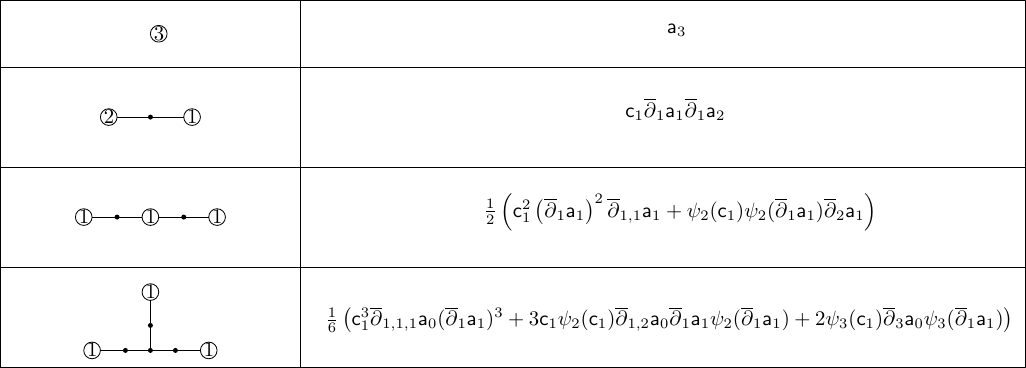}
    \caption{The generating function $\cat{a}^{(0)}_3$ is computed by taking the sum of each of the graph contributions in the table, and then performing plethysm with $p_1 + \overline{\cat{a}}_0'$. As in Figure \ref{fig:genus-two-fig}, we set $\cat{c}_1 = \frac{1}{1 - \cat{a}_0''}$ and $\cat{c}_2 = \frac{1 + 2 \dot{\cat{a}}_0}{1 - \psi_2(\cat{a_0'')}}$.}
    \label{fig:genus3ct}
\end{figure}
\begin{rem}
    By a direct translation of the methods of \cite[\S4]{GetzlerPandharipande} one can derive an explicit recursive formula that determines the generating function
    \[ \sum_{g \geq 0} \cat{a}_{g}^{(0)} t^{g-1} \]
    for all Serre characteristics of moduli spaces of curves of compact type, in terms of the generating function $\cat{a}$ for moduli spaces of smooth pointed curves.
    We omit this derivation here, but point out that an advantage of Theorem \ref{thm:graph-genus-gen-fun} is that it gives a formula for $\cat{a}_g^{(0)}$ for each individual genus $g$, in terms of explicit tree sums.
\end{rem}

\subsection{Cores, higher genus caterpillars, and curves of compact type}
An alternative formula for $\overline{\cat{a}}$, independent from both Theorem \ref{thm:fixed-genus-thm} and the Getzler--Kapranov formula (Theorem \ref{thm:GKformula}), can be derived using a combinatorial decomposition of $\Mbar$ as follows. 
\begin{defn}
Let
\[ \Mbar^{\mathrm{core}} \subset \Mbar \]
denote the moduli space of pointed curves $C$ whose \emph{coarse} dual graphs $G_C$ are equal to their own cores. Coarse dual graphs do not have vertex weights, so in this context the core of a connected graph $G \in \mathrm{Ob}(\cat{Graph})$ is simply the minimal connected subgraph $G' \subset G$ such that \[\dim_{\QQ} H_1(G';\QQ) = \dim_{\QQ}H_1(G;\QQ).\]
Now let
\[ \Graph_{\circ} \subset \Graph \]
be the subgroupoid of connected graphs which do not have any vertices of valence $1$ or $2$; this implies that in particular every graph in $\Graph_{\circ}$ is equal to its own core. Finally, let $\widehat{\Graph_{\circ}}$ be the groupoid obtained by barycentrically subdividing each edge of a graph in $\Graph_{\circ}$. The following formula follows from applying the same grafting techniques from Section \ref{sec:catails} to the setting of coarse graphs.
\end{defn}
\begin{thm}\label{thm:core_gen_fun}
\[\overline{\cat{a}} = \left(\sum_{\Theta \in \Partt} O_\circ(\Theta) \cdot t^{||\Theta||/2}\prod_{\mu \in \Part^\star} \psi_{\Theta(\mu)}(\cat{b}_\mu)\right) \circ \left(p_1 + \frac{\partial}{\partial p_1}\sum_{g \geq 0} \cat{a}^{(0)}_g t^{g}\right), \]
where
\[\cat{b}_\mu = \begin{cases} \frac{1}{t\cdot(1 - t
\cdot\cat{a}'')} &\text{if }\mu = 1^{2}\\
\frac{1 + 2t\cdot\dot{\cat{a}}}{t\cdot(1 - \psi_2(t\cdot\cat{a}''))} &\text{if }\mu = 2^1\\
\overline{\partial}_{\mu} \cat{a} &\text{otherwise},
\end{cases}\]

and 
\[O_\circ(\Theta) = \sum_{G \in \mathrm{Iso}(\widehat{\Graph}_\circ)} \frac{|\Aut^\Theta(G)|}{|\Aut(G)|}. \]
\end{thm}

\begin{proof}
We sketch the proof; it can be made formal by imitating the techniques in Section \ref{sec:graph_strat}. The formula reflects the combinatorial fact that all connected graphs in $\cat{Graph}$ can be obtained from graphs in $\widehat{\Graph}_\circ$ in the following two steps:
\begin{itemize}
\item subdivide edges arbitrarily;
\item attach arbitrary trees to the resulting graph.
\end{itemize}    
We note that the vertices that are present in both edge subdivisions and rational tree attachments may carry positive genus weightings. Subdividing edges is accounted for by the $\cat{b}_\mu$ terms when $|\mu| = 2$, while the $\overline{\partial}_\mu \cat{a}$ terms are inserted at vertices with valence greater than $2$. The operation of taking plethysm with
\[p_1 + \frac{\partial}{\partial p_1} \sum_{g \geq 0} \cat{a}_g^{(0)} t^g \]
accounts for attaching trees on the core; this operation is the compact-type analogue of Lemma \ref{lem:attaching_rat_tails}.
\end{proof}

Our main interest in Theorem \ref{thm:core_gen_fun} is the possibility for future applications: we expect that the calculation of the factors $O_\circ(\Theta)$ appearing in the formula should be simpler to compute than the terms $O(\Theta)$ in Theorem \ref{thm:all_graphs}, though we do not yet have a systematic approach for computing these numbers in general.

\section{$\C^\star$-fixed stable maps}\label{sec:stable-maps}
Let $X$ be a smooth projective variety over $\C$ and let $\Mbar_{g,n}(X,\beta)$ be the moduli space of Kontsevich stable maps to $X.$ When $X$ admits a suitable $\C^\star$-action, the induced $\C^\star$-fixed points $\Mbar_{g,n}(X,\beta)^{\C^\star}$ have been classified by Kontsevich \cite{KontsevichTorusActions} and Graber--Pandharipande \cite{graberpand}: they are a disjoint union of smooth, proper Deligne--Mumford substacks of $\Mbar_{g,n}(X,\beta)$ indexed by decorated graphs. In the context of Gromov--Witten theory, \textit{torus localization} reduces virtual intersection theory on the reducible space $\Mbar_{g,n}(X,\beta)$ to calculations on $\Mbar_{g,n}(X,\beta)^{\C^\star}$ and is one of the most powerful techniques in the subject.

We revisit $\Mbar_{g,n}(X,\beta)^{\C^\star}$ on the same footing as $\Mbar_{g,n}$ and calculate its $S_n$-equivariant Serre characteristics using P\'olya--Petersen characters. As pointed out in the introduction, the general formula that $\chi^{S_n}(Y) = \chi^{S_n}(Y^{\C^\star})$ for varieties $Y$ with a $\C^\star$-action implies that $$\chi^{S_n}(\Mbar_{g,n}(X,\beta)) = \chi^{S_n}(\Mbar_{g,n}(X,\beta)^{\C^\star}),$$ so our calculation will recover the $S_n$-equivariant Euler characteristics of $\Mbar_{g,n}(X,\beta).$

\begin{defn}\label{def:good-action}
    Let $X$ be a smooth projective variety with a $\C^\star$-action. The $\C^\star$-action on $X$ is \textit{good} if the fixed points of the $\C^\star$-action are isolated, and that every $\C^\star$-invariant curve is a smooth rational curve between two distinct $\C^\star$-fixed points.
\end{defn}
\begin{exmp}
    These hypotheses are satisfied in the following settings:
\begin{itemize}
    \item $X$ is a smooth complete toric variety, and $\C^\star$ acts as a general $1$-parameter subgroup of the dense torus,
    \item Let $\mathbf{d} = (d_1,\dots, d_\ell)$ with $0\leq d_1<\dots<d_\ell\leq n.$ The partial flag variety $X = \mathrm{Fl}(\mathbf{d}, n)$ parametrizes flags of linear subspaces $$\{0\subseteq V_1\subsetneq\dots\subsetneq V_{\ell}\subsetneq \C^n\mid \dim V_i = d_i\},$$ and $\C^\star$ acts diagonally with generic weights on $\C^n$. In particular, the Grassmannian $\mathrm{Gr}(k, n)$ of $k$-dimensional subspaces of $\C^n$ and the complete flag variety $\mathrm{Fl}(n)$ arise as special cases.
\end{itemize}
\end{exmp}

In this setting, there is a natural way to associate a graph to the $\C^\star$-action on $X$. 
\begin{defn}
Let $X$ be a smooth projective variety with a good $\C^\star$-action. The \textit{torus graph} $G_X$ of the action has a vertex for each $\C^\star$-fixed point, and an edge between two vertices for each $\C^\star$-invariant curve between the corresponding fixed points. Given an edge $e \in E(G_X)$, we write $\beta_e \in H_2(X;\Z)$ for the homology class of the $\C^\star$-invariant curve represented by $e$.
\end{defn}

\begin{rem}
    In the examples given above, the torus graphs agree with the GKM graphs of the torus action constructed in the work of Goresky--Kottwitz--MacPherson \cite{GKM}. \begin{itemize}
        \item When $X$ is a projective toric variety, the torus graph is the 1-skeleton of the polytope of $X.$
        \item When $X = \mathrm{Gr}(k,n)$ and $\mathrm{Fl}(n),$ the torus graphs are the 1-skeleta of the hypersimplex $\Delta(k,n)$ and the permutahedron, respectively.
    \end{itemize}
\end{rem}

We will require the notion of graph homomorphisms.
\begin{defn}
Let $G$ and $G'$ be two graphs. A \textit{graph homomorphism} $f: G \to G'$ is the data of maps \[f_V: V_G \to V_{G'},\]
\[f_H: H_G \to H_{G'} \]
which are compatible with the root maps and involutions. 
\end{defn}



\subsection{The groupoid of a good $\C^\star$-action}
The fixed point locus $\Mbar_{g,n}(X,\beta)^{\C^\star}$ admits a stratification indexed by decorated graph homomorphisms. The following definition is a combinatorial reformulation of the classification in \cite[§3.2]{KontsevichTorusActions} and \cite[§4]{graberpand}.

\begin{defn}
Suppose $X$ is a variety with a good $\C^\star$-action. We define the groupoid $\sfgamma_X$ to be the groupoid of triples $(G, f, \delta)$ where
\begin{itemize}
    \item $G$ is a connected graph,
    \item $f: G \to G_X$ is a graph homomorphism, and
    \item $\delta: E(G) \to \Z_{> 0}$ is a weighting of the edges of $G$ by positive integers.
    \end{itemize}
An isomorphism between $(G, f, \delta)$ and $(G', f', \delta')$ is a graph isomorphism $\varphi: G \to G'$ such that the diagrams
\[\begin{tikzcd}
&E(G) \arrow[rr,"\varphi"] \arrow[dr, "\delta"] & & E(G') \arrow[dl, "\delta'"']\\
&  &\Z_{> 0} & 
\end{tikzcd}\]
and 
\[\begin{tikzcd}
&G \arrow[rr,"\varphi"] \arrow[dr, "f"] & & G' \arrow[dl, "f'"']\\
&  &G_X & 
\end{tikzcd}\]
each commute.

We write $\sfgamma_{X, \beta} \subset \sfgamma_X$ for the subgroupoid of $\sfgamma_X$ consisting of maps $(G,f,\delta)$ such that
\[ \sum_{e \in E(G)} \delta(e)\cdot \beta_{f(e)} = \beta \in H_2(X; \Z). \]
\end{defn}

Recall the notation
\[\widehat{\cat{a}} = t^{-1}(h_1 + h_2) + \sum_{2g - 2 + n > 0} \cat{e}^{S_n}(\Mbar_{g, n}) \cdot t^{g-1}. \]
where \[h_1 = p_1 \quad\mbox{and}\quad h_2 = \frac{p_1^2 + p_2}{2}. \]
The terms $h_1$ and $h_2$ can be thought of as formally setting $\Mbar_{0,1}$ and $\Mbar_{0,2}$ to be points.

\begin{defn}
For a variety $X$ with a good $\C^\star$-action, we define
\[\overline{\cat{a}}^{\C^\star}_{X, \beta} := \sum_{g,n} \cat{e}^{S_n}(\Mbar_{g, n}(X, \beta)^{\C^\star})\cdot t^{g - 1} \in K_0(\cat{MHS}) \otimes \LLambda, \]
where the sum is over all $g, n \geq 0$ such that $\Mbar_{g, n}(X, \beta)$ is nonempty. If $X= \Spec \C$ is a point, then
\[ \overline{\cat{a}}_{\Spec \C, 0}^{\C^\star} = \overline{\cat{a}}, \]
where $\overline{\cat{a}}$ is defined as in (\ref{eqn:abar_defn}).
\end{defn}
For the remainder of this section, we will assume that $X$ is nontrivial, i.e., $X \neq \Spec \C$ and $X$ admits a good $\C^\star$-action.

By definition, the automorphism group $\Aut(G, f, \delta)$ is a subgroup of $\Aut_{\cat{Graph}}(G).$ Hence, as before, we have an embedding 
\[ \Aut(G, f, \delta) \hookrightarrow \bbS_{\nu(G)} \]
for any $(G, f, \delta) \in \mathrm{Iso}(\cat{\Gamma}_X)$, well-defined up to conjugacy. 

\begin{defn}
    The P\'olya--Petersen character of $(G, f, \delta)\in \mathrm{Iso}(\sfgamma_{X, \beta})$ is defined as \[ \zeta_{(G, f, \delta)}:= \ch(\Ind_{\Aut(G, f, \delta)}^{\bbS_{\nu(G)}} \mathrm{triv}) \in \WRing. \]
\end{defn}

We can reformulate prior work Graber--Pandharipande \cite{graberpand} and Kontsevich \cite{KontsevichTorusActions} on $\C^\star$-fixed stable maps in the following proposition.

\begin{prop}\label{prop:localized_stable_maps}
\[\overline{\cat{a}}^{\C^\star}_{X, \beta} = \sum_{(G, f, \delta) \in \mathrm{Iso}(\cat{\Gamma}_{X, \beta})} t^{|E(G)|}\cdot(\zeta_{(G, f, \delta)} \tcirc \Delta \widehat{\cat{a}}). \]
\end{prop}
\begin{proof}
    Let $\Mbar(X,\beta)^{\mathbb{C}^\star}$ be the genus-graded $\bbS$-variety given by $$\Mbar(X,\beta)^{\mathbb{C}^\star}(g,n):=\Mbar_{g,n}(X,\beta)^{\mathbb{C}^\star}.$$ As explained in \cite[§3.2]{KontsevichTorusActions} and \cite[§4]{graberpand}, each $(G, f, \delta)$ specifies a union of connected components in $\Mbar(X,\beta)^{\mathbb{C}^\star}$ isomorphic to\footnote{The statements given in loc. lit. concern triples $(G,f,\delta)$ together with a genus weighting, and the formula takes a union across all genus weightings of $(G,f,\delta).$} $$\left[\left(\midboxtimes_{\hspace{1pt}v\in V(G)} \mathrm{Res}^{\mathbb{S}}_{S_{\mathrm{val}(v)}\times \bbS} \Mbar^\dagger\right)/\Aut(G)\right][1-\chi(G)],$$
    where $\Mbar^{\dagger}$ is the genus-graded $\bbS$-space such that
    \[ \Mbar^\dagger(g, n) = \begin{cases}
        \Mbar(g,n) & \text{if }(g, n) \notin\{(0, 1), (0, 2)\}\\
        \Spec \C &\text{if }(g, n) \in \{(0, 1), (0, 2)\}
    \end{cases}. \]  
    
    Upon taking graded Serre characteristics, this gives precisely $t^{|E(G)|}\cdot(\zeta_{(G,f,\delta)}\tcirc \Delta\widehat{\cat{a}}).$
\end{proof}

Computing the P\'olya--Petersen characters appearing in Proposition \ref{prop:localized_stable_maps}, we obtain Theorem \ref{thm:torus_fixed_all_graphs} on Euler characteristics of moduli spaces of stable maps. In the following, given $\Theta\in \Partt,$ we define a union of conjugacy classes $\Aut^{\Theta}(G,f,\delta)\subset \Aut(G,f,\delta)$ in the same way as Definition \ref{defn:AutTheta}. 

\begin{customthm}{C}\label{thm:torus_fixed_all_graphs}
The formula
\[\overline{\cat{a}}^{\C^\star}_{X, \beta} = \sum_{\Theta \in \Partt} O_{X, \beta}(\Theta)\cdot t^{||\Theta||/2}  \prod_{\mu \in \Part^\star} \psi_{\Theta(\mu)}\left(\overline{\partial}_{\mu}\widehat{\cat{a}}\right)  \]
holds, where $O_{X, \beta}(\Theta) \in \QQ$ is defined by
\[O_{X, \beta}(\Theta) := \sum_{(G, f, \delta) \in \mathrm{Iso}(\cat{\Gamma}_{X, \beta})} \frac{|\Aut^{\Theta}(G, f, \delta)|}{|\Aut(G, f, \delta)|} \]
\end{customthm}

As remarked in the introduction, Theorem \ref{thm:torus_fixed_all_graphs} expresses $\overline{\cat{a}}_{X, \beta}^{\C^\star}$ as a polynomial in the variables $\psi_k(\overline{\partial}_\mu\overline{\cat{a}})$, for finitely many $k$ and $\mu$.

\begin{exmp}
When $X = \P^r$, we have $H_2(X;\Z) \cong \Z \cdot [L]$ where $[L]$ is the class of a line; thus $\beta$ is determined by an integer $d$. In this case the torus-fixed graph $G_{\P^{r}}$ is the complete graph with vertex set $(\P^r)^{\C^{\star}}$, and a graph homomorphism from an arbitrary graph $G$ to $G_{\P^r}$ is determined by an $(r +1)$-coloring of $G$. Graphs in the groupoid $\cat{\Gamma}_{\P^r, d}$ can have at most $d + 1$ vertices and at most $d$ edges. See Figure \ref{fig:stable_map_exmp} for the calculation of the generating function $\overline{\cat{a}}^{\C^\star}_{\P^r, 3}$ for degree $3$ torus-fixed stable maps to $\P^r$, and Tables \ref{table:stable_maps_no_markings} and \ref{table:degree3-frob-chars} for sample calculations.
\end{exmp}
\begin{rem}
By expanding the definition of the groupoid $\cat{\Gamma}_{X, \beta}$ to allow for contractions of edges in the graph morphism $G \to G_X$, one could give a formula for $\cat{\overline{a}}_{X, \beta}^{\C^\star}$ as a function of the power series $\cat{a}$, and then recover Theorem \ref{thm:all_graphs} as a special case of the resulting formula by taking $X = \Spec \C$. For ease of exposition, we do not work this out here, and note that plugging Theorem \ref{thm:all_graphs} into Theorem \ref{thm:torus_fixed_all_graphs} gives a formula for $\cat{\overline{a}}_{X, \beta}^{\C^\star}$ in terms of $\cat{a}$.
\end{rem}

\begin{figure}[h]
    \centering
    \includegraphics[scale=1]{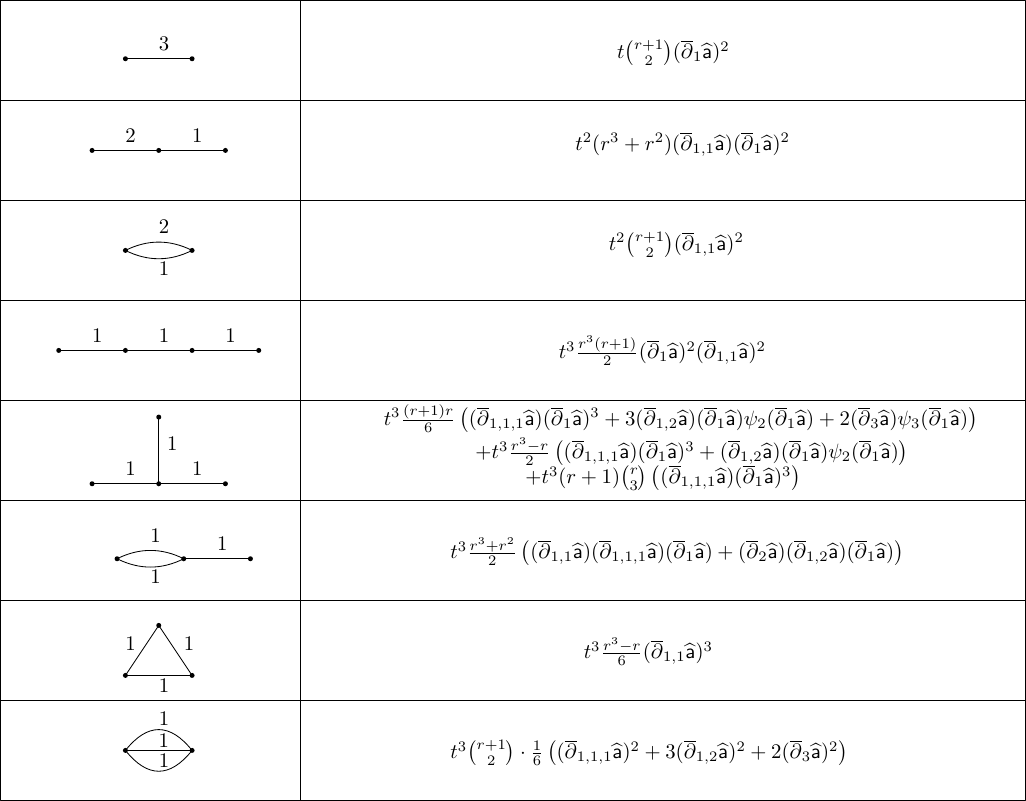}
\caption{The generating function $\overline{\cat{a}}_{\P^r, 3}^{\C^\star}$ is obtained by summing each of the graph contributions in the table.}
\label{fig:stable_map_exmp}
\end{figure}

\appendix
\section{Background on wreath products}\label{appendix:wreath_products}
Let $G$ be an arbitrary finite group, and let $G_*$ denote the set of conjugacy classes of $G$.

\begin{defn}Define the wreath product of $G$ and $S_n$ as \[G_n := G \wr S_n := G^n \rtimes S_n,\] where $S_n$ acts on $G^n$ by permutation. Concretely, $G_n$ can be seen as the group of $n\times n$ permutation matrices with entries in $G$. Note that $G_0$ is the trivial group.
\end{defn}
Conjugacy classes in $G_n$ have been completely classified, as we now recall.
\begin{defn}
Let $\omega = (g_1,\dots,g_n; \sigma) \in G_n.$ For each conjugacy class $c \in G_*$, define  $\lambda_\omega(c) \in \Part^\star$ by setting ${\lambda_\omega(c)}_i$ to be the number of $i$-cycles in $\sigma$ such that the product of the corresponding elements of $G$, in the cyclic order determined by the cycle, lies in $c$. This defines $\lambda_\omega: G_*\to \Part^\star$ as the \textit{cycle type} of $\omega$.
\end{defn}
A map $G_*\to \Part^\star$ that corresponds to a cycle type in $G\wr S_n$ can be thought of as a partition of $n$ with the additional data of an element of $G_*$ for each part of the partition. This is equivalent to requiring that $\sum_{[g]\in G_*}|\lambda_\omega([g])| = n.$
\begin{lem}\cite[§I.B.3]{Macdonald}\label{lem:wrconj}
 The conjugacy classes of $G_n$ are in bijection with maps $f: G_*\to \Part^\star$ such that $\sum_{[g]\in G_*}|f([g])| = n.$
\end{lem}

For the remainder of this section, let $k$ denote a splitting field of $G.$

\begin{defn}Let $R(G_n)$ be the $k$-vector space of virtual $G_n$-representations. Define
    \[ \Lambda(G) := \bigoplus_{n \geq 0} R(G_n) .\]
    $\Lambda(G)$ has a graded ring structure via the $\boxtimes$-product:
    $ V \boxtimes W = \Ind_{G_m \times G_n}^{G_{m + n}} V \otimes W,$ where for each $m,n$, there is an embedding $G_m \times G_n \to G_{m + n}$, well-defined up to conjugation.
    \end{defn}

\begin{thm}\cite[§I.B]{Macdonald}\label{eq:wreathchar}
    The ring $\Lambda(G)$ is isomorphic to the polynomial ring \[k[p_i(c) \mid i\geq 1, \,c\in G_* ].\] The isomorphism sends a $G_n$-representation $V$ to its character
    \[ \ch_{G_n}(V) : = \frac{1}{|G|^n\cdot n!} \sum_{\omega \in G_n} \mathrm{Tr}_V(\omega) p_{\lambda_\omega}, \]
    where for $\rho: G_*\to \Part_n^\star$, writing $\rho(c) = (1^{\rho(c)_1}\dots)$, $p_\rho$ is defined as $$p_\rho:=\prod_{c\in G_*}\prod_{i\geq 1}p_i^{\rho(c)_i}(c).$$
\end{thm}
The elements $p_i(c)$ are analogues of the power sum symmetric functions; in $\Lambda(G)$ there is one set of power sum symmetric functions for each element of $G_*$. This furnishes an isomorphism of $k$-algebras
\[ \Lambda(G) \cong \bigotimes_{c \in G_*} \Lambda. \]







\begin{exmp}
When $G = S_m$, \[G_* = \Part^\star_m \] so the ring $\Lambda(S_m) = \C[p_i(\mu) \mid i\geq 1, \, \mu \in \Part^\star_m ]$ has one set of variables for each partition of $m$, and they freely generate the ring.
\end{exmp}

\subsection{Generalized plethysm}

Extending the plethysm action of $\Lambda$ on itself, there is a plethysm action $\circ_{G}$ of $\Lambda(G)$ on $R(G) \otimes \Lambda$. In particular, for $G = S_m$ we get an action $\circ_m$ of $\Lambda(S_m)$ on $\Lambda_m \otimes \Lambda$ for all $m \geq 0$.

\begin{defn}
Define $\Lambda_G$ as \[\Lambda_G := R(G) \otimes_{k} \Lambda \cong \bigoplus_{c \in G_*} \Lambda. \]
The ring $\Lambda_G$ inherits a grading from $\Lambda$ and is the Grothendieck group of bounded $(G \times \mathbb{S})$-modules, tensored with $k$. 
\end{defn}
We can also set
\[\hat{\Lambda}_G := R(G) \otimes_k \hat{\Lambda}, \]
and $\hat{\Lambda}_G$ is the Grothendieck group of not necessarily bounded $(G \times \mathbb{S})$-modules, tensored with $k$. It inherits a filtration from the degree filtration on $\hat{\Lambda}$, and we set
\[ F_1 \Lambda_G := R(G) \otimes _{k} \left( \bigoplus_{n \geq 1} \Lambda_n\right) 
\subset \Lambda .\]

The vector space $\hat{\Lambda}_G$ admits a projection map
\[ \mathrm{proj}_c : \Lambda_G \to \Lambda \]
for each $c \in G_*$: each element  $q \in \hat{\Lambda}_G$ can be written uniquely as a sum
\[q = \sum_{c \in G_*} \frac{c^\vee}{z_c} \otimes f_c \] 
where 
\begin{itemize}
\item $f_c \in \hat{\Lambda}$;
\item $c^\vee \in R(G)$ is the class function taking value $1$ on the conjugacy class $c$ and $0$ elsewhere;
\item $z_c$ is the order of the centralizer of any element in the conjugacy class $c$.
\end{itemize}

\begin{defn}
   With the same notation as above, we define the projection of $q \in \hat{\Lambda}_G$ to $c\in G_*$ as \[ \mathrm{proj}_c(q):  = f_c.\]
\end{defn}
\begin{defn}
    A $G\wr \mathbb{S}$-module $\mathcal{V}$ is a collection $\{\mathcal{V}(n)\mid n\geq 0\}$ where each $\mathcal{V}(n)$ is a finite-dimensional $G_n$-representation over $k$. We say $\mathcal{V}$ is \textit{bounded} if $\mathcal{V}(n) = 0$ for $n \gg 0$. The ring $\Lambda(G)$ is isomorphic to the Grothendieck ring of the category of bounded $G\wr \mathbb{S}$-modules, tensored with $k$.
\end{defn}

\begin{lem}\cite{semiclassicalremark}\label{lem-Gwrformulas}
    Let $\mathcal{V}$ be a $G\times \bbS$-module, and let $\mathcal{W}$ be a bounded $G\wr \bbS$-module. The formula $$(\mathcal{W}\circ_{G}\mathcal{V})(n):=\bigoplus_{k\geq 0}\left(\mathcal{W}(k)\otimes \mathcal{V}^{\otimes k}(n)\right)^{G_k}$$ defines an $\mathbb{S}$-module $\mathcal{W}\circ_{G}\mathcal{V}$ and extends to a map \[\circ_G:  \Lambda(G) \times F_1 \hat{\Lambda}_G \to \hat{\Lambda}, \]
    determined by the following properties:
    \begin{enumerate}[(1)]
    \item for all $q \in F_1 \hat{\Lambda}_G$, the map $\Lambda(G) \to \Lambda$ defined by $f \mapsto f \circ_G q$ is an algebra homomorphism,
    \item for all $q \in \hat{\Lambda}_G$, we have
    \[ p_n(c) \circ_G q = \psi_n\left( \mathrm{proj}_c(q) \right). \]
    \end{enumerate}
    \end{lem}
    
    \begin{rem}
    When $G = S_1$, the splitting field is $\QQ,$ and there are canonical identifications $\Lambda(G) \cong \Lambda$, and $\Lambda_G = (\mathbb{Q}\cdot p_1)\otimes \Lambda$. Under these identifications, we have
    \[ p_1 \circ_{S_1} (p_1 \otimes f) = f,  \]
    hence $\circ_{S_1}$ recovers the usual plethysm of symmetric functions.
    \end{rem}

\section{Polynomial functors}\label{sec:appendixPFun}
For this section, fix a field $k$ of characteristic zero. Plethysms can be understood conceptually as compositions of polynomial functors between $k$-linear categories, which we now recall following the foundational text of Macdonald \cite{MacDonald1980}. We work with the following class of categories.

\begin{defn}
    A $k$\textit{-linear abelian category} $\catC$ is an abelian category such that there exists an embedding \[k\hookrightarrow \mathrm{End}(\mathbf{1}_{\catC}, \mathbf{1}_{\catC}),\] the ring of natural transformations from the identity functor $\mathbf{1}_{\catC}: \catC\to \catC$ to itself.
\end{defn}
One example of a $k$-linear abelian category is the category $\mathsf{Vect}_k$ of finite-dimensional $k$-vector spaces.
\begin{rem}
    Let $X,Y,Z$ be objects in a $k$-linear category $\catC.$ Then there are inclusion maps ${k\hookrightarrow \mathrm{Hom}_{\catC}(X,X)}$ such that $1\mapsto \mathrm{id}_{X}.$ All Hom sets $\mathrm{Hom}_{\catC}(X,Y)$ are $k$-vector spaces, and composition maps $\mathrm{Hom}_{\catC}(X,Y)\times \mathrm{Hom}_{\catC}(Y,Z)\to \mathrm{Hom}_{\catC}(X, Z)$ are $k$-bilinear.
\end{rem}

\begin{defn}
    Let $V,W$ be $k$-vector spaces. A map $\varphi: V\to W$ is said to be \textit{polynomial} if for any $v_1, \ldots, v_n \in V$, the map $k^n  \to W$
    defined by
    \[ (\alpha_1, \ldots, \alpha_n) \mapsto \varphi(\alpha_1 v_1 + \cdots + \alpha_n v_n) \]
    is polynomial in the variables $\alpha_1, \ldots, \alpha_n$ with coefficients in $W$. 
\end{defn}


\begin{exmp}
    $k${-linear categories} typically arise as categories of vector spaces with additional structures. Examples of interest include the category of $G$-representations over $k$ and the category of mixed Hodge structures for $k=\QQ$.
\end{exmp}
\begin{defn}\cite[§1]{MacDonald1980}\label{defn-polyfun}
   Let $\catC_1, \catC_2$ be $k$-linear categories. A functor $F: \catC_1\to \catC_2$ is a \textit{polynomial functor} if for all $U,V\in \Ob(\catC_1),$ the map $$F: \mathrm{Hom}_{\catC_1}(U, V)\to \mathrm{Hom}_{\catC_2}(F(U), F(V))$$ is a polynomial mapping between the $k$-vector spaces. 
\end{defn}

\subsection{Homogeneous components}
Let $F : \cat{C}_1 \to \cat{C}_2$ be a polynomial functor. By polynomiality of $F$, for any $X \in \Ob(\cat{C}_1)$ and $\alpha \in k$, the endomorphism $F(\alpha)\in \mathrm{End}(FX)$ is polynomial in $\alpha$: there exists unique elements $\varphi_0(X), \ldots, \varphi_N(X) \in \mathrm{End}_{\cat{C}_2}(FX)$ such that
\begin{equation}\label{eq:poly-fun-defn}
F(\alpha) = \sum_{i = 0}^{N} \varphi_i(X) \alpha^i,
\end{equation}
for all $\alpha \in k$, where $N$ may depend on $F$ and $X$.

\begin{defn}
    A polynomial functor $F : \cat{C}_1 \to \cat{C}_2$ is homogeneous of degree $k$ if for all $X$, $\varphi_i(X) = 0$ for all $i\neq k.$
\end{defn}

\begin{lem}\cite[§2]{MacDonald1980}
    The assignment $F_iX: = \operatorname{im}\varphi_i(X)$ induces polynomial functors $F_i: \cat{C}_1 \to \cat{C}_2$ homogeneous of degree $i,$ such that $F = \bigoplus_{i=0}^\infty F_i$ as functors.
\end{lem}

\begin{defn}
    The functor $F_i$ is called the \textit{homogeneous degree $i$ component} of $F$. We say $F$ is \textit{bounded} if $F_i = 0$ for $i \gg 0$. Let $\mathbf{P}[\cat{C}_1, \cat{C}_2]$ be the category whose objects are bounded polynomial functors $\cat{C}_1 \to \cat{C}_2$, and whose morphisms are natural transformations. Then $\mathbf{P}[\cat{C}_1, \cat{C}_2]$ is a $k$-linear category. The full subcategory of polynomial functors which are homogeneous of degree $i$ is denoted as $\mathbf{P}_i[\cat{C}_1, \cat{C}_2].$
\end{defn}

\begin{rem}
    While $F = \bigoplus_{i = 0}^{\infty} F_i$ may well involve infinitely many terms for an unbounded polynomial functor $F,$ for each given $X$ only finitely many terms $F_i X$ will be non-zero, since there are only finitely many nonzero terms on the right-hand side of (\ref{eq:poly-fun-defn}).
\end{rem}

\subsection{Linearization}
The connection between polynomial functors and (wreath) symmetric sequences is established by the linearization construction for polynomial functors, which we now describe.
Let $\catC$ be a $k$-linear category, and let
\[\cat{C}^n := \underbrace{\cat{C} \times \cdots \times \cat{C}}_{n}\]
be the $n$-fold product of the category $\cat{C}$, which is again a $k$-linear category. 

Let $U: \cat{C}_1^n \to \cat{C}_2$ be a polynomial functor. For any object \[\tup{X} =  (X_1, \ldots, X_n) \in \cat{C}_1^n,\] the entrywise scaling action $k^{n} \to \End_{\cat{C}_2}(\tup{X})$ composes to a polynomial mapping
\[ k^n \to \End_{\cat{C}_1^n}(\tup{X}) \xrightarrow{U} \End_{\cat{C}_2}(U\tup{X}). \]
As before, decomposing this polynomial into homogeneous parts lifts to a decomposition of functors \[U  = \bigoplus_{m_1, \ldots, m_n \geq  0} U_{m_1, \ldots, m_n}. \]
If $F: \cat{C}_1 \to \cat{C}_2$ is a polynomial functor, then we obtain a polynomial functor $F^{(n)} : \cat{C}_1^n \to \cat{C}_2$ by $F(X_1, \ldots, X_n) = F(X_1 \oplus \cdots \oplus X_n)$.
\begin{defn}
Let $F: \cat{C}_1 \to \cat{C}_2$ be a polynomial functor homogeneous of degree $n$. Define the \textit{linearization} $L_F$ of $F$ as the functor $L_F: \cat{C}_1^n \to \cat{C}_2$ defined by
\[L_F(X_1, \ldots, X_n) = F^{(n)}_{1,\ldots, 1}(X_1, \ldots, X_n). \]
\end{defn}
By construction, $L_F$ is a \textit{multilinear} functor: for any $\tup{X} \in \cat{A}^n$, the map
\[ k^n \to \End_{\cat{A}^n}(\tup{X}) \xrightarrow{L_F} \End_{\cat{B}}(F\tup{X}) \]
is linear in each variable. Let $\Delta : \cat{C}_1 \to \cat{C}_1^{n}$ be the diagonal functor, and define $L_F^{(n)}: \mathsf{C}_1\to \mathsf{C}_2$ as the composition of functors $$L_F^{(n)}:=L_F\circ\Delta: \catC_1\to \catC_1^n\to \catC_2.$$

The construction $L_F^{(n)}(\-)$ recovers the homogeneous polynomial functor $F$ as follows.
\begin{thm}\cite[5.7]{MacDonald1980} \label{thm:maclin}
Let $F: \catC_1\to \catC_2$ be a polynomial functor homogeneous of degree $n.$ The functor $L_F^{(n)}$ admits an $S_n$-action, so that for all $X \in \cat{C}_1$, $L_F^{(n)}(X) = L_F(X,\dots,X)\in \Ob(\catC_2)$ admits the structure of an $S_n$-representation in $\catC_2$.

Assume further that $\catC_1$ is the category of finite projective left $A$-modules, where $A$ is a $k$-algebra. If we let $T^n(\-)$ denote the $n$-th tensor power over $k$, then the construction $$F(P)\mapsto  \left(L_F^{(n)}(A)\otimes_{T^n A} (T^n P)\right)^{S_n}$$ is an isomorphism of functors.
\end{thm}

\subsection{Application to wreath products}
The following correspondence results follow from the last theorem by taking $A = k$ or $A = kG$ (the group ring of $G$). As above, let $\cat{Vect}_k$ denote the category of finite-dimensional vector spaces, and write $G{\-}\cat{Vect}_k$ for the category of $G$-representations in $\cat{Vect}_k$. Again, when discussing $G$-vector spaces, we work over a splitting field $k$ of $G.$

\begin{cor}\cite[§6]{MacDonald1980}
    There is an equivalence of categories between $G\wr S_i$-representations over $k$ and $\mathbf{P}_i[G{\-}\mathsf{Vect}_{k}, \mathsf{Vect}_{k}]$. The equivalence sends a $G\wr S_i$-representation $U$ to the functor \[V\mapsto (U\otimes V^{\otimes i})^{G\wr S_i},\]where we note that $V^{\otimes i}$ carries a $G\wr S_i$-representation. Assembling the graded pieces, there is an equivalence of categories between bounded\footnote{{$G\wr \mathbb{S}$-modules should be bounded because each item on the right hand side has finite degree.}} $G\wr \mathbb{S}$-modules over $k$ and $\mathbf{P}[G{\-}\mathsf{Vect}_{k}, \mathsf{Vect}_{k}].$

    Taking $G$ as the trivial group, there is an equivalence of categories between bounded $\mathbb{S}$-modules over $k$ and $\mathbf{P}[\mathsf{Vect}_{k}, \mathsf{Vect}_{k}].$
\end{cor}

After identifying $G\wr\mathbb{S}$-modules with polynomial functors, we describe generalized plethysm as a composition of polynomial functors. To to this, it is helpful to understand $G\times \bbS$-modules as polynomial functors as well.

\begin{lem}
    There is an equivalence of categories between bounded $G\times \mathbb{S}$-modules over $k$ with $\mathbf{P}[\mathsf{Vect}_{k}, G{\-}\mathsf{Vect}_{k}].$
\end{lem}
\begin{proof}
    A bounded $G\times \mathbb{S}$-module $\mathcal{V}$ is a $G$-object in the category of bounded $\mathbb{S}$-modules. Hence the polynomial functor $F_{\mathcal{V}}$ is a $G$-object in $\mathbf{P}[\mathsf{Vect}_k,\mathsf{Vect}_k].$ This implies that $G$ acts by natural transformations on $F_{\mathcal{V}},$ namely there is a group homomorphism $G\to \mathsf{End}(F_{\mathcal{V}}).$ Plugging in $V\in \mathsf{Ob}(\mathsf{Vect}_k),$ this gives a group homomorphism $G\to \mathsf{End}(F_{\mathcal{V}}(V))$ that is compatible with morphisms $\varphi: V\to W$ in $\mathsf{Vect}_k.$ Therefore, the functor $F_{\mathcal{V}}$ defines a functor from $\mathsf{Vect}_k$ to $G\- \mathsf{Vect}_k.$ It is polynomial because the underlying $\bbS$-module is so. 
    
    In the opposite direction, given a polynomial functor $F\in \mathbf{P}[\mathsf{Vect}_k, G\-\mathsf{Vect}_k],$ the linearization procedure described above gives an $\bbS$-module with a commuting $G$-action, hence a $G\times \bbS$-module. The two operations are $G$-equivariant and inverse to each other on the level of bounded $\bbS$-modules and $\mathbf{P}[\mathsf{Vect}_k, \mathsf{Vect}_k],$ hence they give the equivalence of categories between bounded $G\times \bbS$-modules and $\mathbf{P}[\mathsf{Vect}_k, G\-\mathsf{Vect}_k].$
\end{proof}

\begin{cor}
    There are isomorphisms of rings \[\Lambda_G\cong K_0(\mathbf{P}[\mathsf{Vect}_{k}, G{\-}\mathsf{Vect}_{k}])\]
    and
    \[\Lambda(G)\cong K_0(\mathbf{P}[G\-\mathsf{Vect}_{k}, \mathsf{Vect}_{k}])\]
    Applying $K_0(\-)$ on the composition of polynomial functors $$\mathbf{P}[G{\-}\mathsf{Vect}_{k}, \mathsf{Vect}_{k}]\times \mathbf{P}[\mathsf{Vect}_{k}, G{\-}\mathsf{Vect}_{k}]\to \mathbf{P}[\mathsf{Vect}_{k}, \mathsf{Vect}_{k}]$$ recovers the plethysm $\circ_{G}: \Lambda(G)\times \Lambda_G\to \Lambda.$ 
\end{cor}

\begin{rem}
    The categories and their Grothendieck rings above are graded by $i\in \mathbb{N}$ via $G\wr S_i$-representations and homogeneous degree $i$ polynomials respectively. As the equivalence respects the grading, the respective Grothendieck rings remain isomorphic upon completions.
\end{rem}


\begin{rem}
    Let $\mathsf{C}$ be a $k$-linear category. The above equivalences of categories can be generalized to an equivalence between bounded $G\wr \mathbb{S}$-modules in $\mathsf{C}$ and bounded polynomial functors $\mathbf{P}[[G, \mathsf{C}], \mathsf{C}],$ where $[G, \mathsf{C}]$ denotes the category of $G$-representations in $\mathsf{C}$. Tensoring with $\C$, their Grothendieck rings are isomorphic to ${K_0(\mathsf{C})\otimes_{\Z} \Lambda(G)}$ by Getzler's Peter--Weyl theorem \cite[Theorem 3.2]{GetzlerPreprint}. Of interest to us will be the case $\mathsf{C} = \mathsf{MHS}_{\mathbb{Q}},$ the category of mixed Hodge structures over $\QQ$.
\end{rem}

\subsection{Proof of Theorem \ref{thm:s2polyfun}}\label{sec:pfpolyfun}
We prove the theorem via a few lemmas. In the following, fix some field $k$ of characteristic zero, and continue to let $\mathsf{Vect}_k$ denote the category of finite-dimensional $k$-vector spaces, and let $\bbS\-\cat{Vect}_k$ denote the category of $\bbS$-modules over $k$.
\begin{lem}
    For a bounded $\mathbb{S}^{[2]}$-module $\mathcal{W},$ the functor ${F_\mathcal{W}: \bbS\-\cat{Vect}_k \to \cat{Vect}_k}$ given by \[F_{\calW}:\mathcal{V}\mapsto \bigoplus_{\nu\in \Part}\mathcal{W}(\nu)\otimes_{\mathbb{S}_{\nu}}\mathcal{V}^{\otimes \nu}\] is polynomial.
\end{lem}
\begin{proof}
    Let $\mathcal{V}_1, \mathcal{V}_2$ be $\mathbb{S}$-modules, then $$\mathrm{Hom}_{\mathbb{S}}(\mathcal{V}_1, \mathcal{V}_2) =\bigoplus_{n\geq 0} \mathrm{Hom}_{S_i}(\mathcal{V}_1(i), \mathcal{V}_2(i)).$$ Let $(f_i)_{i\geq 0}\in \mathrm{Hom}_{\mathbb{S}}(\mathcal{V}_1, \mathcal{V}_2).$ Applying the functor $F_{\nu}: \mathcal{V}\mapsto \mathcal{V}^{\otimes \nu}$ to $(f_i)_{i\geq 0}$ gives $f^{\otimes \nu} = \bigotimes_{i\geq 0} f_i^{\otimes \nu_i}: \mathcal{V}_1^{\otimes \nu}\to \mathcal{V}_2^{\otimes \nu},$ which is polynomial in $(f_i)_{i\geq 0}.$ Therefore, $F_{\mathcal{W}}$ are also polynomial, as they arise from taking the coinvariants of finite sums of $F_{\nu}.$
\end{proof}

We decompose $F$ into homogeneous parts via scalar multiplications. Let $k^{\Z_{\geq0}} := \bigoplus_{i \geq 0} k$ with its $k$-algebra structure and let $\mathcal{V}$ be an $\mathbb{S}$-module. The entry-wise scaling action is $S_i$-equivariant for each $i$ and gives a $k$-algebra homomorphism $k^{\Z_{\geq 0}} \to \mathrm{End}_{\mathbb{S}}(\mathcal{V}).$
Let $F \in \mathbf{P}[\mathbb{S}\-\mathsf{Vect}, \mathsf{Vect}_{}]$. Post-composing with $F$ gives a map
\[\varphi_{F, \mathcal{V}}: k^{\Z_{\geq 0}} \to \mathrm{End}_{\mathbb{S}}(\mathcal{V}) \xrightarrow{F_{\mathcal{V}}} \mathrm{End}(F\mathcal{V}). \]Abusing notation, we denote the image of $\tup{\alpha} = (\alpha_0, \alpha_1, \alpha_2, \ldots)\in k^{\mathbb{Z}_{\geq 0}}$ under this map as $F(\tup{\alpha}).$

Because $F$ is a bounded polynomial functor, there exists $u_\pi\in \mathrm{End}(F\mathcal{V})$ which is zero for all but finitely many $\pi\in \Part$, such that $F(\tup{\alpha}) = \sum_{\pi\in \Part}u_{\pi}\prod_{i\geq 0}\alpha_i^{\pi_i}.$

\begin{lem}
    The construction $F_\pi(\mathcal{V}):=u_{\pi}(\mathcal{V})$ defines a direct sum $F = \bigoplus_{\pi\in \Part}F_\pi$ as functors.
\end{lem}
\begin{proof}
    We follow the discussion in \cite[§2]{MacDonald1980} verbatim. Because $F(\tup{\alpha}\tup{\beta}) = F(\tup{\alpha})F(\tup{\beta}),$ expanding we have $$\left(\sum_{\pi\in \Part} u_\pi \prod_{i\geq 0}\alpha_i^{\pi_i} \right)\left(\sum_{\pi\in \Part} u_\pi \prod_{i\geq 0}\beta_i^{\pi_i} \right) = \sum_{\pi\in \Part} u_\pi \prod_{i\geq 0}(\alpha_i\beta_i)^{\pi_i}.$$ Equating coefficients on both sides, there is $u_\pi^2  =u_\pi$ and $u_\pi u_\mu = 0.$ For $\mathcal{V}\in \Ob(\mathbb{S}\-\mathsf{Vect}_{}),$ define $F_{\pi}(\mathcal{V}):=u_\pi(F (\mathcal{V})),$ we have a direct sum decomposition $F (\mathcal{V})=\bigoplus_{\pi\in \Part} F_{\pi}(\mathcal{V}).$

    It remains to show that the above decomposition is functorial. Let $\mathcal{V}'\in \Ob(\mathbb{S}\-\mathsf{Vect}_{})$ and let $\boldsymbol{f} = (f_i)_{i\geq 0}\in \End_{\mathbb{S}}(\mathcal{V}, \mathcal{V}').$ Because the entrywise scaling action of $\tup{\alpha}$ commutes with $(f_i)_{i\geq 0},$ applying $F$ gives $F(\tup{\alpha})F(\boldsymbol{f}) = F(\boldsymbol{f}) F(\tup{\alpha}).$ Expanding into $$\sum_{\pi\in \Part}\left(u_{\pi}\prod_{i\geq 0}\alpha_i^{\pi_i} \right)F(\boldsymbol{f}) = \sum_{\pi\in \Part} F(\boldsymbol{f})\left(u_{\pi}\prod_{i\geq 0}\alpha_i^{\pi_i} \right)$$ and equating coefficients again, we see that $u_\pi F(\boldsymbol{f}) = F(\boldsymbol{f}) u_\pi,$ so that $F\mapsto u_\pi$ is a natural transformation of functors. Therefore, $F = \bigoplus_{\pi\in \Part}F_\pi$ as functors as desired.
\end{proof}

\begin{defn}
    Let $F\in \mathbf{P}[\mathbb{S}\-\mathsf{Vect}_{}, \mathsf{Vect}_{}]$ and let $\pi\in \Part.$ The degree $\pi$ homogeneous part of $F$ is defined as $u_\pi\in \mathbf{P}[\mathbb{S}\-\mathsf{Vect}_{}, \mathsf{Vect}_{}].$ $F$ is said to be homogeneous of degree $\pi$ if $u_{\nu}=0$ for all $\nu\neq \pi.$ Denote the full subcategory of homogeneous degree $\pi$ polynomial functors as $\mathbf{P}_\pi[\mathbb{S}\-\mathsf{Vect}_{}, \mathsf{Vect}_{}].$
\end{defn}

Let $F\in \mathbf{P}_\pi[\mathbb{S}\-\mathsf{Vect}_{}, \mathsf{Vect}_{}],$ we adapt Macdonald's construction to linearize the functor $F_\pi$ as a $\mathbb{S}_\pi$-representation.

\begin{defn}
    Let $\left(\mathbb{S}\-\mathsf{Vect}_{}\right)^{\pi}$ denote the auxiliary category $\prod_{i\geq 0} \left(S_{i}\-\mathsf{Vect}_{}\right)^{\pi_i},$ which is $k$-linear, and let $P^\pi: \left(\mathbb{S}\-\mathsf{Vect}_{}\right)^{\pi}\to  \mathbb{S}\-\mathsf{Vect}_{}$ be the functor given by $$\boldsymbol{\mathcal{V}} = \left(\mathcal{V}(i)^{(1)},\dots, \mathcal{V}(i)^{(\pi_i)}\right)_{i\geq 0}\mapsto \left(\bigoplus_{j=1}^{\pi_i}\mathcal{V}(i)^{(j)}\right)_{i\geq 0}.$$

    A homogeneous polynomial functor $F\in \mathbf{P}_\pi[\mathbb{S}\-\mathsf{Vect}_{}, \mathsf{Vect}_{}]$ induces a polynomial functor $F^{\pi}: \left(\mathbb{S}\-\mathsf{Vect}_{}\right)^{\pi}\to \mathsf{Vect}_{}$ given by the composition $F^{\pi} := F\circ P^\pi.$
\end{defn}

\begin{defn}
Let $k^{\pi}$ be the $k$-algebra $\prod_{i\geq 0}k^{\pi_i}.$ We use indices $(\tup{\alpha}) = \left(\alpha_{(i,j)}\right)_{i\geq 0, 1\leq j\leq \pi_i}$ to denote elements in $k^{\pi}.$ Similarly, define the monoid $\mathbb{N}^{\pi} = \prod_{i\geq 0}{\mathbb{N}}^{\pi_i}$ and use the indices ${\boldsymbol{m}} = \left(\boldsymbol{m}_{(i,j)}\right)_{i\geq 0, 1\leq j\leq \pi_i}\in \mathbb{N}^{\pi}.$ 
\end{defn} 

As earlier, entry-wise scaling action of $(\tup{\alpha})\in k^{\pi}$ on $\left(\mathbb{S}\-\mathsf{Vect}_{}\right)^{\pi}$ induces a decomposition of functors $F^{\pi} = \bigoplus_{\boldsymbol{m}\in \mathbb{N}^{\pi}} u_{\boldsymbol{m}}$ such that $$F^{\pi}(\tup{\alpha}) = \sum_{\boldsymbol{m}\in \mathbb{N}^{\pi}}u_{\boldsymbol{m}}\prod_{i\geq 0}\prod_{j=1}^{\pi_i} \alpha_{(i, j)}^{\boldsymbol{m}_{(i,j)}}.$$

\begin{defn}
    Let $F\in \mathbf{P}_\pi[\mathbb{S}\-\mathsf{Vect}_{}, \mathsf{Vect}_{}].$ Define $L_F: \left(\mathbb{S}\-\mathsf{Vect}_{}\right)^{\pi}\to \mathsf{Vect}_{}$ as the homogeneous term $u_{\boldsymbol{1}}$ corresponding to the multi-index having all entries equal to one.

    Let $\Delta: \mathbb{S}\-\mathsf{Vect}_{}\to \left(\mathbb{S}\-\mathsf{Vect}_{}\right)^{\pi}$ be the functor given by \[\mathcal{V} = (\mathcal{V}(i))_{i\geq 0}\mapsto (\underbrace{\mathcal{V}(i),\dots, \mathcal{V}(i)}_{\pi_i}).\] Define $L_F^{(\pi)}: \mathbb{S}\-\mathsf{Vect}_{}\to \mathsf{Vect}_{}$ as the composition $L_F^{(\pi)}:= L_F \circ \Delta.$
\end{defn}

\begin{lem}
    The functor $L_F^{(\pi)}$ admits a $\prod_{i\geq 0}S_{\pi_i}$-action.
\end{lem}
\begin{proof}
     Let $\mathcal{V}$ be an $\bbS$-module. We observe that $\prod_{i\geq 0}S_{\pi_i}$ acts on the $\bbS$-module $P^\pi\circ \Delta (\mathcal{V})$, which is given by $i\mapsto \mathcal{V}(i)^{\oplus \pi_i}$, by permuting entries. The action is compatible with morphisms of $\mathbb{S}$-modules, so $\prod_{i\geq 0}S_{\pi_i}$ acts on the functor $P^\pi\circ \Delta.$
    
    Let $k^{\pi}\ni (\tup{\alpha}): P^\pi \circ \Delta \to P^\pi\circ \Delta $ be the entry-wise multiplication map, and let $\boldsymbol{\gamma} =\left(\gamma_i\right)_{i\geq 0}\in \prod_{i\geq 0}S_{\pi_i}$. We note that $\prod_{i\geq 0}S_{\pi_i}$ acts on $k^\pi$ by permuting the coordinates and denote $(\tup{\alpha})\mapsto \boldsymbol{\gamma}\cdot (\tup{\alpha})$ by the action. From an explicit calculation, $\boldsymbol{\gamma}\circ (\tup{\alpha}) = (\boldsymbol{\gamma}\cdot (\tup{\alpha}))\circ \boldsymbol{\gamma}\in \End\left(P^\pi\circ \Delta (\mathcal{V}) \right).$ Applying the functor $F$ and expanding, there is $$F(\boldsymbol{\gamma})\sum_{\boldsymbol{m}\in \mathbb{N}^{\pi}}u_{\boldsymbol{m}}\prod_{i\geq 0}\prod_{j=1}^{\pi_i} \alpha_{(i, j)}^{\boldsymbol{m}_{(i,j)}} = \sum_{\boldsymbol{m}\in \mathbb{N}^{\pi}}u_{\boldsymbol{m}}\prod_{i\geq 0}\prod_{j=1}^{\pi_i} \alpha_{(i, j)}^{\boldsymbol{m}_{(i,\gamma_i(j))}} F(\boldsymbol{\gamma}).$$
        
     This calculation implies that the multilinear part of $F$ is preserved under the $\prod_{i\geq 0}S_{\pi_i}$-action, as the monomial $\prod_{i}\prod_{j=1}^{\pi_i}\alpha_{(i,j)}$ is preserved by the permutation. Therefore, the functor $L_F^{(\pi)}$ is $\prod_{i\geq 0}S_{\pi_i}$-invariant.
\end{proof}

\begin{rem}\label{rem:wract}
    In fact, an adaptation of the above calculation shows that for an $\bbS$-module $\mathcal{V}$, the vector space $L_F^{(\pi)}(\mathcal{V})$ admits a $\mathbb{S}_{\pi} = \prod_{i\geq 0}S_i\wr S_{\pi_i}$-action, in line with the fact that each item $P^\pi \circ\Delta (\mathcal{V})(i)$ is an $S_i\wr S_{\pi_i}$-representation.
\end{rem}

We relate the linearization $L_F^{(\pi)}$ back to $F$ via the following constructions. To begin with, recall that they are both functors from $\mathbb{S}$-modules to vector spaces and that $L_F^{(\pi)}$ is the multilinear summand\footnote{A priori, we only know that $L_F$ is the multilinear summand of $F^\pi$, and the statement follows from the fact that $\Delta$ is linear and preserves multidegrees.} of $F^\pi\circ \Delta.$

\begin{defn}
    Let $j: L_F^{(\pi)}\to F^\pi\circ \Delta$ and $q: F^\pi\circ \Delta\to L_F^{(\pi)}$ be the inclusion and projection to the multi-linear summand, respectively. The composition $v:=jq: F^\pi\circ \Delta\to  F^\pi\circ \Delta$ retains the summand of $L_F^{(\pi)}$ and sets all the rest to zero.

    Define $p: P^\pi\circ \Delta\to \mathrm{id}_{\mathbb{S}\-\mathsf{Vect}_{}}$ as the natural transformation given by entry-wise addition: for an $\mathbb{S}$-module $\mathcal{V}$, this is given in coordinates by \[P^\pi\circ \Delta(\mathcal{V})\ni(w_{(i,j)})_{i\geq 0, 1\leq j\leq \pi_i}\mapsto \left(\sum_{j=1}^{\pi_i} w_{(i,j)}\right)_{i\geq 0}.\] In the other direction, let $i: \mathrm{id}_{\mathbb{S}\-\mathsf{Vect}_{}}\to P^\pi\circ \Delta$ be the diagonal map.
    
    Define $\tilde{\xi}:=q\circ F(i): F\to L_F^{(\pi)}$ and $\tilde{\eta}:=F(p)\circ j: L_F^{(\pi)}\to F$.
\end{defn}

\begin{lem}
    The map $v: F^\pi\circ \Delta\to  F^\pi\circ \Delta$ is $\prod_{i\geq 0}S_{\pi_i}$-equivariant.
\end{lem}
\begin{proof}
    This follows from the fact that $L_F^{(\pi)}$ is an $\prod_{i\geq 0}S_{\pi_i}$-invariant summand of $F^\pi\circ \Delta.$
\end{proof}

We recall that $\boldsymbol{\sigma}\in \prod_{i\geq 0}S_{\pi_i}$ acts on the functor $P^{\pi}\circ \Delta$. Applying $F$ gives an action $F(\boldsymbol{\sigma})$ on the functor $F^{\pi}\circ \Delta.$

\begin{lem}
    There is an equality of natural transformations from $F^\pi\circ \Delta$ to itself $$v\circ F(i\circ p)\circ v = \sum_{\boldsymbol{\sigma}\in \prod_{i\geq 0}S_{\pi_i}} F(\boldsymbol{\sigma})\circ v.$$ 
\end{lem}
\begin{proof}
    Because $v$ is projection to the multilinear summand, it is idempotent, so $v^2 = v.$ Thus $F(\boldsymbol{\sigma})\circ v = F(\boldsymbol{\sigma})\circ v \circ v.$ Applying $\prod_{i\geq 0}S_{\pi_i}$-equivariance from the previous lemma, $F(\boldsymbol{\sigma})\circ v \circ v = v\circ F(\boldsymbol{\sigma})\circ v,$ which we calculate. 
    
    From how $v$ is defined, $v\circ F(\boldsymbol{\sigma})\circ v$ is equal to the coefficient of $(\tup{\alpha})^{\boldsymbol{1}}(\tup{\beta})^{\boldsymbol{1}} = \prod_{i\geq 0}\left(\prod_{j=1}^{\pi_i}\alpha_{(i,j)}\beta_{(i,j)}\right)$ in $F(\tup{\alpha})\circ F(\boldsymbol{\sigma})\circ F(\tup{\beta}) = F(\boldsymbol{\sigma}\cdot \tup{\alpha})F(\tup{\beta})F(\boldsymbol{\sigma}).$

    In parallel, $v\circ F(i\circ p)\circ v$ is equal to the coefficient of $(\tup{\alpha})^{\boldsymbol{1}}(\tup{\beta})^{\boldsymbol{1}} $ in $F(\tup{\alpha})\circ F(i\circ p)\circ F(\tup{\beta}),$ where $i\circ p$ is given by the formula $$(w_{(i,j)})_{i\geq 0, 1\leq j\leq \pi_i}\mapsto \left(\nu_{(i,j)}:=\sum_{\ell=1}^{\pi_i}w_{(i, \ell)}\right)_{i\geq 0, 1\leq j\leq \pi_i}.$$

    It is helpful to spell out the maps $\boldsymbol{\sigma}: P^{\pi}\circ \Delta\to P^{\pi}\circ \Delta$ and $i\circ p$ more explicitly. For a sequence $\boldsymbol{m}=(m_i)_{i\geq 0}$ of integers $1\leq m_i\leq \pi_i,$ there are natural transformations $\boldsymbol{i}_{\boldsymbol{m}}: \mathrm{id}_{\mathbb{S}\-\mathsf{Vect}_{}}\to P^\pi\circ \Delta$, that, for an $\mathbb{S}$-module $\mathcal{V}$, includes each $\mathcal{V}(i)$ into the $m_i$-th summand of $\mathcal{V}(i)^{\oplus \pi_i} = (P^\pi\circ \Delta(\mathcal{V}))(i)$ and assigns zero to all the other summands. Similarly, there are projections $\boldsymbol{p}_{\boldsymbol{m}}: P^\pi\circ \Delta\to \mathrm{id}_{\mathbb{S}\-\mathsf{Vect}_{}}$ along the summands specified by $\boldsymbol{m}.$ We may also project the scalars 
    
    With the notation, $\boldsymbol{\sigma} = \sum_{\boldsymbol{m}}\boldsymbol{i}_{(\boldsymbol{\sigma}\cdot \boldsymbol{m})}\boldsymbol{p}_{\boldsymbol{m}}$ and $i\circ p = \sum_{\boldsymbol{m}^{(1)}, \boldsymbol{m}^{(2)}}\boldsymbol{i}_{\boldsymbol{m}^{(1)}}\boldsymbol{p}_{\boldsymbol{m}^{(2)}}.$ 

    Applying $F$ to the two formulas, $v\circ F(\boldsymbol{\sigma})\circ v$ and $v\circ F(i\circ p)\circ v$ are the $(\tup{\alpha})^{\boldsymbol{1}}(\tup{\beta})^{\boldsymbol{1}}$-coefficients of $$F\left(\sum_{\boldsymbol{m}}\left(\prod_{i\geq 0}\alpha_{(i, \boldsymbol{\sigma}(m_i))}\beta_{(i, m_i)}\right)\boldsymbol{i}_{(\boldsymbol{\sigma}\cdot \boldsymbol{m})}\boldsymbol{p}_{\boldsymbol{m}}\right)\text{ and }F\left(\sum_{\boldsymbol{m}^{(1)}, \boldsymbol{m}^{(2)}}\left(\prod_{i\geq 0}\alpha_{(i, m_i^{(1)})}\beta_{(i, m_i^{(2)})}\right)\boldsymbol{i}_{\boldsymbol{m}^{(1)}}\boldsymbol{p}_{\boldsymbol{m}^{(2)}}\right).$$ Matching $\{\boldsymbol{m}^{(1)}\}$ with $\{(\boldsymbol{\sigma}\circ \boldsymbol{m}^{(2)})\},$ we see that $v\circ F(i\circ p)\circ v = \sum_{\boldsymbol{\sigma}\in \prod_{i\geq 0}S_{\pi_i}} v\circ F(\boldsymbol{\sigma})\circ v$ as desired.
\end{proof}

\begin{defn}
    For $\boldsymbol{s}\in \prod_{i\geq 0}S_{\pi_i}$, restrict let $\tilde{F}(\boldsymbol{s}): L_F^{(\pi)}\to L_F^{(\pi)}$ be its action. As this is the restriction of the $\boldsymbol{s}\in \prod_{i\geq 0}S_{\pi_i}$-action on $F^{\pi}\circ\Delta$, we have $\tilde{F}(\boldsymbol{s})  = q F(\boldsymbol{s}) j.$
\end{defn}

\begin{lem}
     The compositions of the natural transformations satisfy $$\tilde{\eta}\tilde{\xi} = \prod_{i\geq 0}\left(\pi_i!\right), \text{ and }\tilde{\xi}\tilde{\eta} = \sum_{\boldsymbol{s}\in \prod_{i\geq 0}S_{\pi_i}} \tilde{F}(\boldsymbol{s}).$$
\end{lem}
\begin{proof}
    The composition $\tilde{\eta}\tilde{\xi}$ is equal to $F(p)\circ jq\circ F(i)$. Recall that $jq = v$, which is the coefficient of $(\tup{\alpha})^{\mathbf{1}}=$ in the map $F^{\pi}\circ\Delta\circ  (\tup{\alpha})$. Thus, $F(p)jqF(i)$ is the coefficient of $(\tup{\alpha})^{\mathbf{1}}$ in $F(p)F(\tup{\alpha})F(i) = F(p \circ (\tup{\alpha}) \circ i).$ 
    
    Applying the formulas of the maps $i$ and $p$, the composition $p\circ (\tup{\alpha})\circ i: V\to V$ is given by $$(v_i)_{i\geq 0}\mapsto \left(\left(\sum_{j=1}^{\pi_i}\alpha_{(i,j)}\right)v_j\right)_{i\geq 0}.$$ Because $F$ is homogeneous of degree $\pi$, the map $F(p \circ (\tup{\alpha}) \circ i)$ is given by the scalar multiplication by $\prod_{i\geq 0} \left(\sum_{j=1}^{\pi_i}\alpha_{(i,j)}\right)^{\pi_i}$, of which the $(\tup{\alpha})^{\mathbf{1}}$-coefficient is given by $\prod_{j\geq 0}\left(\pi_j!\right).$

    In the other direction, we recall that $qj = \mathrm{id}_{L_F^{(\pi)}}$ and $v = jq.$ Applying the previous lemmas,\begin{align*}
        \tilde{\xi}\tilde{\eta} & = (qj)\circ \tilde{\xi}\tilde{\eta}\circ (qj) = q\circ (j \tilde{\xi}\tilde{\eta} q)\circ j \\
        & = q\left(\sum_{\boldsymbol{s}\in \prod_{i\geq 0}S_{\pi_i}}F(\boldsymbol{s})v\right)j = q\left(\sum_{\boldsymbol{s}\in \prod_{i\geq 0}S_{\pi_i}}F(\boldsymbol{s})j\right)\circ (qj) = \sum_{\boldsymbol{s}\in \prod_{i\geq 0}S_{\pi_i}}\tilde{F}(\boldsymbol{s}).
    \end{align*}    
\end{proof}

The formula $\sum_{\boldsymbol{s}\in \prod_{i\geq 0}S_{\pi_i}}\tilde{F}(\boldsymbol{s})$ is invariant under $\prod_{i\geq 0}S_{\pi_i}$-action, which allows us to relate $F$ to the $\prod_{i\geq 0}S_{\pi_i}$-invariant summand of $L_F^{(\pi)}.$

\begin{defn}
    The invariant $\left(L_F^{(\pi)}\right)^{\prod_{i\geq 0}S_{\pi_i}}$ is a direct summand of $L_F^{(\pi)}.$ Let $$\epsilon: \left(L_F^{(\pi)}\right)^{\prod_{i\geq 0}S_{\pi_i}}\to L_F^{(\pi)} \text{ and }\pi: L_F^{(\pi)}\to \left(L_F^{(\pi)}\right)^{\prod_{i\geq 0}S_{\pi_i}}$$ be the inclusion and projection respectively. Define $$\xi:= \pi \tilde{\xi}: F\to \left(L_F^{(\pi)}\right)^{\prod_{i\geq 0}S_{\pi_i}}\text{ and }\eta:= \tilde{\eta}\epsilon: \left(L_F^{(\pi)}\right)^{\prod_{i\geq 0}S_{\pi_i}}\to F.$$
\end{defn}

The following formula follows from general facts about finite groups acting on $k$-linear categories or functors for $k$ a field of characteristic zero.

\begin{lem}
    The composition $\epsilon\pi: L_F^{(\pi)}\to L_F^{(\pi)}$, which is projection to the invariant summand, is equal to $\frac{1}{\prod_{i\geq 0}(\pi_i!)}\sum_{\boldsymbol{s}\in \prod_{i\geq 0}S_{\pi_i}}\tilde{F}(\boldsymbol{s}).$
\end{lem}

\begin{lem}
    The natural transformations $\xi: F\to \left(L_F^{(\pi)}\right)^{\prod_{i\geq 0}S_{\pi_i}}$ and $\eta: \left(L_F^{(\pi)}\right)^{\prod_{i\geq 0}S_{\pi_i}}\to F$ are isomorphisms such that $\xi\eta = \prod_{i\geq 0}(\pi_i!)$ and $\eta\xi = \prod_{i\geq 0}(\pi_i!).$
\end{lem}
\begin{proof}
    Using the previous lemmas, we compute 
    $$\xi\eta  = \pi \tilde{\xi} \tilde{\eta}\epsilon  = \pi \left(\sum_{\boldsymbol{s}\in \prod_{i\geq 0}S_{\pi_i}} \tilde{F}(\boldsymbol{s})\right)\epsilon   = \pi\left(\prod_{i\geq 0}(\pi_i!)\epsilon\pi\right) \epsilon = \prod_{i\geq 0}(\pi_i!),$$
    $$\eta \xi =  \tilde{\eta}\epsilon \pi \tilde{\xi}  = \tilde{\eta} \left(\frac{1}{\prod_{i\geq 0}(\pi_i!)} \sum_{\boldsymbol{s}\in \prod_{i\geq 0}S_{\pi_i}} \tilde{F}(\boldsymbol{s})\right) \tilde{\xi} = \frac{1}{\prod_{i\geq 0}(\pi_i!)} \tilde{\eta}\tilde{\xi}\tilde{\eta}\tilde{\xi} = \prod_{i\geq 0}(\pi_i!).$$
\end{proof}

To finish, we construct the $\mathbb{S}_\pi$-representations corresponding to the homogeneous degree $\pi$ polynomial functors.


Recall that $L_F: (\mathbb{S}\-\mathsf{Vect})^{\pi}\to \mathsf{Vect}$ is defined as the multilinear component of $F^\pi = F\circ P^\pi.$ Applying the multilinear functor to the standard formulas $k[S_i]\otimes_{S_i} \mathcal{V}(i)^{(j)} = \mathcal{V}(i)^{(j)},$ we have that for all $\boldsymbol{\mathcal{V}} = \left(\mathcal{V}(i)^{(1)},\dots, \mathcal{V}(i)^{(\pi_i)}\right)_{i\geq 0}\in \Ob((\mathbb{S}\-\mathsf{Vect})^{\pi}),$ $$L_F(\boldsymbol{\mathcal{V}}) = L^{(\pi)}_F(k[\mathbb{S}])\otimes_{\prod_{i\geq 0}S_i^{\pi_i}} \left(\bigotimes_{i\geq 0}\mathcal{V}(i)^{(1)}\otimes\cdots \otimes \mathcal{V}(i)^{(\pi_i)}\right).$$

\begin{defn}
    Let $k[\mathbb{S}]$ denote the $\mathbb{S}$-module given by the group rings $i\mapsto k[S_i].$ From remark \ref{rem:wract}, $L_F^{(\pi)}(k[\mathbb{S}])$ is an $\mathbb{S}_\pi$-representation.
\end{defn}

The proof of theorem \ref{thm:s2polyfun} is finished by the following lemma.

\begin{lem}
    The functor $L^{(\pi)}:\mathbf{P}_\pi[\mathbb{S}\-\mathsf{Vect}, \mathsf{Vect}]\to \mathbb{S}_\pi\-\mathsf{Vect}$ given by $F\mapsto L_F^{(\pi)}(k[\mathbb{S}])$ is an inverse to $P: \mathbb{S}_\pi\-\mathsf{Vect}\to \mathbf{P}_\pi[\mathbb{S}\-\mathsf{Vect}, \mathsf{Vect}]$ given by $W\mapsto P_W:= \left( \mathcal{V}\mapsto W\otimes_{\mathbb{S}_\pi}\mathcal{V}^{\otimes \pi}\right).$
\end{lem}
\begin{proof}
    Let $F\in \Ob(\mathbf{P}_\pi[\mathbb{S}\-\mathsf{Vect}, \mathsf{Vect}]),$ we check that $P \circ L^{(\pi)}(F)$ is the functor given by $$\mathcal{V}\mapsto L_F^{(\pi)}(k[\mathbb{S}])\otimes_{\mathbb{S}_\pi}\mathcal{V}^{\otimes \pi} = \left(L_F^{(\pi)}(k[\mathbb{S}])\otimes_{\prod_{i\geq 0}S_{i}^{\pi_i}}\mathcal{V}^{\otimes \pi}\right)^{\prod_{i\geq 0}S_{\pi_i}} = \left(L_F^{(\pi)}(\mathcal{V})\right)^{\prod_{i\geq 0}S_{\pi_i}} = F(\mathcal{V}),$$
    where the first equality uses the short exact sequence $1\to \prod_{i\geq 0}S_{i}^{\pi_i}\to \mathbb{S}_\pi\to \prod_{i\geq 0}S_{\pi_i}\to 1.$

    In the other direction, let $W$ be a $\mathbb{S}_\pi$-representation, so that $P_W(\mathcal{V}) = W\otimes_{\mathbb{S}_{\pi}} \bigotimes_{i\geq 0} \mathcal{V}(i)^{\otimes \pi_i}.$ The composition $P_W\circ P^{\pi}(\boldsymbol{\mathcal{V}})$ is given by $W\otimes_{\mathbb{S}_{\pi}} \bigotimes_{i\geq 0} \left(\bigoplus_{j=1}^{\pi_i} \mathcal{V}(i)^{(j)}\right)^{\otimes \pi_i}.$ The scalar $(\tup{\alpha})\in k^{\pi}$ acts on each $ \mathcal{V}(i)^{(j)}$ via scalar multiplication by $\alpha_{(i,j)},$ which induces its action on $P_W\circ P^{\pi}(\boldsymbol{\mathcal{V}}).$ Therefore, we may explicitly calculate the multilinear part $L_{P_W}$ as $\boldsymbol{\mathcal{V}}\mapsto W\otimes_{\mathbb{S}_{\pi}} \bigotimes_{i\geq 0}\left(\bigoplus_{\sigma\in S_{\pi_i}}\bigotimes_{j=1}^{\pi_i}\mathcal{V}(i)^{(\sigma(j))}\right).$ Now we calculate $$L_{P_W}\circ \Delta (k[\mathbb{S}]) =  W\otimes_{\mathbb{S}_{\pi}} \bigotimes_{i\geq 0}\left(\bigoplus_{\sigma\in S_{\pi_i}}\bigotimes_{j=1}^{\pi_i}k[S_i]\right)$$ We note that there is an $S_i\wr S_{\pi_i}$-isomorphism $\bigoplus_{\sigma\in S_{\pi_i}}\bigotimes_{j=1}^{\pi_i}\mathcal{V}(i)^{(\sigma(j))}\cong k[S_i\wr S_{\pi_i}].$ Hence the above is isomorphic to $W\otimes_{\mathbb{S}_{\pi}} k[\mathbb{S}_\pi] = W$ as desired. 
    \end{proof}
\begin{rem}
    Let $\mathsf{A}$ be a $k$-linear abelian category, which is symmetric monoidal, admits finite colimits, and contains a unit object. All of the constructions in this section so far can be adapted to polynomial functors from the category of $\mathbb{S}$-modules in $\mathsf{A}$ to $\mathsf{A}.$
\end{rem}
\subsection{Proof of Proposition \ref{prop:SSpoly}}\label{sec:pfSSpoly}

    As before, a polynomial functor $F\in \mathbf{P}[\mathsf{Vect}_{}, \bbS\-\mathsf{Vect}_{}]$ admits a unique direct sum decomposition $F = \sum_{i\geq 0} F_i$ where each $F_i\in \mathbf{P}_i[\mathsf{Vect}_{}, \bbS\-\mathsf{Vect}_{}]$ is homogeneous of degree $i.$ The definitions are similar to the previous section and to \cite{MacDonald1980}, so we omit them for brevity.

    Because $\bbS\-\mathsf{Vect}_{}$ is a $k$-linear abelian category, the general discussion in \cite{MacDonald1980} applies. In particular, Theorem 5.7 in loc. lit., which we recalled as Theorem \ref{thm:maclin}, applies to the setting. The theorem implies that for $F\in \mathbf{P}_n[\mathsf{Vect}_{}, \bbS\-\mathsf{Vect}_{}],$ the assignment $$P\mapsto \left(L_F^{(n)}(k)\otimes_{k} P^{\otimes n}\right)^{S_n}$$ gives an isomorphism of functors. By construction, the linearization $L_F^{(n)}(k)$ is an $\mathbb{S}$-module with commuting $S_n$-action, hence $L_F^{(n)}(k)$ is an $S_n\times \mathbb{S}$-module. By taking the direct sum over homogeneous components of $F$, we obtain a $(\bbS \times \bbS)$-module, which fits into the statement of the proposition.

    Finally, the condition that $F$ has finite degree as a polynomial functor is equivalent to asking that  $L_F^{(n)}(k) = 0$ for all but finitely many $n$, which is equivalent to the definition of (1)-boundedness as in Definition \ref{defn:boundedSS}.
\subsection{Proof of Lemma \ref{lem:tcirc}}\label{sec:pftcirc}
Unwinding the definitions of polynomial functors, the composition is the polynomial functor given by the formula $$V\mapsto \bigoplus_{\nu\in \Part} \mathcal{W}(\nu)\otimes_{\bbS_{\nu}} \left[\bigotimes_{\hspace{1pt}i\geq 0}\left(\bigoplus_{m\geq 0}\mathcal{V}(i,m)\otimes_{S_m} V^{\otimes m}\right)^{\otimes \nu_i}\right].$$

We explicitly linearize this formula. Firstly, the homongeneous degree $n$ part of the formula is given by $$V\mapsto \bigoplus_{\nu\in \Part} \mathcal{W}(\nu)\otimes_{\bbS_{\nu}}\left(\bigoplus_{\substack{(n_{(i,j)}\mid i\geq 0, 1\leq j\leq \nu_i):\\ \sum_{i\geq 0}\sum_{j = 1}^{\nu_i}n_{(i,j)}=n}}\left(\bigotimes_{i\geq 0}\bigotimes_{j = 1}^{\nu_i} \mathcal{V}(i, n_{(i,j)})\otimes_{S_{n_{(i,j)}}}V^{\otimes n_{(i,j)}}\right)\right).$$ The multi-linear part of the formula is given by plugging in $V = V_1\oplus\cdots\oplus V_n$ and extracting the multilinear terms, which sends $V$ to $$\bigoplus_{\nu\in \Part} \mathcal{W}(\nu)\otimes_{\bbS_{\nu}}\left[\bigoplus_{\substack{(n_{(i,j)}\mid i\geq 0, 1\leq j\leq \nu_i):\\ \sum_{i\geq 0}\sum_{j = 1}^{\nu_i}n_{(i,j)}=n}}\left(\bigoplus_{\substack{f:\{1,\dots,n\}\to \\ \{(i,j)\mid i\geq 0, 1\leq j\leq \nu_j\}\\ |f^{-1}(i,j)| = n_{(i,j)}}}\bigotimes_{i\geq 0}\bigotimes_{j = 1}^{\nu_i} \mathcal{V}(i, n_{(i,j)})\otimes\left(\bigotimes_{k\in f^{-1}(i,j)}V_k\right)\right)^{\prod_{i\geq 0}\prod_{j=1}^{\nu_i}S_{n_{(i,j)}}}\right].$$

Plugging in $V_1 = \cdots = V_n = k,$ the $S_n$-module corresponding to the degree $m$ homogeneous part of the polynomial functor is given by 
$$\bigoplus_{\nu\in \Part} \mathcal{W}(\nu)\otimes_{\bbS_{\nu}}\left[\bigoplus_{\substack{(n_{(i,j)}\mid i\geq 0, 1\leq j\leq \nu_i):\\ \sum_{i\geq 0}\sum_{j = 1}^{\nu_i}n_{(i,j)}=n}}\left(\bigoplus_{\substack{f:\{1,\dots,n\}\to \\ \{(i,j)\mid i\geq 0, 1\leq j\leq \nu_j\}\\ |f^{-1}(i,j)| = n_{(i,j)}}}\bigotimes_{i\geq 0}\bigotimes_{j = 1}^{\nu_i} \mathcal{V}(i, n_{(i,j)})\right)^{\prod_{i\geq 0}\prod_{j=1}^{\nu_i}S_{n_{(i,j)}}}\right]$$
$$ = \bigoplus_{\nu\in \Part} \mathcal{W}(\nu)\otimes_{\bbS_{\nu}}\left[\bigoplus_{\substack{(n_{(i,j)}\mid i\geq 0, 1\leq j\leq \nu_i):\\ \sum_{i\geq 0}\sum_{j = 1}^{\nu_i}n_{(i,j)}=n}}\Ind_{\prod_{i\geq 0}\prod_{j=1}^{\nu_i}S_{n_{(i,j)}}}^{S_n} \bigotimes_{i\geq 0}\bigotimes_{j = 1}^{\nu_i} \mathcal{V}(i, n_{(i,j)})\right],$$ which agrees with the degree $n$ homogeneous part of the $\mathbb{S}$-module $$\bigoplus_{\nu\in \Part} \mathcal{W}(\nu)\otimes_{\bbS_{\nu}} \left(\midboxtimes_{\hspace{1pt}i\geq 0}\mathcal{V}_i^{\boxtimes \nu_i}\right)$$ as claimed in the Lemma.

\bibliographystyle{amsalpha}
\bibliography{reference}

\end{document}